\documentclass[10pt,reqno]{article}

\usepackage{etex}
\usepackage{enumitem}
\usepackage{setspace}
\usepackage{amsmath}

\usepackage{tabu}
\usepackage{tikz}
\usepackage{tikz-cd}
\usepackage[margin=.8in]{geometry}
\usepackage[matrix,arrow,curve,frame]{xy}
\usepackage{hyperref}
\usepackage{amsthm}
\usepackage[capitalise,noabbrev]{cleveref}
\usepackage{textcomp}
\usepackage{amsfonts}
\usepackage{amssymb}
\usepackage{wasysym}
\usepackage{mathrsfs}
\usepackage{mathtools}
\usetikzlibrary{decorations.pathreplacing}
\usetikzlibrary{decorations.markings}
\usetikzlibrary{calc}

\definecolor{my-linkcolor}{rgb}{0.75,0,0}
\definecolor{my-citecolor}{rgb}{0.1,0.57,0}
\definecolor{my-urlcolor}{rgb}{0,0,0.75}
\hypersetup{
	colorlinks, 
	linkcolor={my-linkcolor},
	citecolor={my-citecolor}, 
	urlcolor={my-urlcolor}
}

\title{Twisted modules and $G$-equivariantization in logarithmic conformal field theory}
\author{Robert McRae\\
\small \it 
Yau Mathematical Sciences Center\\
\small \it Tsinghua University, Beijing 100084, China\\
\small \textit{E--mail address:} \texttt{rhmcrae@tsinghua.edu.cn}} \date{}

\numberwithin{equation}{section}

\theoremstyle{definition}\newtheorem{rema}{Remark}[section]
\theoremstyle{plain}\newtheorem{propo}[rema]{Proposition}
\newtheorem{theo}[rema]{Theorem}
\newtheorem{mthm}{Main Theorem}
\theoremstyle{definition}\newtheorem{defi}[rema]{Definition}
\theoremstyle{plain}\newtheorem{lemma}[rema]{Lemma}
\newtheorem{corol}[rema]{Corollary}
\theoremstyle{definition}
\theoremstyle{definition}
\theoremstyle{definition}
    
\theoremstyle{definition}\newtheorem{assum}[rema]{Assumption}

\newcommand{\cY}{\mathcal{Y}}
\newcommand{\cA}{\mathcal{A}}
\newcommand{\cR}{\mathcal{R}}
\newcommand{\cM}{\mathcal{M}}

\newcommand{\cF}{\mathcal{F}}

\newcommand{\cC}{\mathcal{C}}

\newcommand{\sC}{{\mathcal{SC}}}

\newcommand{\CC}{\mathbb{C}}
\newcommand{\ZZ}{\mathbb{Z}}
\newcommand{\NN}{\mathbb{N}}

\newcommand{\RR}{\mathbb{R}}

\newcommand{\fus}[1]{\mathbin{\boxtimes_{#1}}}

\newcommand{\tens}{\boxtimes}
\newcommand{\vac}{\mathbf{1}}

\DeclareMathOperator{\coker}{Coker}

\DeclareMathOperator{\rep}{Rep}

\let\ker\relax
\let\hom\relax
\DeclareMathOperator{\ker}{Ker}
\DeclareMathOperator{\hom}{Hom}
\DeclareMathOperator{\Endo}{End}

\newcommand{\even}{{\bar{0}}}
\newcommand{\odd}{{\bar{1}}}

\newcommand{\repgV}{\rep^g V}
\newcommand{\repGV}{\rep^G V}

\usetikzlibrary{knots}

\usepackage{xparse}

\begin{document}

\bibliographystyle{alpha}

\maketitle

\begin{abstract}
 A two-dimensional chiral conformal field theory can be viewed mathematically as the representation theory of its chiral algebra, a vertex operator algebra. Vertex operator algebras are especially well suited for studying logarithmic conformal field theory (in which correlation functions have logarithmic singularities arising from non-semisimple modules for the chiral algebra) because of the logarithmic tensor category theory of Huang, Lepowsky, and Zhang. In this paper, we study not-necessarily-semisimple or rigid braided tensor categories $\mathcal{C}$ of modules for the fixed-point vertex operator subalgebra $V^G$ of a vertex operator (super)algebra $V$ with finite automorphism group $G$. The main results are that every $V^G$-module in $\mathcal{C}$ with a unital and associative $V$-action is a direct sum of $g$-twisted $V$-modules for possibly several $g\in G$, that the category of all such twisted $V$-modules is a braided $G$-crossed (super)category, and that the $G$-equivariantization of this braided $G$-crossed (super)category is braided tensor equivalent to the original category $\cC$ of $V^G$-modules. This generalizes results of Kirillov and M\"{u}ger proved using rigidity and semisimplicity. We also apply the main results to the orbifold rationality problem: whether $V^G$ is strongly rational if $V$ is strongly rational. We show that $V^G$ is indeed strongly rational if $V$ is strongly rational, $G$ is any finite automorphism group, and $V^G$ is $C_2$-cofinite.
\end{abstract}

\tableofcontents

\section{Introduction}

Orbifolding is a way to produce new conformal field theories from old ones. Mathematically, a two-dimensional (chiral) conformal field theory can be treated as the representation theory of its chiral algebra, a vertex operator algebra, and from this point of view, an orbifold conformal field theory amounts to the representation theory of the fixed-point subalgebra $V^G$ of an automorphism group $G$ of the original vertex operator algebra $V$. In the physics literature, systematic study of orbifolds for rational conformal field theories began in \cite{DVVV}. A key feature is the adddition of twisted sectors to the Hilbert space of the original conformal field theory, which correspond in mathematics to categories of twisted $V$-modules associated to automorphisms in $G$. In the mathematics literature, twisted modules for lattice vertex operator algebras had already been introduced in \cite{FLM} for the construction of the moonshine module, a vertex operator algebra on which the Monster finite simple group acts by automorphisms.

Perhaps the first question in orbifold conformal field theory is, how are representations of the orbifold $V^G$ related to the representations of $V$ and $G$? In particular, can the tensor category structures on the three categories of $V$-, $G$-, and $V^G$-modules be related? In this paper, for finite $G$, we answer these questions under minimal assumptions, namely that $V^G$ actually has a module category $\cC$ that both includes $V$ and admits vertex and braided tensor category structures as in \cite{HLZ1}-\cite{HLZ8} (see also the review article \cite{HL}). These conditions hold for many vertex operator algebras in logarithmic conformal field theory, where correlation functions have logarithmic singularities arising from (necessarily non-semisimple) modules on which the Virasoro operator $L(0)$ acts non-semisimply. For example, by \cite{H-cofin}, our results apply when $V^G$ is positive energy (also called CFT-type) and $C_2$-cofinite. Moreover, $V^G$ indeed satisfies these conditions if $V$ itself is simple, positive energy, $C_2$-cofinite and $G$ is finite solvable \cite{Mi1}. (An effort to extend this result to general finite groups has recently appeared in \cite{Mi}, but unfortunately there appears to be a gap in the argument at the moment.) 

In the narrower setting of rational conformal field theory, when $V^G$ is a so-called ``strongly rational'' vertex operator algebra (with a semisimple modular tensor category of
representations \cite{H-rigidity}), it follows from the work of Kirillov \cite{KirillovOrbifoldI}-\cite{KirillovGequivariant} and M\"{u}ger \cite{Mu1}-\cite{Mu2} how to describe the category of $V^G$-modules in terms of $V$- and $G$-modules. Take the category $\rep\,V$ of all $V^G$-modules which admit suitably associative and unital actions of $V$. This category is semisimple and every simple object is a $g$-twisted $V$-module for some $g\in G$, and from there it follows that $\rep\,V$ is a braided $G$-crossed tensor category in the sense of Turaev \cite{TuraevGcrossed}. Any braided $G$-crossed tensor category has an associated braided tensor category called the $G$-equivariantization. For $\rep\,V$, objects of the equivariantization $(\rep\,V)^G$ are modules in $\rep\,V$ equipped with a compatible $G$-module structure; morphisms commute with both $V$- and $G$-actions. Finally, $(\rep\,V)^G$ is braided tensor equivalent to the original category of $V^G$-modules: taking $G$-fixed points gives the functor from $(\rep\,V)^G$ to $V^G$-modules, while there is an induction functor in the other direction.

The main result of this paper is that these results remain true for vertex operator algebras, and indeed vertex operator superalgebras, in logarithmic conformal field theory. Summarizing Theorems \ref{thm:mainVOAthm1} and \ref{thm:mainVOAthm2}:
\begin{mthm}\label{mthm1}
 Let $V$ be a simple vertex operator superalgebra, $G$ a finite automorphism group of $V$ that includes the parity involution, and $\cC$ an abelian category of strongly-graded generalized $V^G$-modules that includes $V$ and admits vertex and braided tensor category structures as in \cite{HLZ1}-\cite{HLZ8}. Then:
 \begin{enumerate}
  \item Every indecomposable object of $\rep\,V$ is a $g$-twisted $V$-module for some $g\in G$.
  \item The category $\rep\,V$ admits the structure of a braided $G$-crossed supercategory.
  \item The induction functor $\cF: \cC\rightarrow(\rep\,V)^G$ is an equivalence of braided tensor categories.
 \end{enumerate}
\end{mthm}

In particular, we do not assume that any category of $V^G$-modules is semisimple. More importantly, we do not assume that any category of $V^G$-modules is rigid, that is, that its objects have duals. It is usually difficult to show that a tensor category of modules for a vertex operator algebra is rigid, and the only general rigidity theorem, due to Huang \cite{H-rigidity}, applies only to rational vertex operator algebras. This is why the work of Kirillov and M\"{u}ger, which uses rigidity, does not apply in the generality of the main theorem. For example, \cite{KirillovOrbifoldI, KirillovOrbifoldII}  show that a simple object $W$ in $\rep\,V$ is $g$-twisted by constructing $g$, which must be an endomorphism of $V$ having something to do with $W$. But the only way to construct such an endomorphism is to create a copy of $W$ (and its dual) using rigidity and have them interact with $V$ in some way.

In Section \ref{sec:MainCatThm}, we prove that indecomposable objects of $\rep\,V$ are $g$-twisted by constructing the complete set of projections from any object $W$ in $\rep\,V$ onto its $g$-twisted summands for $g\in G$. The formula for the projections generalizes a construction from \cite{KO} and requires rigidity only for $V$ itself: the projection $\pi_g$ onto the $g$-twisted summand must be an endomorphism of $W$ having something to do with $g$, so the formula for $\pi_g$ uses rigidity to create a copy of $V$ on which $g$ can act. But we do not even need to assume the rigidity of $V$ as a $V^G$-module, because this follows from results in \cite{McR}. For $G$ finite abelian, a variant of Main Theorem \ref{mthm1}(1), also without assuming rigidity, has appeared previously in \cite[Corollary 4.3]{CarM}.

Main Theorem \ref{mthm1} can be used when rigidity is expected but not known \textit{a priori}. In Section \ref{sec:OrbRat}, we consider the orbifold rationality problem: if $V$ is strongly rational and $G$ is finite, is $V^G$ also strongly rational? Carnahan and Miyamoto have shown \cite{CarM} that the answer is yes if $G$ is cyclic (and by extension if $G$ is solvable), but the question has remained open for general finite $G$. One needs to show that $V^G$ is both $C_2$-cofinite and rational, and in the presence of $C_2$-cofiniteness, rationality means that the category of (strongly-graded) $V^G$-modules is semisimple. In Theorem \ref{thm:VG_ss} and Corollary \ref{cor:OrbRat} below, we reduce the orbifold rationality problem to the question of $C_2$-cofiniteness for $V^G$:
\begin{mthm}\label{mthm2}
 Let $V$ be a strongly rational vertex operator algebra and $G$ any finite group of automorphisms of $V$. If $\cC$ is an abelian category of strongly-graded generalized $V^G$-modules that includes $V$ and admits vertex tensor category structure as in \cite{HLZ1}-\cite{HLZ8}, then $\cC$ is semisimple. In particular, if $V^G$ is $C_2$-cofinite, then $V^G$ is strongly rational.
\end{mthm}

The idea is to show that $\rep\,V$ is semisimple under the conditions of Main Theorem \ref{mthm2}, because it is easy to show that this implies $\cC$ is semisimple. But because every module in $\rep\,V$ is a direct sum of $g$-twisted modules for certain $g\in G$, it is enough to show that the category of $g$-twisted modules is semisimple for any fixed $g\in G$. This then follows using the rationality of each $V^{\langle g\rangle}$ proved in \cite{CarM}.

The equivalence in Main Theorem \ref{mthm1} between $\cC$ and the $G$-equivariantization of $\rep\,V$ has interesting implications even in what is perhaps the simplest non-trivial case: $V$ is a superalgebra and $G\cong\ZZ/2\ZZ$ is generated by the parity involution, so that $V^G$ is the even vertex operator subalgebra $V^\even$. In this case, the objects of $(\rep\,V)^{\ZZ/2\ZZ}$ are simply the ordinary and parity-twisted $V$-modules (Neveu-Schwarz and Ramond sectors in physics terminology) equipped with parity decompositions. In particular, morphisms in $(\rep\,V)^{\ZZ/2\ZZ}$ must be even. Applying the induction functor then shows that every indecomposable module in the base category $\cC$ of $V^\even$-modules is the even part of either a Neveu-Schwarz or Ramond $V$-module (equivalently, the odd part since reversing the parity decomposition yields another module in $(\rep\,V)^{\ZZ/2\ZZ}$).

Examples of vertex operator superalgebras with non-semisimple representation theory that can be studied using Main Theorem \ref{mthm1} include the symplectic fermion superalgebras $SF(d)$, $d\in\ZZ_+$, of $d$ pairs of symplectic fermions \cite{Ka,Ab, Ru}. As the even subalgebras $SF(d)^\even$ are $C_2$-cofinite \cite{Ab}, the full categories of grading-restricted, generalized $SF(d)^\even$-modules have braided tensor category structure \cite{H-cofin}. In fact, one of the goals that stimulated the work in this paper was to understand Runkel's construction \cite{Ru} of a braided tensor category conjecturally equivalent to the Huang-Lepowsky-Zhang braided tensor category (as in \cite{HLZ1}-\cite{HLZ8}) of strongly-graded generalized $SF(d)^\even$-modules. The braided tensor category in \cite{Ru} was constructed using Neveu-Schwarz and Ramond $SF(d)$-modules and seems to be the $\ZZ/2\ZZ$-equivariantization of $\rep\,SF(d)$. In future work, we plan to verify the identification of $(\rep\,SF(d))^{\ZZ/2\ZZ}$ with Runkel's braided tensor category, thus proving (in light of Main Theorem \ref{mthm1}(3)) the conjectured equivalence with the Huang-Lepowsky-Zhang braided tensor category of $SF(d)^\even$-modules. As Runkel's category is braided tensor equivalent to a non-semisimple modular tensor category of finite-dimensional modules for a quasi-Hopf algebra \cite{GR, FGR}, this would provide a family of examples of non-rational $C_2$-cofinite vertex operator algebras whose module categories admit non-semisimple modular tensor category structure.


The methods of this paper are primarily categorical, using the theory of commutative associative (super)algebras in braided tensor categories developed in, for instance, \cite{KO} and expanded upon in \cite{CKM}. To apply results on algebras in tensor categories to vertex operator algebras, we use the identification from \cite{HKL} of vertex operator algebra extensions with algebras in a braided tensor category of modules for a vertex operator subalgebra. This result was extended to superalgebras in \cite{CKL}, and the relationship between the module categories of vertex operator (super)algebra extensions and (super)algebras in a braided tensor category was established in \cite{CKM}. Tensor-categorical techniques have proven highly useful in the representation theory of vertex operator algebras in recent years: in this paper, we use braid diagrams to concisely express proofs that would otherwise require complex manipulations of compositions of up to four vertex algebraic intertwining operators ($6$-point correlation functions in conformal-field-theoretic terminology). That tensor-categorical techniques can be used at all to study vertex operator algebras is a consequence of the work of Huang-Lepowsky-(Zhang), culminating in \cite{HLZ1}-\cite{HLZ8}, developing the (logarithmic) vertex tensor category theory for module categories of a vertex operator algebra. The reader may notice that while the present paper is designed to handle vertex operator algebras in logarithmic conformal field theory, very few logarithms appear explicitly. This is because most of the complex analysis needed for this paper has been worked out already in \cite{HLZ1}-\cite{HLZ8} as well as subsequent works such as \cite{CKM}.

The remaining contents of this paper are structured as follows. In Section \ref{subsec:Superalgebras}, we present definitions and basic results on superalgebras $V$ in braided tensor categories $\cC$, including an overview of the monoidal supercategory structure on the representation category $\rep\,V$ and the induction tensor functor from $\cC$ to $\rep\,V$ (see \cite{CKM} for a fuller discussion). In Section \ref{sec:TwistMods}, we define the notion of $g$-twisted $V$-module in $\rep\,V$, associated to an automorphism $g$ of the superalgebra $V$, as well as the notion of braided $G$-crossed supercategory, a suitable supercategory version of the notion of braided $G$-crossed category from \cite{TuraevGcrossed}. We also discuss how categories of twisted $V$-modules provide examples of braided $G$-crossed supercategories. This result seems to be well known, since it is stated as a theorem in \cite{KirillovGequivariant}, but a detailed proof seems to be missing from the literature. We provide one in Appendix \ref{app:BGC_proof}. In Section \ref{sec:Gequiv} we give the definition of $G$-equivariantization of a braided $G$-crossed supercategory and show that in the case of twisted modules for a superalgebra, induction is a braided tensor functor from $\cC$ to the $G$-equivariantization $(\rep\,V)^G$, provided that every object in $\rep\,V$ is a direct sum of twisted modules. This result is also known, appearing in \cite{KirillovOrbifoldI} and \cite{Mu2}, but we include a full proof to emphasize that it does not require $\cC$ to be semisimple or rigid.

In Section \ref{sec:MainCatThm}, we prove the main categorical result of this paper (Theorem \ref{thm:repV=repGV}): under some conditions which do not include rigidity or semisimplicity of $\cC$, every object in $\rep\,V$ is a direct sum of $g$-twisted $V$-modules for possibly several $g\in G$. In particular, $\rep\,V$ is a braided $G$-crossed supercategory and induction is a braided tensor functor from $\cC$ to $(\rep\,V)^G$. In the proof of Theorem \ref{thm:repV=repGV}, we use braid diagrams for brevity and clarity, but the reader interested in fuller details may consult Appendix \ref{app:mthm_details}.

In Section \ref{sec:VOAs}, we interpret the categorical results of the previous two sections as theorems for vertex operator (super)algebras using the connection between vertex operator superalgebra extensions and superalgebras in braided tensor categories established in \cite{HKL, CKL, CKM}. After reviewing the definitions of vertex operator superalgebras and their $g$-twisted modules in Section \ref{subsec:VOAdefs}, we prove the main general theorems in Section \ref{subsec:VOAthms}. First we show that the categorical and vertex algebraic notions of $g$-twisted $V$-module agree when $V$ is a vertex operator superalgebra, and then we prove the first two parts of Main Theorem \ref{mthm1} by verifying the conditions of Theorem \ref{thm:repV=repGV} using \cite{DLM} and \cite{McR}. Then after arguing that the braided $G$-crossed supercategory structure on $\rep\,V$ is natural from a vertex algebraic point of view, we prove the third part of Main Theorem \ref{mthm1} in Theorem \ref{thm:mainVOAthm2}. In Section \ref{subsec:VOSAs}, we describe the braided tensor category structure on the $\ZZ/2\ZZ$-equivariantization of $\rep\,V$ when $V$ is a superalgebra and $\ZZ/2\ZZ$ is generated by the parity automorphism; by Main Theorem \ref{mthm1}(3), this provides a description of the braided tensor category of modules for the even subalgebra $V^\even$. Finally, in Section \ref{sec:OrbRat}, we prove Main Theorem \ref{mthm2}.

\paragraph{Acknowledgements}
This work was partially supported by the United States National Science Foundation grant DMS-1362138. I would also like to thank Thomas Creutzig and Shashank Kanade for discussions.

\section{Braided \texorpdfstring{$G$}{G}-crossed supercategories}\label{sec:BGXSC}

In this section, we collect definitions and results on braided $G$-crossed supercategories of twisted modules for a superalgebra $V$ in a braided tensor category associated to a group $G$ of automorphisms of $V$. For more details refer, for example, to \cite{TuraevGcrossed}, \cite{KO}, \cite{KirillovOrbifoldI}-\cite{KirillovGequivariant}, and \cite{CKM}.

\subsection{Superalgebra objects in supercategories}\label{subsec:Superalgebras}

See \cite{BE} for a helpful analysis of the relations between various notions of supercategory in the mathematical literature. Here we fix a field $\mathbb{F}$ of characteristic not equal to $2$ and work with $\mathbb{F}$-linear \textit{supercategories} (in the terminology of \cite{BE}) whose morphism sets are $\mathbb{F}$-superspaces. We will denote the parity of a morphism $f$ in a supercategory by $\vert f\vert\in\ZZ/2\ZZ$. Composition of morphisms in an $\mathbb{F}$-linear supercategory is an even linear map between superspaces, in the sense that
\begin{equation*}
 \vert f g\vert=\vert f\vert+\vert g\vert
\end{equation*}
when $f$ and $g$ have parity. Superfunctors between supercategories induce even linear maps on morphisms. If $\sC$ is a supercategory, then so is $\sC\times\sC$ with composition of morphisms defined by
\begin{equation}\label{compinSCtimesSC}
 (f_1, f_2)(g_1,g_2)=(-1)^{\vert f_2\vert\vert g_1\vert}(f_1 g_1, f_2 g_2)
\end{equation}
when $f_2$ and $g_1$ have parity. Also, $\sC\times\sC$ has a supersymmetry superfunctor $\sigma$ given on objects by $\sigma(W_1,W_2)=(W_2,W_1)$ and on parity-homogeneous morphisms by $\sigma(f_1,f_2)=(-1)^{\vert f_1\vert\vert f_2\vert} (f_2,f_1)$.

A monoidal supercategory $\sC$ is a supercategory with monoidal category structure: a tensor product superfunctor $\tens : \sC\times\sC\rightarrow \sC$, a unit object $\vac$, even natural left and right unit isomorphisms $l$ and $r$, and even natural associativity isomorphisms $\cA$, which satisfy the triangle and pentagon axioms. Since $\boxtimes$ is a superfunctor, it induces an even linear map on morphisms: $\vert f_1\tens f_2\vert=\vert f_1\vert+\vert f_2\vert$ when $f_1$ and $f_2$ are parity-homogeneous, and by \eqref{compinSCtimesSC},
\begin{equation}\label{superinterchange}
 (f_1\boxtimes f_2)(g_1\boxtimes g_2)=(-1)^{\vert f_2\vert\vert g_1\vert}(f_1 g_1)\boxtimes(f_2 g_2)
\end{equation}
for appropriately composable morphisms with $f_2$ and $g_1$ parity-homogeneous (this is called the super interchange law in \cite{BE}). 

The monoidal supercategory $\sC$ is \textit{braided} if it has an even natural braiding isomorphism $\cR: \tens\rightarrow\tens\circ\sigma$. From the definition of $\sigma$, for parity-homogeneous morphisms $f_1: W_1\rightarrow\widetilde{W}_1$ and $f_2: W_2\rightarrow\widetilde{W}_2$ in $\sC$,
\begin{equation}\label{Rnaturality}
 \cR_{\widetilde{W}_1, \widetilde{W}_2}(f_1\tens f_2)=(-1)^{\vert f_1\vert\vert f_2\vert}(f_2\boxtimes f_1)\cR_{W_1,W_2}.
\end{equation}
We will say that a (braided) monoidal supercategory is \textit{rigid} if every object $W$ has a (left) dual $W^*$ for which the evaluation morphism $e_W: W^*\tens W\rightarrow\vac$ and coevaluation morphism $i_W: \vac\rightarrow W\tens W^*$ are both even.

All $\mathbb{F}$-linear supercategories $\sC$ in this paper will be $\mathbb{F}$-additive in the sense that $\sC$ has a zero object and the biproduct $\bigoplus W_i$ of any finite set of objects $\lbrace W_i\rbrace$ in $\sC$ exists. We will not generally require our supercategories to be abelian, as kernels and cokernels of non-parity-homogeneous morphisms may not exist in our examples. By an $\mathbb{F}$-additive monoidal supercategory, we will mean an $\mathbb{F}$-additive supercategory with a monoidal supercategory structure such that the tensor product of morphisms is bilinear.

Now we recall the main definitions and results from \cite[Section 2.2]{CKM}. Fix an $\mathbb{F}$-linear (abelian) braided tensor category $\cC$, with unit object $\vac$, left and right unit isomorphisms $l$ and $r$, associativity isomorphisms $\cA$, and braiding isomorphisms $\cR$. The only additional requirement is that for any object $W$ in $\cC$, the functors $W\tens\bullet$ and $\bullet\tens W$ are right exact. We define an auxiliary supercategory $\sC$ whose objects are ordered pairs $W=(W^\even, W^\odd)$, with $W^\even$ and $W^\odd$ objects of $\cC$, and whose morphisms are given by
\begin{equation*}
 \hom_{\sC}(W_1, W_2)=\hom_{\cC}(W_1^\even\oplus W_1^\odd, W_2^\even\oplus W_2^\odd).
\end{equation*}
Every object $W$ of $\sC$ has a parity involution $P_W\in\mathrm{End}_{\sC}(W)$ given by
\begin{equation*}
 P_W = 1_{W^\even}\oplus(-1_{W^\odd}).
\end{equation*}
The parity involutions determine the superspace structure of the morphism spaces in $\sC$: $f\in\hom_{\sC}(W_1, W_2)$ has parity $\vert f\vert\in\ZZ/2\ZZ$ if
\begin{equation*}
 f P_{W_1}=(-1)^{\vert f\vert} P_{W_2} f.
\end{equation*}
The supercategory $\sC$ is also $\mathbb{F}$-additive with zero object $0=(0,0)$ and biproducts defined by $W_1\oplus W_2 =(W_1^\even\oplus W_2^\even, W_1^\odd\oplus W_2^\odd)$. Moreover, $\sC$ is abelian, with every morphism having a kernel and cokernel, because $\cC$ is. However, if $f: W_1\rightarrow W_2$ is a morphism in $\sC$, we cannot assume the kernel morphism $k: \ker f\rightarrow W_1$ and cokernel morphism $c: W_2\rightarrow\coker f$ are even. They can be taken to be even if $f$ is parity homogeneous. Moreover every parity-homogeneous monomorphism in $\sC$ is the kernel of an even morphism, and every parity-homogeneous epimorphism in $\sC$ is the cokernel of an even morphism (see \cite[Proposition 2.15]{CKM}).

Next, $\sC$ has braided tensor supercategory structure as follows: for objects $W_1$, $W_2$ in $\sC$,
\begin{equation*}
 W_1\tens W_2 = \left( (W_1^\even\tens W_2^\even)\oplus(W_1^\odd\tens W_2^\odd), (W_1^\even\tens W_2^\odd)\oplus(W_1^\odd\tens W_2^\even)\right).
\end{equation*}
For morphisms $f_1: W_1\rightarrow\widetilde{W}_1$ and $f_2: W_2\rightarrow\widetilde{W}_2$ in $\sC$, the tensor product $f_1\tens f_2$ must be a $\cC$-morphism
\begin{equation*}
 \bigoplus_{i_1,i_2\in\ZZ/2\ZZ} W_1^{i_1}\tens W_2^{i_2} \rightarrow\bigoplus_{j_1,j_2\in\ZZ/2\ZZ} \widetilde{W}_1^{j_1}\tens\widetilde{W}_2^{j_2}.
\end{equation*}
Since the tensor product in $\cC$ distributes over biproducts, $f_1\tens f_2$ in $\sC$ can be identified with a $\cC$-morphism
\begin{equation*}
 (W_1^\even\oplus W_1^\odd)\tens(W_2^\even\oplus W_2^\odd)\rightarrow (\widetilde{W}_1^\even\oplus\widetilde{W}_1^\odd)\tens(\widetilde{W}_2^\even\oplus\widetilde{W}_2^\odd),
\end{equation*}
and this $\cC$-morphism is $\big(f_1 P_{W_1}^{\vert f_2\vert}\big)\tens f_2$ when $f_2$ is parity-homogeneous. The factor $P_{W_1}^{\vert f_2\vert}$ allows the super interchange law \eqref{superinterchange} to hold. 

The unit object in $\sC$ is $\vac=(\vac, 0)$, so for any object $W=(W^\even, W^\odd)$ in $\sC$, we can identify
\begin{equation*}
 \vac\tens W=(\vac\tens W^\even, \vac\tens W^\odd),\hspace{3em} W\tens\vac =(W^\even\tens\vac, W^\odd\tens\vac)
\end{equation*}
by first identifying $0\tens W^i=0=W^i\tens 0$ for $i\in\ZZ/2\ZZ$ and then identifying $(\vac\tens W^i)\oplus 0=\vac\tens W^i$. Under these identifications, the unit isomorphisms for $W$ in $\sC$ are given by $l_W=l_{W^\even}\oplus l_{W^\odd}$ and $r_W=r_{W^\even}\oplus r_{W^\odd}$. The associativity isomorphism for objects $W_1$, $W_2$, and $W_3$ in $\sC$ is
\begin{equation*}
 \cA_{W_1,W_2,W_3}=\bigg(\bigoplus_{i_1+i_2+i_3=\even} \cA_{W_1^{i_1}, W_2^{i_2}, W_3^{i_3}}\bigg)\oplus\bigg(\bigoplus_{i_1+i_2+i_3=\odd} \cA_{W_1^{i_1}, W_2^{i_2}, W_3^{i_3}}\bigg).
\end{equation*}
For objects $W_1$ and $W_2$ in $\sC$, the braiding isomorphism is
\begin{equation*}
 \cR_{W_1,W_2}=\bigg(\bigoplus_{i_1+i_2=\even} (-1)^{i_1 i_2} \cR_{W_1^{i_1}, W_2^{i_2}}\bigg)\oplus\bigg(\bigoplus_{i_1+i_2=\odd} (-1)^{i_1 i_2} \cR_{W_1^{i_1}, W_2^{i_2}}\bigg).
\end{equation*}
The sign factors in the braiding isomorphisms guarantee that \eqref{Rnaturality} holds. As for $\cC$, the functors $W\tens\bullet$ and $\bullet\tens W$ on $\sC$ are right exact for any object $W$.

Now we define a superalgebra in the braided tensor category $\cC$ to be a commutative associative algebra, with even structure morphisms, in $\sC$. Specifically:
\begin{defi}
 A \textit{superalgebra} in $\cC$ is an object $V=(V^\even, V^\odd)$ in $\sC$ equipped with even morphisms $\mu_V: V\tens V\rightarrow V$ and $\iota_V: \vac\rightarrow V$ that satisfy the following conditions:
 \begin{enumerate}
 \item Left unit: $\mu_V(\iota_V\tens 1_V) l_V^{-1} = 1_V$.
 
  \item Associativity:  $\mu_V(1_V\tens\mu_V)= \mu_V(\mu_V\tens 1_V)\cA_{V,V,V}$.
  
  \item Supercommutativity: $\mu_V= \mu_V\cR_{V,V}$.
 \end{enumerate}
\end{defi}

\begin{rema}
 No sign factor is needed in the supercommutativity axiom because this is built into the braiding isomorphisms in $\sC$. Also, the left unit and supercommutativity axioms together imply the right unit property: $\mu_V(1_V\tens\iota_V) r_V^{-1} = 1_V$.
\end{rema}

Given a superalgebra $V$ in $\cC$, we define the supercategory $\rep V$ of (left) $V$-modules with objects $(W,\mu_W)$ where $W$ is an object of $\sC$ and $\mu_W: V\tens W\rightarrow W$ is an even morphism in $\sC$ satisfying
\begin{enumerate}
 \item Unit: $\mu_W(\iota_V\tens 1_W) l_W^{-1}=1_W$, and
 \item Associativity: $\mu_W(1_V\tens\mu_W)= \mu_W(\mu_V\tens 1_W)\cA_{V,V,W}$.
\end{enumerate}
A morphism $f: (W_1,\mu_{W_1})\rightarrow(W_2,\mu_{W_2})$ in $\rep V$ is an $\sC$-morphism $f: W_1\rightarrow W_2$ such that
\begin{equation*}
 f\mu_{W_1}=\mu_{W_2}(1_V\tens f).
\end{equation*}
The parity of a morphism in $\rep V$ agrees with its parity as a morphism in $\sC$. The supercategory $\rep V$ is an $\mathbb{F}$-additive supercategory (see for instance \cite[Proposition 2.32]{CKM}) but is not necessarily abelian because the natural actions of $V$ on the kernel and cokernel of a morphism in $\rep V$ might not be even unless the morphism is parity-homogeneous. However, the underlying category $\underline{\rep V}$, which has the same objects as $\rep V$ but only the even morphisms, is an $\mathbb{F}$-linear abelian category.

The supercategory $\rep V$ also has a monoidal supercategory structure as follows. Given two objects $W_1$, $W_2$ of $\rep V$, $V$ can act on either factor of $W_1\tens W_2$: define $\mu_{W_1,W_2}^{(1)}$ to be the composition
\begin{equation*}
 V\tens(W_1\tens W_2)\xrightarrow{\cA_{V,W_1,W_2}} (V\tens W_1)\tens W_2\xrightarrow{\mu_{W_1}\tens 1_{W_2}} W_1\tens W_2
\end{equation*}
and define $\mu_{W_1,W_2}^{(2)}$ to be the composition
\begin{align*}
V\tens(W_1\tens W_2)\xrightarrow{\cA_{V,W_1,W_2}} (V\tens W_1)\tens W_2 & \xrightarrow{\cR_{V,W_1}\tens 1_{W_2}} (W_1\tens V)\tens W_2\nonumber\\
&\xrightarrow{\cA_{W_1,V,W_2}^{-1}} W_1\tens(V\tens W_2)\xrightarrow{1_{W_1}\tens\mu_{W_2}} W_1\tens W_2.
\end{align*}
Then the tensor product of $W_1$ and $W_2$ in $\rep V$, $W_1\tens_V W_2$, is the cokernel of $\mu_{W_1,W_2}^{(1)}-\mu_{W_1,W_2}^{(2)}$, which exists because $\sC$ is abelian. Let $I_{W_1,W_2}: W_1\tens W_2\rightarrow W_1\tens_V W_2$ denote the cokernel morphism, which we take even, as we may since $\mu_{W_1,W_2}^{(1)}-\mu_{W_1,W_2}^{(2)}$ is an even morphism in $\sC$. The multiplication action $\mu_{W_1\tens_V W_2}$ is characterized by the commutative diagram
\begin{equation*}
 \xymatrixcolsep{6pc}
 \xymatrix{
 V\tens(W_1\tens W_2) \ar[r]^(.55){\mu_{W_1,W_2}^{(1)}\,or\,\mu_{W_1,W_2}^{(2)}} \ar[d]^{1_V\tens I_{W_1,W_2}} & W_1\tens W_2 \ar[d]^{I_{W_1,W_2}}\\
 V\tens(W_1\tens_V W_2) \ar[r]^(.55){\mu_{W_1\tens_V W_2}} & W_1\tens_V W_2\\
 }
\end{equation*}
The tensor product of morphisms $f_1: W_1\rightarrow\widetilde{W}_1$ and $f_2: W_2\rightarrow\widetilde{W}_2$ in $\rep V$ is characterized by the commuting diagram
\begin{equation*}
 \xymatrixcolsep{4pc}
 \xymatrix{
 W_1\tens W_2 \ar[r]^{f_1\tens f_2} \ar[d]^{I_{W_1,W_2}} & \widetilde{W}_1\tens\widetilde{W}_2 \ar[d]^{I_{\widetilde{W}_1,\widetilde{W}_2}} \\
 W_1\tens_V W_2 \ar[r]^{f_1\tens_V f_2} & \widetilde{W}_1\tens_V\widetilde{W}_2\\
 }
\end{equation*}

In \cite[Proposition 2.47]{CKM}, we showed that the universal property of the cokernel tensor product in $\rep V$ can be expressed in terms of what we called categorical $\rep V$-intertwining operators. For objects $W_1$, $W_2$, and $W_3$, of $\rep V$, a categorical $\rep V$-intertwining operator of type $\binom{W_3}{W_1\,W_2}$ is an $\sC$-morphism $I: W_1\tens W_2\rightarrow W_3$ such that
\begin{equation*}
 I\mu_{W_1,W_2}^{(1)}=I\mu_{W_1,W_2}^{(2)}=\mu_{W_3}(1_V\tens I): V\tens(W_1\tens W_2)\rightarrow W_3.
\end{equation*}
That is, an intertwining operator intertwines the left and right actions of $V$ on $W_1\tens W_2$ with the action of $V$ on $W_3$. Examples of intertwining operators include $\mu_W$ of type $\binom{W}{V\,W}$ for an object $W$ of $\rep V$, and $I_{W_1,W_2}$ of type $\binom{W_1\tens_V W_2}{W_1\,W_2}$ for objects $W_1$ and $W_2$.  Any categorical intertwining operator of type $\binom{W_3}{W_1\,W_2}$ induces a unique $\rep V$-morphism $f_I: W_1\tens_V W_2\rightarrow W_3$ such that $I=f_I I_{W_1,W_2}$. For example the tensor product of morphisms $f_1: W_1\rightarrow\widetilde{W}_1$ and $f_2: W_2\rightarrow\widetilde{W}_2$ is induced by the intertwining operator $I_{\widetilde{W}_1,\widetilde{W}_2}(f_1\tens f_2)$ of type $\binom{\widetilde{W}_1\tens_V\widetilde{W}_2}{W_1\,W_2}$.

The unit object of $\rep\,V$ is $(V, \mu_V)$ and the unit isomorphisms $l_W^V$ and $r_W^V$ for an object $W$ of $\rep V$ are characterized by the commuting diagrams
\begin{equation*}
 \xymatrixcolsep{4pc}
\xymatrix{
V\tens W \ar[rd]^{\mu_W} \ar[d]^{I_{V,W}} & \\
 V\tens_V W \ar[r]^{l^V_W} & W
 }
\end{equation*}
and
\begin{equation*}
 \xymatrixcolsep{4pc}
 \xymatrix{
 W\tens V \ar[r]^{\cR_{V,W}^{-1}} \ar[d]^{I_{W,V}} & V\tens W \ar[d]^{\mu_W} \\
 W\tens_V V \ar[r]^{r_W^V} & W
 }
\end{equation*}
The associativity isomorphism for objects $W_1$, $W_2$, and $W_3$ in $\rep V$ is characterized by the commutative diagram
\begin{equation*}
 \xymatrixcolsep{6pc}
 \xymatrix{
 W_1\tens(W_2\tens W_3) \ar[r]^{\cA_{W_1,W_2,W_3}} \ar[d]^{1_{W_1}\tens I_{W_2,W_3}} & (W_1\tens W_2)\tens W_3 \ar[d]^{I_{W_1, W_2}\tens 1_{W_3}}\\
 W_1\tens(W_2\tens_V W_3) \ar[d]^{I_{W_1,W_2\tens_V W_3}} & (W_1\tens_V W_2)\tens W_3 \ar[d]^{I_{W_1\tens_V W_2, W_3}} \\
 W_1\tens_V (W_2\tens_V W_3) \ar[r]^{\cA^V_{W_1,W_2,W_3}} & (W_1\tens_V W_2)\tens_V W_3 \\
 }
\end{equation*}

We have an induction superfunctor $\cF: \sC\rightarrow\rep V$ defined on objects by
\begin{equation*}
 \cF(W)=(V\tens W,\mu_{\cF(W)})
\end{equation*}
 where $\mu_{\cF(W)}$ is the composition
\begin{equation*}
 V\tens(V\tens W)\xrightarrow{\cA_{V,V,W}} (V\tens V)\tens W\xrightarrow{\mu_V\tens 1_W} V\tens W.
\end{equation*}
On morphisms, we define $\cF(f)=1_V\tens f$. Induction is a tensor superfunctor: there is an even isomorphism $\varphi: \cF(\vac)\rightarrow V$ (given by $r_V$) and an even natural isomorphism $f: \cF\circ\tens\rightarrow\tens_V\circ(\cF\times\cF)$, where $f_{W_1,W_2}$ is defined as the composition
\begin{align*}
 V\tens(W_1\tens W_2)\xrightarrow{\cA_{V,W_1,W_2}} & (V\tens W_1)\tens W_2\xrightarrow{1_{V\tens W_1}\tens l_{W_2}^{-1}} (V\tens W_1)\tens(\vac\tens W_2)\nonumber\\
 & \xrightarrow{1_{V\tens W_1}\tens(\iota_V\tens 1_{W_2})} (V\tens W_1)\tens(V\tens W_2)\xrightarrow{I_{V\tens W_1,V\tens W_2}} (V\tens W_1)\tens_V(V\tens W_2).
\end{align*}
These isomorphisms are compatible with the unit and associativity isomorphisms of $\sC$ and $\rep V$ in the required sense. Induction is left adjoint to the obvious forgetful functor from $\rep V$ to $\sC$ since if $W$ is an object of $\sC$, $X$ is an object of $\rep V$ and $f: W\rightarrow X$ is a morphism in $\sC$, there is a unique morphism $\Psi(f): \cF(W)\rightarrow X$ such that the diagram
\begin{equation*}
 \xymatrixcolsep{4pc}
 \xymatrix{
 \cF(W)=V\tens W \ar[rd]^{\Psi(f)} & \\
 W \ar[u]^{(\iota_V\tens 1_W) l_W^{-1}} \ar[r]^{f} & X\\
 }
\end{equation*}
commutes. In fact, $\Psi(f)=\mu_X(1_V\tens f)$.

\subsection{Braided \texorpdfstring{$G$}{G}-crossed supercategories of twisted modules}\label{sec:TwistMods}

 Fix a superalgebra $V$ in a braided tensor category $\cC$ with right exact tensor functors $W\tens\bullet$ and $\bullet\tens W$ for any object $W$ in $\cC$. We say that a subgroup $G\subseteq\mathrm{Aut}_{\sC}(V)^\even$ is an automorphism group if
\begin{equation*}
 g\mu_V=\mu_V(g\tens g)
\end{equation*}
and
\begin{equation*}
 g\iota_V=\iota_V
\end{equation*}
for every $g\in G$. Fix an automorphism group $G$ of $V$.
\begin{defi}\label{def:CatTwistMods}
For $g\in G$, an object $(W,\mu_W)$ in $\rep V$ is a \textit{$g$-twisted $V$-module} if
\begin{equation*}
 \mu_W(g\tens 1_W)\cM_{V,W}=\mu_W,
\end{equation*}
where $\cM_{V,W}=\cR_{W,V}\cR_{V,W}$ is the natural \textit{monodromy} isomorphism in $\sC$.
\end{defi}
For $g\in G$, let $\repgV$ denote the full subcategory of $g$-twisted $V$-modules in $\rep V$. Then define $\repGV$ to be the full subcategory of $\rep V$ whose objects are isomorphic to finite biproducts of $g$-twisted $V$-modules for possibly several different $g\in G$. The category $\repGV$ is an $\mathbb{F}$-additive monoidal supercategory. Indeed, the unit object $(V,\mu_V)$ of $\rep V$ is in $\rep^1 V$ by the supercommutativity of $\mu_V$ (we say that objects in $\rep^1 V$ are \textit{untwisted}), and to show that $\repGV$ is closed under tensor products, we use the following result which is essentially part of \cite[Theorem 4.7 (4)]{KirillovOrbifoldI} where, however, the proof used strong assumptions on $\cC$ that we do not need here:
\begin{propo}\label{prop:TensProdandGgrading}
 If $W_1$ is a $g_1$-twisted $V$-module, $W_2$ is a $g_2$-twisted $V$-module, and $I$ is a surjective intertwining operator of type $\binom{W_3}{W_1\,W_2}$, then $W_3$ is a $g_1 g_2$-twisted $V$-module.
\end{propo}
\begin{proof}
 We need to show that $\mu_{W_3}(g_1 g_2\tens 1_{W_3})\cM_{V, W_3}=\mu_{W_3}$. Since $I$ is surjective and $V\tens\bullet$ is right exact, $1_V\tens I$ is surjective as well, and it is sufficient to prove that
 \begin{equation*}
  \mu_{W_3}(g_1g_2\tens 1_{W_3})\cM_{V,W_3}(1_V\tens I)=\mu_{W_3}(1_V\tens I).
 \end{equation*}
Using naturality of the monodromy isomorphisms, the left side of this equation equals the composition
\begin{align*}
 V\tens(W_1\tens W_2)\xrightarrow{\cM_{V,W_1\tens W_2}} V\tens(W_1\tens W_2)\xrightarrow{g_1g_2\tens 1_{W_1\tens W_2}} V\tens(W_1\tens W_2)\xrightarrow{1_V\tens I} V\tens W_3\xrightarrow{\mu_{W_3}} W_3.
\end{align*}
Using the hexagon axiom and the fact that $I$ is an intertwining operator, this composition becomes
\begin{align*}
 V\tens(W_1\tens & W_2)  \xrightarrow{\cA_{V,W_1,W_2}} (V\tens W_1)\tens W_2\xrightarrow{\cR_{V,W_1}\tens 1_{W_2}} (W_1\tens V)\tens W_2\xrightarrow{\cA_{W_1,V,W_2}^{-1}} W_1\tens(V\tens W_2)\nonumber\\
 &\xrightarrow{1_{W_1}\tens\cM_{V,W_2}} W_1\tens(V\tens W_2)\xrightarrow{\cA_{W_1, V,W_2}} (W_1\tens V)\tens W_2\xrightarrow{\cR_{W_1,V}\tens 1_{W_2}} (V\tens W_1)\tens W_2\nonumber\\
 & \xrightarrow{\cA_{V,W_1,W_2}^{-1}} V\tens(W_1\tens W_2)\xrightarrow{g_1g_2\tens 1_{W_1\tens W_2}} V\tens(W_1\tens W_2)\xrightarrow{\cA_{V,W_1,W_2}} (V\tens W_1)\tens W_2\nonumber\\
 & \xrightarrow{\mu_{W_1}\tens 1_{W_2}} W_1\tens W_2\xrightarrow{I} W_3.
\end{align*}
We apply the naturality of the associativity isomorphisms to $g_1g_2\tens 1_{W_1\tens W_2}$ to cancel the associativity isomorphism and its inverse in the third line. Then since $W_1$ is $g_1$-twisted, we replace $\mu_{W_1}(g_1\tens 1_{W_1})$ with $\mu_{W_1}\cM_{V,W_1}^{-1}$ to get
\begin{align}\label{calc1}
 V\tens(W_1\tens & W_2)  \xrightarrow{\cA_{V,W_1,W_2}} (V\tens W_1)\tens W_2\xrightarrow{\cR_{V,W_1}\tens 1_{W_2}} (W_1\tens V)\tens W_2\xrightarrow{\cA_{W_1,V,W_2}^{-1}} W_1\tens(V\tens W_2)\nonumber\\
 &\xrightarrow{1_{W_1}\tens\cM_{V,W_2}} W_1\tens(V\tens W_2)\xrightarrow{\cA_{W_1, V,W_2}} (W_1\tens V)\tens W_2\xrightarrow{\cR_{W_1,V}\tens 1_{W_2}} (V\tens W_1)\tens W_2\nonumber\\
 &\xrightarrow{g_2\tens 1_{W_1\tens W_2}} (V\tens W_1)\tens W_2\xrightarrow{\cM_{V,W_1}^{-1}\tens 1_{W_2}} (V\tens W_1)\tens W_2\xrightarrow{\mu_{W_1}\tens 1_{W_2}} W_1\tens W_2\xrightarrow{I} W_3.
\end{align}
In the presence of the intertwining operator $I$, $\mu_{W_1}\tens 1_{W_2}$ can be replaced with
\begin{equation*}
 (V\tens W_1)\tens W_2\xrightarrow{\cR_{V,W_1}\tens 1_{W_2}} (W_1\tens V)\tens W_2\xrightarrow{\cA_{W_1,V,W_2}^{-1}} W_1\tens(V\tens W_2)\xrightarrow{1_{W_1}\tens\mu_{W_2}} W_1\tens W_2.
\end{equation*}
Insert this into \eqref{calc1}, apply naturality of the associativity and braiding isomorphisms to $g_2\tens 1_{W_1\tens W_2}$, and cancel to obtain
\begin{align*}
 V\tens(W_1\tens W_2) & \xrightarrow{\cA_{V,W_1,W_2}} (V\tens W_1)\tens W_2\xrightarrow{\cR_{V,W_1}\tens 1_{W_2}} (W_1\tens V)\tens W_2\xrightarrow{\cA_{W_1,V,W_2}^{-1}} W_1\tens(V\tens W_2)\nonumber\\
 &\xrightarrow{1_{W_1}\tens\cM_{V,W_2}} W_1\tens(V\tens W_2)\xrightarrow{1_{W_1}\tens(g_2\tens 1_{W_2})} W_1\tens(V\tens W_2)\xrightarrow{1_{W_1}\tens\mu_{W_2}} W_1\tens W_2\xrightarrow{I} W_3.
\end{align*}
Since $W_2$ is $g_2$-twisted this reduces to $I\mu_{W_1,W_2}^{(2)}$, which equals $\mu_{W_3}(1_V\tens I)$ because $I$ is an intertwining operator.
\end{proof}

Now for objects $W_1$ and $W_2$ in $\rep V$, the intertwining operator $I_{W_1,W_2}$ of type $\binom{W_1\tens_V W_2}{W_1\,W_2}$ is surjective because it is a cokernel morphism. Thus the preceding proposition immediately implies:
\begin{corol}\label{g1g2corol}
 If $W_1$ is a $g_1$-twisted $V$-module and $W_2$ is a $g_2$-twisted $V$-module, then $W_1\tens_V W_2$ is a $g_1 g_2$-twisted $V$-module. In particular, $\rep^G\,V$ is closed under tensor products.
\end{corol}

Note that the subcategory $\rep^1\,V$ is a monoidal supercategory, and it is braided \cite{Pa, KO, CKM} with braiding isomorphisms characterized by the commutative diagram
\begin{equation*}
 \xymatrixcolsep{4pc}
 \xymatrix{
 W_1\tens W_2 \ar[r]^{\cR_{W_1,W_2}} \ar[d]^{I_{W_1,W_2}} & W_2\tens W_1 \ar[d]^{I_{W_2,W_1}} \\
 W_1\tens_V W_2 \ar[r]^{\cR^V_{W_1,W_2}} & W_2\tens_V W_1 \\
 } .
\end{equation*}
The braiding isomorphisms in $\sC$ do not induce well-defined braiding isomorphisms on the entire category $\rep^G\,V$, but $\repGV$ does admit the structure of a braided
$G$-crossed supercategory, with braiding isomorphisms twisted by an action of $G$ on $\repGV$. We discuss this structure after presenting the definition of braided $G$-crossed supercategory.

An $\mathbb{F}$-additive supercategory $\sC$ decomposes as a direct sum of (not necessarily finitely many) full subcategories $\lbrace\sC_i\rbrace_{i\in I}$, denoted $\sC=\bigoplus_{i\in I} \sC_i$, if
\begin{enumerate}
 \item Every object in $\sC$ is isomorphic to a biproduct $\bigoplus_{i\in I} W_i$ with $W_i$ an object of $\sC_i$ and finitely many $W_i$ non-zero.
 \item If $W_i$ is an object of $\sC_i$ and $W_j$ is an object of $W_j$ for $i\neq j$, then $\mathrm{Hom}_{\sC}(W_i, W_j)=0$.
\end{enumerate}
The second condition implies that the only object in both $\sC_i$ and $\sC_j$ for $i\neq j$ is the zero object. Also, if an object $W$ is isomorphic to both $\bigoplus_{i\in I} W_i$ and $\bigoplus_{i\in I} \widetilde{W}_i$ with $W_i, \widetilde{W}_i$ objects in $\sC_i$, then $W_i\cong\widetilde{W}_i$ for each $i$.

For a (braided) monoidal supercategory $\sC$, let $\mathrm{Aut}^{(\mathrm{br})}_\tens(\sC)$ denote the group of equivalence classes of even (braided) tensor autoequivalences of $\sC$. Such an autoequivalence consists of a triple $(T, \tau, \varphi)$ where $T: \sC\rightarrow\sC$ is an equivalence of categories inducing even linear maps on morphisms, $\tau: T\circ\boxtimes\rightarrow\boxtimes\circ(T\times T)$ is an even natural isomorphism, and $\varphi: T(\vac)\rightarrow\vac$ is an even isomorphism. These isomorphisms must be suitably compatible with the unit, associativity, and braiding isomorphisms (if any) of $\sC$. The composition of two autoequivalences $(T_1,\tau_1,\varphi_1)$ and $(T_2,\tau_2,\varphi_2)$ is the functor $T_1\circ T_2$ together with the isomorphism
\begin{equation*}
 T_1(T_2(\vac))\xrightarrow{T_1(\varphi_2)} T_1(\vac)\xrightarrow{\varphi_1} \vac
\end{equation*}
and natural isomorphims
\begin{equation*}
 T_1(T_2(W_1\tens W_2))\xrightarrow{T_1((\tau_2)_{W_1,W_2})} T_1(T_2(W_1)\tens T_2(W_2))\xrightarrow{(\tau_1)_{T_2(W_1),T_2(W_2)}} T_1(T_2(W_1))\tens T_1(T_2(W_2))
\end{equation*}
for objects $W_1$ and $W_2$ of $\sC$.

Now for $G$ a (not necessarily finite) group, the following is a natural generalization of the notion of $G$-crossed category from \cite{TuraevGcrossed} (see also \cite{KirillovGequivariant, EGNO}) to the supercategory setting: 
\begin{defi}
 A \textit{braided $G$-crossed supercategory} over $\mathbb{F}$ is an $\mathbb{F}$-additive monoidal supercategory $\sC$ with the following structures:
 \begin{enumerate}
  \item $G$-grading: As a category $\sC$ decomposes as a direct sum
  \begin{equation*}
   \sC=\bigoplus_{g\in G} \sC_g
  \end{equation*}
where each $\sC_g$ is a full subcategory, called the \textit{$g$-twisted sector}. The $G$-grading is compatible with the monoidal structure in the sense that:
\begin{enumerate}
 \item The unit object $\vac$ is an object of $\sC_1$.
 \item For objects $W_1$ in $\sC_{g_1}$ and $W_2$ in $\sC_{g_2}$, $W_1\tens W_2$ is an object of $\sC_{g_1 g_2}$.
\end{enumerate}

\item $G$-action: There is a group homomorphism $\varphi: G\rightarrow\mathrm{Aut}_\tens(\sC)$, denoted $g\mapsto(T_g, \tau_g, \varphi_g)$, such that $g, h\in G$ and an object $W$ in $\sC_g$, $T_h(W)$ is an object of $\sC_{h g h^{-1}}$.

\item Braiding isomorphisms: For every $g\in G$, there is an even natural isomorphism $\cR$ from the functor $\tens$ on $\sC_g\times\sC$ to the functor $\tens\circ(T_g\times 1_{\sC_g})\circ\sigma$ satisfying the following properties:
\begin{enumerate}
 \item Compatibility with the $G$-action: For $g\in G$ and $W_1$ in $\sC_g$, the diagram
\begin{equation*}
 \xymatrixcolsep{6pc}
 \xymatrix{
 T_h(W_1\boxtimes W_2) \ar[r]^{T_h(\cR_{W_1,W_2})} \ar[d]^{\tau_{h; W_1, W_2}} & T_h(T_g(W_2)\boxtimes W_1) \ar[d]^{\tau_{h; T_g(W_2),W_1}} \\
 T_h(W_1)\boxtimes T_h(W_2) \ar[r]^{\cR_{T_h(W_1),T_h(W_2)}} & T_{hg}(W_2)\boxtimes T_h(W_1) \\
 }
\end{equation*}
commutes for all $h\in G$ and all objects $W_2$ in $\sC$.

\item The hexagon/heptagon axioms: First, for $g_1, g_2\in G$, $W_1$ in $\sC_{g_1}$, and $W_2$ in $\sC_{g_2}$, the diagram
\begin{equation*}
 \xymatrixcolsep{6pc}
 \xymatrix{
 W_1\tens(W_2\tens W_3) \ar[r]^{\cA_{W_1,W_2,W_3}} \ar[d]^{1_{W_1}\tens\cR_{W_2,W_3}} & (W_1\tens W_2)\tens W_3 \ar[d]^{\cR_{W_1\tens W_2, W_3}} \\
 W_1\tens(T_{g_2}(W_3)\tens W_2) \ar[d]^{\cA_{W_1,T_{g_2}(W_3),W_2}} & T_{g_1 g_2}(W_3)\tens(W_1\tens W_2) \ar[d]^{\cA_{T_{g_1,g_2}(W_3),W_1,W_2}} \\
 (W_1\tens T_{g_2}(W_3))\tens W_2 \ar[r]^{\cR_{W_1,T_{g_2}(W_3)}\tens 1_{W_3}} & (T_{g_1 g_2}(W_3)\tens W_1)\tens W_2 \\
 }
\end{equation*}
commutes for any object $W_3$ in $\sC$; and second, for $g\in G$ and $W_1$ in $\sC_g$, the diagram
\begin{equation*}
 \xymatrixcolsep{3.8pc}
 \xymatrix{
 (W_1\tens W_2)\tens W_3 \ar[r]^(.47){\cR_{W_1,W_2}\tens 1_{W_3}} \ar[d]^{\cA^{-1}_{W_1,W_2,W_3}} & (T_g(W_2)\tens W_1)\tens W_3 \ar[d]^{\cA^{-1}_{T_g(W_2),W_1,W_3}} & \\
 W_1\tens(W_2\tens W_3) \ar[d]^{\cR_{W_1,W_2\tens W_3}} & T_g(W_2)\tens(W_1\tens W_3) \ar[rd]^(.55){1_{T_g(W_2)}\tens\cR_{W_1,W_3}} & \\
 T_g(W_2\tens W_3)\tens W_1 \ar[r]_(.47){\tau_{g; W_2, W_3}\tens 1_{W_1}} & (T_g(W_2)\tens T_g(W_3))\tens W_1 \ar[r]_(0.53){\cA^{-1}_{T_g(W_2),T_g(W_3),W_1}} & T_g(W_2)\tens(T_g(W_3)\tens W_1) \\
 }
\end{equation*}
commutes for all objects $W_2$, $W_3$ in $\sC$.
\end{enumerate}
 \end{enumerate}
\end{defi}

\begin{rema}
In the axioms for the braiding isomorphisms, we have implicitly assumed the homomorphism $\varphi$ is strict in the sense that $\varphi(g_1 g_2) =\varphi(g_1)\varphi(g_2)$. More generally, one could require that $\varphi(g_1)\varphi(g_2)$ and $\varphi(g_1 g_2)$ be naturally isomorphic via an isomorphism with suitable coherence properties, as in \cite{KirillovGequivariant}. One could also impose stronger strictness conditions: in \cite{TuraevGcrossed}, for example, it is assumed that $G$ acts by strict tensor functors, that is, $T_g(W_1\tens W_2)=T_g(W_1)\tens T_g(W_2)$, $T_g(\vac)=\vac$, and $\tau_g$, $\varphi_g$ are identity isomorphisms for all $g\in G$. Here, we have chosen the level of strictness that actually occurs in the examples that we will consider.
\end{rema}

\begin{rema}
 The only modification in the notion of braided $G$-crossed category needed for the supercategory setting is the evenness requirement for $\cR$ and each $T_g$. The naturality of the braiding isomorphism $\cR$ means that for parity homogeneous morphisms $f_1: W_1\rightarrow\widetilde{W}_1$ and $f_2: W_2\rightarrow\widetilde{W}_2$, where $W_1$, $\widetilde{W}_1$ are objects of $\sC_g$,
 \begin{equation*}
  \cR_{\widetilde{W}_1,\widetilde{W}_2}(f_1\tens f_2)=(-1)^{\vert f_1\vert\vert f_2\vert} (T_g(f_2)\tens f_1)\cR_{W_1,W_2}.
 \end{equation*}
As $T_g$ induces even linear maps on morphisms, there is no question of whether $\vert f_2\vert$ or $\vert T_g(f_2)\vert$ should appear in the sign factor here.
\end{rema}

\begin{rema}
 If a braided $G$-crossed supercategory $\sC$ is rigid and $W$ is an object of $\sC_g$, then its dual $W^*$ is an object of $\sC_{g^{-1}}$. Indeed, if $W^*=\bigoplus_{h\in G} W^*_h$, then the restriction of the evaluation $e_W: W^*\tens W\rightarrow\vac$ to $W^*_h\tens W$ is zero unless $h=g^{-1}$. Similarly, the image of the coevaluation $i_W: \vac\rightarrow W\tens W^*$ is contained in $W\tens W^*_{g^{-1}}$, and we find that $(W^*_{g^{-1}}, e_W\vert_{W^*_{g^{-1}}\tens W}, i_W)$ is already a (left) dual of $W$.
\end{rema}


Under mild conditions, the category $\rep^G\,V$ of twisted modules for a superalgebra $V$ in a braided tensor category $\cC$ is a braided $G$-crossed supercategory. This result was stated in \cite[Theorem 4.2 (2)]{KirillovGequivariant}, although a detailed proof was not given.
As a full proof seems to be missing from the literature, we will give one in Appendix \ref{app:BGC_proof}, here only discussing the definitions of the $G$-action and the $G$-crossed braiding.
\begin{theo}\label{Gcrossedfromtwist}
 Let $\cC$ be a braided tensor category with right exact tensoring functors, $V$ a superalgebra in $\cC$, and $G$ an automorphism group of $V$. If $\hom_{\rep V}(W_1,W_2)=0$ whenever $W_1$ is $g_1$-twisted, $W_2$ is $g_2$-twisted, and $g_1\neq g_2$, then $\repGV$ is a braided $G$-crossed supercategory.
\end{theo}

The condition on $\hom_{\rep V}(W_1, W_2)$ guarantees that $\repGV$ decomposes as a direct sum
\begin{equation*}
 \repGV=\bigoplus_{g\in G} \repgV.
\end{equation*}
For $g\in G$, the superfunctor $T_g: \rep V\rightarrow\rep V$ is defined as follows:
 \begin{itemize}
  \item For an object $(W,\mu_W)$ in $\rep V$, $T_g(W,\mu_W)=(W, \mu_W(g^{-1}\tens 1_W))$.
  \item For a morphism $f: W_1\rightarrow W_2$ in $\rep V$, $T_g(f)=f$.
 \end{itemize}
After showing that $T_g$ sends $\repgV$ to $\rep^{ghg^{-1}}\,V$, we see that $T_g$ restricts to a superfunctor on $\repGV$. The isomorphism
\begin{equation*}
 \varphi_g: T_g(V)\rightarrow V
\end{equation*}
is $g$ itself. Then for objects $W_1$, $W_2$ in $\rep\,V$, the even natural isomorphism 
\begin{equation*}
 \tau_{g; W_1,W_2}: T_g(W_1\tens_V W_2)\rightarrow T_g(W_1)\tens_V T_g(W_2),
\end{equation*}
which as a morphism in $\sC$ is an isomorphism from $W_1\tens_V W_2$ to $T_g(W_1)\tens_V T_g(W_2)$, is characterized by the commutative diagram
\begin{equation*}
 \xymatrixcolsep{4pc}
 \xymatrix{
 W_1\tens W_2 \ar[d]^{I_{W_1,W_2}} \ar[rd]^{I_{T_g(W_1),T_g(W_2)}} & \\
 W_1\tens_V W_2 \ar[r]^(.4){\tau_{g; W_1, W_2}} & T_g(W_1)\tens_V T_g(W_2) \\
 } .
\end{equation*}
Note that $I_{T_g(W_1),T_g(W_2)}$ does not equal $I_{W_1,W_2}$ since it is the cokernel of a different $\sC$-morphism into $W_1\tens W_2$, defined using different actions of $V$ on $W_1\tens W_2$. Finally for $W_1$ a $g$-twisted module and $W_2$ any object of $\repGV$, the braiding isomorphism
\begin{equation*}
 \cR_{W_1,W_2}^V: W_1\tens_V W_2\rightarrow T_g(W_2)\tens_V W_1
\end{equation*}
is characterized by the commutative diagram
\begin{equation*}
 \xymatrixcolsep{4pc}
 \xymatrix{
 W_1\tens W_2 \ar[r]^{\cR_{W_1,W_2}} \ar[d]^{I_{W_1,W_2}} & W_2\tens W_1 \ar[d]^{I_{T_g(W_2),W_1}} \\
 W_1\tens_V W_2 \ar[r]^{\cR^V_{W_1,W_2}} & T_g(W_2)\tens_V W_1 \\
 } .
\end{equation*}

\subsection{\texorpdfstring{$G$}{G}-equivariantization}\label{sec:Gequiv}

Given a braided $G$-crossed supercategory $\sC$ with $G$-action $g\mapsto(T_g,\tau_g,\varphi_g)$ and braiding $\cR$, there is a braided monoidal supercategory $\sC^G$ called the $G$-equivariantization of $\sC$ with objects arising from $G$-invariant objects of $\sC$ (see for example Sections 2.7, 4.15 and 8.24 of \cite{EGNO}). Formally,
\begin{itemize}
 \item The $\mathbb{F}$-additive supercategory $\sC^G$ has objects $(W,\lbrace\varphi_W(g)\rbrace_{g\in G})$ where $W$ is an object of $\sC$ and the $\varphi_W(g): T_g(W)\rightarrow W$ are even isomorphisms in $\sC$ such that the diagram
 \begin{equation*}
 \xymatrixcolsep{4pc}
 \xymatrix{
  T_{gh}(W)=T_g(T_h(W)) \ar[r]^(.65){T_g(\varphi_W(h))} \ar[rd]_{\varphi_W(gh)} & T_g(W) \ar[d]^{\varphi_W(g)} \\
   & W
  }
 \end{equation*}
commutes for $g,h\in G$.

\item Morphisms $f: (W_1,\varphi_{W_1})\rightarrow (W_2,\varphi_{W_2})$ in $\sC^G$ are morphisms $f: W_1\rightarrow W_2$ in $\sC$ such that the diagram
\begin{equation*}
 \xymatrixcolsep{4pc}
 \xymatrix{
 T_g(W_1) \ar[r]^{T_g(f)} \ar[d]^{\varphi_{W_1}(g)} & T_g(W_2) \ar[d]^{\varphi_{W_2}(g)} \\
 W_1 \ar[r]^{f} & W_2
 }
\end{equation*}
commutes for all $g\in G$.
\end{itemize}
For objects $(W_1,\varphi_{W_1})$ and $(W_2,\varphi_{W_2})$ in $\sC^G$, their tensor product is $(W_1\tens W_2,\varphi_{W_1\tens W_2})$ where
\begin{equation*}
 \varphi_{W_1\tens W_2}(g)=(\varphi_{W_1}(g)\tens\varphi_{W_2}(g))\tau_{g;W_1,W_2}
\end{equation*}
for $g\in G$. Then the tensor product (in $\sC$) of two morphisms in $\sC^G$ is also a morphism in $\sC^G$ due to the naturality of $\tau_g$. The unit object of $\sC^G$ is $(\vac,\lbrace\varphi_g\rbrace_{g\in G})$, and the unit and associativity isomorphisms of $\sC$ are morphisms in $\sC^G$ due to the compatibility of the $\varphi_g$ and $\tau_g$ with the unit and associativity isomorphisms.

We can also define a braiding on $\sC^G$ as follows. For an object $W$ in $\sC$, let $\pi_g$ denote projection onto the $g$-graded homogeneous summand $W^g$ and let $q_g$ denote the inclusion of $W^g$ into $W$. Then for objects $(W_1,\varphi_{W_1})$ and $(W_2,\varphi_{W_2})$ of $\sC^G$, we define $\widetilde{\cR}_{W_1,W_2}$ to be the sum over $g\in G$ of the compositions
\begin{equation*}
 W_1\tens W_2\xrightarrow{\pi_g\tens 1_{W_2}} W_1^g\tens W_2\xrightarrow{\cR_{W_1^g,W_2}} T_g(W_2)\tens W_1^g\xrightarrow{\varphi_{W_2}(g)\tens q_g} W_2\tens W_1.
\end{equation*}
Showing that $\widetilde{\cR}_{W_1,W_2}$ is a morphism in $\sC^G$ requires the compatibility of $\cR$ with the $G$-action on $\sC$, and the hexagon axioms for $\widetilde{\cR}$ follow using the hexagon/heptagon axioms for $\cR$.

When our braided $G$-crossed supercategory is the category of twisted modules for a superalgebra $V$ in a braided tensor category $\cC$, an object of the $G$-equivariantization is an object $(W,\mu_W)$ of $\repGV$ equipped with a representation $\varphi_W: G\rightarrow\mathrm{Aut}_{\sC}(W)$ such that
\begin{equation*}
 \varphi_W(g)\mu_W=\mu_W(g\tens\varphi_W(g))
\end{equation*}
for all $g\in G$. Morphisms $f: W_1\rightarrow W_2$ in the $G$-equivariantization are morphisms in $\repGV$ that commute with the representations of $G$ on $W_1$ and $W_2$.

For the rest of this section, we will assume that $\repGV$ equals the full category $\rep\,V$; for conditions guaranteeing this occurs, see Assumption \ref{mainassum} in the next section. In this case, the induction functor
 $\cF: \sC\rightarrow\rep\,V$ is actually a functor into the $G$-equivariantization, which we will denote by $\mathcal{S}(\rep\,V)^G$ because we will soon use the notation $(\rep\,V)^G$ for a certain subcategory. Indeed, for an object $W$ in $\sC$, $G$ acts on $\cF(W)=V\tens W$ by $\varphi_{\cF(W)}(g)=g\tens 1_W$. This representation satisfies
\begin{equation*}
 \varphi_{\cF(W)}(g)\mu_{\cF(W)}=\mu_{\cF(W)}(g\tens\varphi_{\cF(W)}(g))
\end{equation*}
because $g$ is an automorphism of $V$. Moreover, if $f: W_1\rightarrow W_2$ is a morphism in $\sC$, then $\cF(f)=1_V\tens f$ is a morphism in $\mathcal{S}(\rep\,V)^G$ because
\begin{equation*}
 (g\tens 1_{W_2})(1_V\tens f)=(1_V\tens f)(g\tens 1_{W_1})
\end{equation*}
(since $g$ is even, there is no sign factor). 

The following theorem can be found in \cite{KirillovOrbifoldII, Mu2}, but we include the proof to emphasize that it does not require rigidity or semisimplicity:
\begin{theo}\label{thm:F_braided}
 If $\repGV=\rep\,V$, then induction $\cF: \sC\rightarrow\mathcal{S}(\rep\,V)^G$ is a braided monoidal superfunctor.
\end{theo}
\begin{proof}
 We need to check that $r_V:\cF(\vac)\rightarrow V$ and $f_{W_1,W_2}: \cF(W_1\tens W_2)\rightarrow\cF(W_1)\tens_V\cF(W_2)$ are morphisms in $\mathcal{S}(\rep\,V)^G$. These isomorphisms will be compatible with the unit and associativity isomorphisms in $\sC$ and $\mathcal{S}(\rep\,V)^G$ because the unit and associativity isomorphisms in $\mathcal{S}(\rep\,V)^G$ are the same as those in $\rep\,V$.
 
 The naturality of the right unit isomorphisms implies $r_V$ is an $\mathcal{S}(\rep\,V)^G$-morphism. For $f_{W_1,W_2}$, we need to show
 \begin{equation*}
  \varphi_{\cF(W_1)\tens_V\cF(W_2)}(g) f_{W_1,W_2}=f_{W_1,W_2}\varphi_{\cF(W_1\tens W_2)}(g)
 \end{equation*}
for $g\in G$. The left side is the composition
\begin{align}\label{Fmonoidal_calc}
 V\tens(W_1\tens & W_2)\xrightarrow{\cA_{V,W_1,W_2}} (V\tens W_1)\tens W_2\xrightarrow{1_{V\tens W_1}\tens(\iota_V\tens 1_{W_2})l_{W_2}^{-1}} (V\tens W_1)\tens(V\tens W_2) \nonumber\\
 & \xrightarrow{I_{\cF(W_1),\cF(W_2)}} (V\tens W_1)\tens_V (V\tens W_2)\xrightarrow{\tau_{g;\cF(W_1),\cF(W_2)}} T_g(V\tens W_1)\tens_V T_g(V\tens W_2)\nonumber\\
 &\xrightarrow{\varphi_{\cF(W_1)}(g)\tens_V\varphi_{\cF(W_2)}(g)} (V\tens W_1)\tens_V(V\tens W_2).
\end{align}
Using the definitions,
\begin{align*}
 \big(\varphi_{\cF(W_1)}(g)\tens_V\varphi_{\cF(W_2)}(g)\big) & \tau_{g;\cF(W_1),\cF(W_2)} I_{\cF(W_1),\cF(W_2)}\nonumber\\
 &= \big(\varphi_{\cF(W_1)}(g)\tens_V\varphi_{\cF(W_2)}(g)\big) I_{T_g(\cF(W_1)),T_g(\cF(W_2))}\nonumber\\
& = I_{\cF(W_1),\cF(W_2)}\big(\varphi_{\cF(W_1)}(g)\tens\varphi_{\cF(W_2)}(g)\big)\nonumber\\
& = I_{\cF(W_1),\cF(W_2)}\big((g\tens 1_{W_1})\tens(g\tens 1_{W_2})\big).
\end{align*}
Inserting this back into \eqref{Fmonoidal_calc}, using $g\iota_V=\iota_V$, and applying naturality of associativity,
we get
\begin{align*}
 V\tens(W_1\tens & W_2)\xrightarrow{g\tens 1_{W_1\tens W_2}} V\tens(W_1\tens W_2)\xrightarrow{\cA_{V,W_1,W_2}} (V\tens W_1)\tens W_2\nonumber\\
 & \xrightarrow{1_{V\tens W_1}\tens(\iota_V\otimes 1_{W_2})l_{W_2}^{-1}} (V\tens W_1)\tens(V\tens W_2)\xrightarrow{I_{\cF(W_1),\cF(W_2)}} (V\tens W_1)\tens_V(V\tens W_2),
\end{align*}
which is $f_{W_1,W_2}\varphi_{\cF(W_1\tens W_2)}(g)$.

We also need to verify that the natural isomorphism $f$ is compatible with the braiding isomorphisms $\cR$ in $\sC$ and $\widetilde{R}^V$ in $\mathcal{S}(\rep\,V)^G$ in the sense that the diagram
\begin{equation*}
 \xymatrixcolsep{4pc}
 \xymatrix{
 \cF(W_1\tens W_2) \ar[r]^{\cF(\cR_{W_1,W_2})} \ar[d]^{f_{W_1,W_2}} & \cF(W_2\tens W_1) \ar[d]^{f_{W_2,W_1}} \\
 \cF(W_1)\tens_V\cF(W_2) \ar[r]^{\widetilde{R}^V_{\cF(W_1),\cF(W_2)}} & \cF(W_2)\tens_V\cF(W_1) \\
 }
\end{equation*}
commutes. The lower left composition here is the sum over $g\in G$ of the compositions
\begin{align*}
 V\tens(W_1 & \tens W_2)\xrightarrow{\cA_{V,W_1,W_2}} (V\tens W_1)\tens W_2 \xrightarrow{1_{V\tens W_1}\tens(\iota_V\tens 1_{W_2})l_{W_2}^{-1}} (V\tens W_1)\tens (V\tens W_2)\nonumber\\
 &\xrightarrow{I_{\cF(W_1),\cF(W_2)}} (V\tens W_1)\tens_V(V\tens W_2)\xrightarrow{\pi_g\tens_V 1_{V\tens W_2}} (V\tens W_1)^g\tens_V(V\tens W_2)\nonumber\\
 & \xrightarrow{\cR^V_{(V\tens W_1)^g,V\tens W_2}} T_g(V\tens W_2)\tens_V (V\tens W_1)^g\xrightarrow{\varphi_{\cF(W_2)}(g)\tens_V q_g} (V\tens W_2)\tens_V(V\tens W_1).
\end{align*}
By the definitions of the tensor product of morphisms in $\rep\,V$, the braiding $\cR^V$, and $\varphi_{\cF(W_2)}(g)$, this is the sum over $g\in G$ of
\begin{align*}
V\tens(W_1 & \tens W_2)\xrightarrow{\cA_{V,W_1,W_2}} (V\tens W_1)\tens W_2 \xrightarrow{1_{V\tens W_1}\tens(\iota_V\tens 1_{W_2})l_{W_2}^{-1}} (V\tens W_1)\tens (V\tens W_2)\nonumber\\
&\xrightarrow{\pi_g\tens 1_{V\tens W_2}} (V\tens W_1)^g\tens(V\tens W_2)\xrightarrow{\cR_{(V\tens W_1)^g,V\tens W_2}} (V\tens W_2)\tens(V\tens W_1)^g\nonumber\\
& \xrightarrow{(g\tens 1_{W_2})\tens q_g} (V\tens W_2)\tens(V\tens W_1)\xrightarrow{I_{\cF(W_2),\cF(W_1)}} (V\tens W_2)\tens_V(V\tens W_1).
\end{align*}
We apply naturality of the braiding to $g$ and use $g\iota_V=\iota_V$ to eliminate $g$. We then apply naturality of the braiding to $q_g$ and get $\sum_{g\in G} q_g \pi_g = 1_{V\tens W_1}$. Thus everything simplifies to
\begin{align*}
 V\tens(W_1 \tens W_2)\xrightarrow{\cA_{V,W_1,W_2}} & (V\tens W_1)\tens W_2 \xrightarrow{1_{V\tens W_1}\tens(\iota_V\tens 1_{W_2})l_{W_2}^{-1}} (V\tens W_1)\tens (V\tens W_2)\nonumber\\
&\xrightarrow{\cR_{V\tens W_1,V\tens W_2}} (V\tens W_2)\tens(V\tens W_1)\xrightarrow{I_{\cF(W_2),\cF(W_1)}} (V\tens W_2)\tens_V(V\tens W_1).
\end{align*}
Now use the hexagon axiom and the unit property of $V$ to rewrite as
\begin{align}\label{Fbraided_calc}
 V & \tens(W_1 \tens W_2)\xrightarrow{\cA_{V,W_1,W_2}} (V\tens W_1)\tens W_2 \xrightarrow{1_{V\tens W_1}\tens(\iota_V\tens 1_{W_2})l_{W_2}^{-1}} (V\tens W_1)\tens (V\tens W_2)\nonumber\\
 & \xrightarrow{\cA_{V,W_1,V\tens W_2}^{-1}} V\tens(W_1\tens(V\tens W_2)) \xrightarrow{1_V\tens\cR_{W_1,V\tens W_2}} V\tens((V\tens W_2)\tens W_1) \xrightarrow{\cA_{V,V\tens W_2,W_1}} (V\tens(V\tens W_2))\tens W_1\nonumber\\
 & \xrightarrow{\cR_{V,V\tens W_2}\tens 1_{W_1}} ((V\tens W_2)\tens V)\tens W_1 \xrightarrow{\cA_{V\tens W_2,V,W_1}^{-1}} (V\tens W_2)\tens(V\tens W_1)\nonumber\\
& \xrightarrow{1_{V\tens W_2}\tens(r_V^{-1}\tens 1_{W_1})} (V\tens W_2)\tens((V\tens\vac)\tens W_1) \xrightarrow{1_{V\tens W_2}\tens((1_V\tens\iota_V)\tens 1_{W_1})} (V\tens W_2)\tens((V\tens V)\tens W_1)\nonumber\\
&\xrightarrow{1_{V\tens W_2}\tens(\mu_V\tens 1_{W_1})} (V\tens W_2)\tens(V\tens W_1) \xrightarrow{I_{\cF(W_2),\cF(W_1)}} (V\tens W_2)\tens_V(V\tens W_1).
\end{align}
Next use the triangle axiom and naturality of associativity to calculate
\begin{align*}
 (\mu_V\tens 1_{W_1})&\big((1_V\tens\iota_V)\tens 1_{W_1}\big)(r_V^{-1}\tens 1_{W_1})\nonumber\\
 & = (\mu_V\tens 1_{W_1})\big((1_V\tens\iota_V)\tens 1_{W_1}\big)\cA_{V,\vac,W_1}(1_V\tens l_{W_1}^{-1})\nonumber\\
 & =(\mu_V\tens 1_{W_1})\cA_{V,V,W_1}\big(1_V\tens(\iota_V\tens 1_{W_1})\big)(1_V\tens l_{W_1}^{-1})\nonumber\\
 & = \mu_{\cF(W_1)}\big(1_V\tens(\iota_V\tens 1_{W_1}) l_{W_1}^{-1}\big).
\end{align*}
We insert this back into \eqref{Fbraided_calc} and use properties of natural isomorphisms to get
\begin{align*}
 V & \tens(W_1\tens W_2)\xrightarrow{1_V\tens(1_{W_1}\tens(\iota_V\tens 1_{W_2})l_{W_2}^{-1})} V\tens(W_1\tens(V\tens W_2))\xrightarrow{1_V\tens\cR_{W_1,V\tens W_2}} V\tens((V\tens W_2)\tens W_1)\nonumber\\
 & \xrightarrow{1_V\tens(1_{V\tens W_2}\tens(\iota_V\tens 1_{W_1})l_{W_1}^{-1})} V\tens((V\tens W_2)\tens(V\tens W_1)) \xrightarrow{\cA_{V,\cF(W_2),\cF(W_1)}} (V\tens(V\tens W_2))\tens(V\tens W_1)\nonumber\\
 & \xrightarrow{\cR_{V,\cF(W_2)}\tens 1_{\cF(W_1)}} ((V\tens W_2)\tens V)\tens(V\tens W_1)\xrightarrow{\cA_{\cF(W_2),V,\cF(W_1)}^{-1}} (V\tens W_2)\tens(V\tens(V\tens W_1))\nonumber\\
 &\xrightarrow{1_{\cF(W_2)}\tens\mu_{\cF(W_1)}} (V\tens W_2)\tens(V\tens W_1)\xrightarrow{I_{\cF(W_2),\cF(W_1)}} (V\tens W_2)\tens_V(V\tens W_1).
\end{align*}
Since $I_{\cF(W_2),\cF(W_1)}$ is an intertwining, we can replace the fifth through seventh arrows above with $\mu_{\cF(W_2)}\tens 1_{\cF(W_1)}$. Then applying naturality of braiding and associativity to $(\iota_V\tens 1_{W_2})l_{W_2}^{-1}$, we get
\begin{align*}
 V & \tens(W_1\tens W_2)\xrightarrow{1_V\tens\cR_{W_1,W_2}} V\tens(W_2\tens W_1) \xrightarrow{1_V\tens(1_{W_2}\tens(\iota_V\tens 1_{W_1})l_{W_1}^{-1})} V\tens(W_2\tens(V\tens W_1))\nonumber\\
 & \xrightarrow{\cA_{V,W_2,V\tens W_1}} (V\tens W_2)\tens(V\tens W_1) \xrightarrow{(1_V\tens(\iota_V\tens 1_{W_2})l_{W_2}^{-1})\tens 1_{V\tens W_1}} (V\tens(V\tens W_2))\tens(V\tens W_1)\nonumber\\
 & \xrightarrow{\cA_{V,V,W_2}\tens 1_{V\tens W_1}} ((V\tens V)\tens W_2)\tens(V\tens W_1) \xrightarrow{(\mu_V\tens 1_{W_2})\tens 1_{V\tens W_1}} (V\tens W_2)\tens(V\tens W_1)\nonumber\\
 & \xrightarrow{I_{\cF(W_2),\cF(W_1)}} (V\tens W_2)\tens_V(V\tens W_1).
\end{align*}
Finally, we use naturality of associativity, the triangle axiom, and the right unit property of $V$ to conclude
\begin{align*}
 (\mu_V\tens 1_{W_2}) & \cA_{V,V,W_2} \big(1_V\tens(\iota_V\tens 1_{W_2})\big)(1_V\tens l_{W_2}^{-1})\nonumber\\
 &= (\mu_V\tens 1_{W_2})\big((1_V\tens\iota_V)\tens 1_{W_2}\big)\cA_{V,\vac,W_2}(1_V\tens l_{W_2}^{-1})\nonumber\\
&= (\mu_V\tens 1_{W_2})\big((1_V\tens\iota_V)\tens 1_{W_2}\big)(r_V^{-1}\tens 1_{W_2}) = 1_{V\tens W_2}.
\end{align*}
This together with naturality of associativity yields the composition
\begin{align*}
 V & \tens(W_1\tens W_2)\xrightarrow{1_V\tens\cR_{W_1,W_2}} V\tens(W_2\tens W_1)\xrightarrow{\cA_{V,W_2,W_1}} (V\tens W_2)\tens W_1\nonumber\\
 &\xrightarrow{1_{V\tens W_2}\tens(\iota_V\tens 1_{W_1})l_{W_1}^{-1}} (V\tens W_2)\tens(V\tens W_1) \xrightarrow{I_{\cF(W_2),\cF(W_1)}} (V\tens W_2)\tens_V(V\tens W_1),
\end{align*}
which is $f_{W_2,W_1}\cF(\cR_{W_1,W_2})$ as required.
\end{proof}

Since we mainly want to understand the original braided tensor category $\cC$ rather than the auxiliary supercategory $\sC$, we would like induction to be a functor from $\cC$ (embedded into $\sC$ via $W\mapsto(W,0)$) into a suitable braided tensor subcategory of $\mathcal{S}(\rep\,V)^G$. For this, we need $G$ to include the parity automorphism $P_V=1_{V^\even}\oplus(-1_{V^\odd})$ of $V$. In this case, define $(\rep\,V)^{G}$ to be the full subcategory of $\mathcal{S}(\rep\,V)^G$ whose objects $(W,\mu_W,\varphi_W)$ satisfy
\begin{equation*}
 \varphi_W(P_V)=P_W.
\end{equation*}
The category $(\rep\,V)^{G}$ is not a supercategory in any meaningful sense because its morphisms $f: W_1\rightarrow W_2$ satisfy
\begin{equation*}
 P_{W_2} f= f P_{W_1}
\end{equation*}
and hence are all even. Also, induction sends $\cC$ to $(\rep\,V)^{G}$ because if $W$ is an object of $\cC$, then
$ \cF(W)=(V^\even\tens W, V^\odd\tens W)$
as an object of $\sC$, and hence
\begin{equation*}
 P_{\cF(W)} = P_V\tens 1_W =\varphi_{\cF(W)}(P_V).
\end{equation*}
Now we have:
\begin{theo}\label{thm:F_even_braided}
Assume $G$ contains $P_V$ and $\repGV=\rep\,V$. Then $(\rep\,V)^{G}$ is a braided tensor category and induction $\cF: \cC\rightarrow(\rep\,V)^{G}$ is a braided tensor functor.
\end{theo}
\begin{proof}
 To show that $(\rep\,V)^{G}$ is a braided monoidal subcategory of $\mathcal{S}(\rep\,V)^G$, we just need to show that it is closed under tensor products. Thus we show that if $\varphi_{W_1}(P_V)=P_{W_1}$ and $\varphi_{W_2}(P_V)=P_{W_2}$ for objects $W_1$, $W_2$ in $\mathcal{S}(\rep\,V)^G$, then $\varphi_{W_1\tens_V W_2}(P_V)=P_{W_1\tens_V W_2}$ as well. Using definitions,
 \begin{align*}
  \varphi_{W_1\tens_V W_2}(P_V) I_{W_1,W_2} & = (\varphi_{W_1}(P_V)\tens_V\varphi_{W_2}(P_V)) \tau_{P_V; W_1,W_2} I_{W_1,W_2}\nonumber\\
  & = (P_{W_1}\fus{V} P_{W_2}) I_{T_{P_V}(W_1), T_{P_V}(W_2)}\nonumber\\
  & = I_{W_1,W_2}(P_{W_1}\tens P_{W_2})\nonumber\\
  & = I_{W_1,W_2} P_{W_1\tens W_2}\nonumber\\
  & = P_{W_1\fus{V} W_2} I_{W_1,W_2},
 \end{align*}
where the last step uses the evenness of $I_{W_1,W_2}$. Since $I_{W_1,W_2}$ is surjective, $\varphi_{W_1\tens_V W_2}(P_V)=P_{W_1\tens_V W_2}$.

The proof that $(\rep\,V)^{G}$ is abelian, and thus a braided tensor category, is similar to the proof of \cite[Theorem 2.9]{CKM}, so we just indicate how to show $(\rep\,V)^{G}$ is closed under cokernels and why epimorphisms in $(\rep\,V)^G$ are cokernels of their kernels. A morphism $f: W_1\rightarrow W_2$ in $(\rep\,V)^{G}$ is in particular an even morphism in $\rep\,V$, so \cite[Proposition 2.32]{CKM} shows $f$ has a cokernel $(C,\mu_C)$ in $\rep\,V$ with even cokernel morphism $c: W_2\rightarrow C$. Then for $g\in G$, define $\varphi_C(g): C\rightarrow C$ to be the unique $\sC$-morphism such that the diagram
\begin{equation*}
 \xymatrixcolsep{4pc}
 \xymatrix{
 W_1 \ar[r]^{f} & W_2 \ar[r]^{\varphi_{W_2}(g)} \ar[d]^{c} & W_2 \ar[d]^{c} \\
  & C \ar[r]^{\varphi_{C}(g)} & C \\
 }
\end{equation*}
commutes; $\varphi_C(g)$ exists because $f$ is a morphism in $(\rep\,V)^{G}$:
\begin{equation*}
 c\varphi_{W_2}(g) f =c f\varphi_{W_1}(g)= 0.
\end{equation*}
To show that $\varphi_C(g)$ is compatible with $\mu_C$ and that $\varphi_C$ is a representation of $G$, one uses the corresponding properties of $\varphi_{W_2}$ and the surjectivity of $c$ and $1_V\tens c$. Showing that $(C,\mu_C,\varphi_C)$ is a cokernel of $f$ in $(\rep\,V)^{G}$ uses the cokernel property $(C,\mu_C)$ in $\rep\,V$, the definition of $\varphi_C$, and the surjectivity of $c$.


Now suppose $f: W_1\twoheadrightarrow W_2$ is an epimorphism in $(\rep\,V)^G$. We claim that $f$ is also an epimorphism in $\rep\,V$. Indeed, for $h: W_2\rightarrow X$ a morphism in $\rep\,V$ such that $hf=0$ and $(C,c)$ a cokernel of $f$ in $\rep\,V$ with $c: W_2\rightarrow C$ even, there is a unique $\widetilde{h}: C\rightarrow X$ such that $h=\widetilde{h}c$. But we have seen that $C$ has a unique structure of $(\rep\,V)^{G}$-object for which $c$ is a morphism in $(\rep\,V)^{G}$. So $cf=0$ implies $c=0$ as $f$
is an epimorphism in $(\rep\,V)^{G}$. Then $h=\widetilde{h}c =0$ as well, showing $f$ is an epimorphism in $\rep\,V$.

Now that $f$ is an even epimorphism in $\rep\,V$, \cite[Proposition 2.32]{CKM} shows that $f$ is the cokernel of its kernel morphism $k: (K,\mu_K)\rightarrow(W_1,\mu_{W_1})$ in $\rep\,V$. But $k$ is also a morphism in $(\rep\,V)^{G}$, and then one shows that $(W_2,f)$ satisfies the universal property of the cokernel of $k$ in $(\rep\,V)^G$ by applying the cokernel property of $(W_2,f)$ in $\rep\,V$, the fact that $f$ is a morphism in $(\rep\,V)^G$, and the surjectivity of $f$ in $\rep\,V$.
%
%
%

The assertion that induction is a braided tensor functor from $\cC$ to $(\rep\,V)^{G}$ is immediate from the discussion preceding the theorem and Theorem \ref{thm:F_braided}.
\end{proof}

\begin{rema}
 From now on we will slightly abuse terminology and refer to $(\rep\,V)^G$ as the $G$-equivariantization of $\rep\,V$.
\end{rema}

\section{The main categorical theorem}\label{sec:MainCatThm}

We continue to fix an (abelian) $\mathbb{F}$-linear braided tensor category $\cC$, a superalgebra $V$ in $\cC$, and an automorphism group $G$ of $V$. In the preceding section, we saw that if $\rep\,V=\rep^G\,V$, that is, all objects of $\rep\,V$ are direct sums of $g$-twisted $V$-modules for $g\in G$, then $\rep\,V$ is a braided $G$-crossed supercategory and induction is a braided tensor functor from $\cC$ to the $G$-equivariantization $(\rep\,V)^G$. In \cite{KirillovOrbifoldI, KirillovOrbifoldII, Mu2}, it was shown that $\rep\,V=\repGV$ under the assumptions that $\cC$ is rigid and semisimple and that the $G$-invariants of $V$ equal $\vac$. Now, we prove the same result without semisimplicity and using rigidity only for $V$. Thus, our result will apply to non-semisimple module categories for vertex operator algebras arising in logarithmic conformal field theory, many of which are not known to be rigid. The following conditions will be in force for the rest of the section:
\begin{assum}\label{mainassum}
 The superalgebra $(V,\mu_V,\iota_V)$ and automorphism group $G$ satisfy:
 \begin{itemize}
 \item $G$ is finite and includes the parity involution $P_V = 1_{V^\even}\oplus (-1_{V^\odd})$, so that $\vert G\vert\in 2\ZZ$.
 
 \item The order of $G$ is invertible in $\mathbb{F}$, so that in particular the characteristic of $\mathbb{F}$ is not $2$.
 
 \item $V$ is \textit{haploid} in the sense that $\hom_\sC(\vac, V)^\even=\mathbb{F}\iota_V$. 
  
 \item There is an even morphism $\varepsilon_V: V\rightarrow\vac$ in $\sC$ such that $\varepsilon_V \iota_V=1_\vac$ and $\iota_V \varepsilon_V=\frac{1}{\vert G\vert}\sum_{g\in G} g$.
  \item There is an even morphism $\widetilde{i}_V:\vac\rightarrow V\tens V$ in $\sC$ such that $(V, \varepsilon_V \mu_V,\widetilde{i}_V)$ is a (left) dual of $V$ in $\cC$, that is,
  \begin{equation*}
   V\xrightarrow{l_V^{-1}}\vac\tens V\xrightarrow{\widetilde{i}_V\tens 1_V} (V\tens V)\tens V\xrightarrow{\cA_{V,V,V}^{-1}} V\tens(V\tens V)\xrightarrow{1_V\tens (\varepsilon_V \mu_V)} V\tens\vac\xrightarrow{r_V} V
  \end{equation*}
and
\begin{equation*}
 V\xrightarrow{r_V^{-1}}V\tens\vac\xrightarrow{1_V\tens\widetilde{i}_V} V\tens (V\tens V)\xrightarrow{\cA_{V,V,V}} (V\tens V)\tens V\xrightarrow{(\varepsilon_V \mu_V)\tens 1_V} \vac\tens V\xrightarrow{l_V} V
\end{equation*}
both equal the identity on $V$.

\item The morphism $\vac\xrightarrow{\widetilde{i}_V} V\tens V\xrightarrow{\mu_V} V$ in $\hom_\cC(\vac,V)$ equals $\vert G\vert\iota_V$.
 \end{itemize}
\end{assum}

\begin{rema}
 The fourth and sixth assumptions above imply that the \textit{dimension} of $V$, defined by
 \begin{equation*}
  \dim V=\varepsilon_V \mu_V \widetilde{i}_V\in\Endo_\cC(\vac)=\mathbb{F},
 \end{equation*}
is equal to $\vert G\vert$. Conversely, since $V$ is haploid, the final condition above follows from $\dim V=\vert G\vert$. 
\end{rema}

We now state the theorem which is the main technical result of this paper:
\begin{theo}\label{thm:repV=repGV}
 Under Assumption \ref{mainassum}, every object $W$ in $\rep V$ is a direct sum $W=\bigoplus_{g\in G} W_g$ where $W_g$ is a (possibly zero) $g$-twisted $V$-module.
\end{theo}

 The idea of the proof is to find the projections from $W$ to all of its $g$-twisted summands. That is, we need to construct morphisms $\lbrace\pi_g: W\rightarrow W\rbrace_{g\in G}$ which satisfy:
 \begin{enumerate}
  \item Each $\pi_g$ is a morphism in $\rep V$.
  \item For each $g\in G$, the image $\pi_g(W)$ is a $g$-twisted $V$-module.
  \item For all $g,h\in G$, $\pi_g\pi_h=\delta_{g,h}\pi_g$, and $\sum_{g\in G} \pi_g=1_W$.
 \end{enumerate}
 We shall verify these properties for the morphisms $\pi_g=\vert G\vert^{-1}\Pi_g$, where $\Pi_g: W\rightarrow W$ is the composition
 \begin{align*}
  W\xrightarrow{l_W^{-1}}\vac\tens W\xrightarrow{\widetilde{i}_V\tens 1_W} (V\tens V)\tens W\xrightarrow{\cA_{V,V,W}^{-1}} V\tens(V\tens W) & \xrightarrow{1_V\tens\cM_{V,W}} V\tens(V\tens W)\nonumber\\
  &\xrightarrow{1_V\tens(g\tens 1_W)} V\tens(V\tens W)\xrightarrow{1_V\tens\mu_W} V\tens W\xrightarrow{\mu_W} W.
 \end{align*}
 We represent $\pi_g$ pictorially using braid diagrams as follows:
 \begin{align*}
\pi_g = \frac{1}{\vert G\vert}
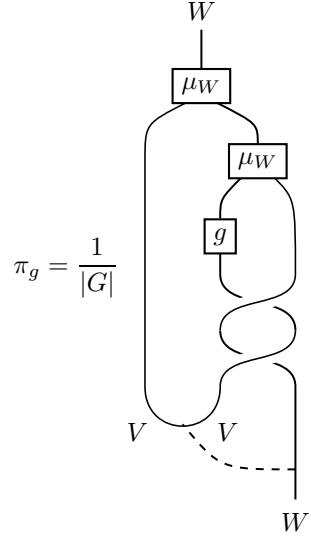
 \begin{matrix}
  \begin{tikzpicture}[scale = 1, baseline = {(current bounding box.center)}, line width=0.75pt]
   \node (w) at (2,-0.3) {$W$};
   \node (g) at (1, 3.5) [draw,minimum width=10pt,minimum height=10pt,thick, fill=white] {$g$};
   \draw(1,2.25) .. controls (1,1.75) and (2,2) .. (2,1.5) -- (2,0);
   \draw[white, double=black, line width = 3pt ] (2,2.25) .. controls (2,1.75) and (1,2) .. (1,1.5) .. controls (1,0.8) and (0,0.8) .. (0,1.5) -- (0, 4.5) .. controls (0,5) .. (0.55,5.28);
   \draw[dashed] (2,0.4) .. controls (1,.4) .. (0.5,1);
   \draw (1,3.28) -- (1,3) .. controls (1,2.5) and (2,2.75) .. (2,2.25);
   \draw[white, double=black, line width = 3pt ] (1.7, 4.3) .. controls (2, 4) .. (2,3) .. controls (2,2.5) and (1,2.75) .. (1,2.25);
   \draw(1, 3.74) .. controls (1, 4) .. (1.3, 4.3);
   \node (mu1) at (1.5, 4.5) [draw,minimum width=20pt,minimum height=10pt,thick, fill=white] {$\mu_W$};
   \node (mu2) at (.75, 5.5) [draw,minimum width=20pt,minimum height=10pt,thick, fill=white] {$\mu_W$};
   \draw (1.5, 4.72) .. controls (1.5,5) .. (.95, 5.28);
   \draw (.75,5.72) -- (.75,6.25);
   \node at (.75, 6.5) {$W$};
   \node at (-.1,.9) {$V$};
   \node at (1.1,.9) {$V$};
  \end{tikzpicture}
 \end{matrix}
\end{align*}
 Before proving the properties of $\pi_g$ listed above, we note two corollaries:
 \begin{corol}\label{cor:RepV_G-crossed}
  Under Assumption \ref{mainassum}, $\rep\,V$ is a braided $G$-crossed supercategory.
 \end{corol}
 \begin{proof}
  This follows immediately from Theorems \ref{Gcrossedfromtwist} and \ref{thm:repV=repGV} once we verify that $\hom_{\rep\,V}(W_1,W_2)=0$ when $W_1$ is $g_1$-twisted, $W_2$ is $g_2$-twisted, and $g_1\neq g_2$. Observe first that for a $g$-twisted module $W$, $\pi_g$ is the identity on $W$ by the definition of $g$-twisted module, associativity of $\mu_W$, the final condition in Assumption \ref{mainassum}, and the unit property of $W$. Second, the projections $\pi_g$ commute with morphisms $f: W_1\rightarrow W_2$ in $\rep\,V$ due to properties of natural isomorphisms in $\sC$ and $f\mu_{W_1}=\mu_{W_2}(1_V\tens f)$. So if $W_1$ is $g_1$-twisted and $W_2$ is $g_2$-twisted,
  \begin{equation*}
   f = \pi_{g_2} f\pi_{g_1} =f\pi_{g_2}\pi_{g_1} =\delta_{g_1,g_2} f\pi_{g_1}=\delta_{g_1,g_2} f,
  \end{equation*}
and $f=0$ if $g_1\neq g_2$.
 \end{proof}

\begin{corol}
 Under Assumption \ref{mainassum},  induction $\cF: \sC\rightarrow\mathcal{S}(\rep\,V)^G$ is a braided monoidal superfunctor and restricts to a braided tensor functor $\cF: \cC\rightarrow(\rep\,V)^{G}$.
\end{corol}
\begin{proof}
 This follows directly from Theorems \ref{thm:F_braided}, \ref{thm:F_even_braided}, and \ref{thm:repV=repGV}.
\end{proof}

 The proof of Theorem \ref{thm:repV=repGV} starts with some preliminary lemmas. In this section, we give proofs by braid diagram for brevity and clarity; see Appendix \ref{app:mthm_details} for full calculations, incorporating for example associativity isomorphisms.
 \begin{lemma}\label{iota_lemma}
  The composition $\vac\xrightarrow{\widetilde{i}_V} V\tens V\xrightarrow{1_V\tens\varepsilon_V} V\tens\vac\xrightarrow{r_V} V$ equals $\iota_V$.
 \end{lemma}
 \begin{proof}
 Consider the linear map $\Phi: \hom_\cC(V,\vac)\rightarrow\hom_\cC(\vac,V)$
which sends $f: V\rightarrow\vac$ to the composition $$\vac\xrightarrow{\widetilde{i}_V}V\tens V\xrightarrow{1_V\tens f} V\tens\vac\xrightarrow{r_V} V.$$
In particular, the morphism indicated in the statement of the lemma is $\Phi(\varepsilon_V)$. Because $(V,\widetilde{i}_V,\varepsilon_V\mu_V)$ is a dual of $V$ in $\cC$, $\Phi$ is an isomorphism with inverse sending $g:\vac\rightarrow V$ to the composition
\begin{equation*}
 V\xrightarrow{r_V^{-1}} V\tens\vac\xrightarrow{1_V\tens g} V\tens V\xrightarrow{\varepsilon_V\mu_V} \vac.
\end{equation*}
In particular $\Phi^{-1}(\iota_V)=\varepsilon_V$ by the right unit property of $V$, so $\Phi(\varepsilon_V)=\Phi(\Phi^{-1}(\iota_V)=\iota_V$.
\end{proof}
 
 \begin{lemma}\label{rigidlike_lemma}
  The two morphisms $V\rightarrow V\tens V$ in $\cC$ given by the compositions
  \begin{equation*}
   V\xrightarrow{l_V^{-1}} \vac\tens V\xrightarrow{\widetilde{i}_V\tens 1_V} (V\tens V)\tens V\xrightarrow{\cA_{V,V,V}^{-1}} V\tens(V\tens V)\xrightarrow{1_V\tens\mu_V} V\tens V
  \end{equation*}
and
\begin{equation*}
 V\xrightarrow{r_V^{-1}} V\tens\vac\xrightarrow{1_V\tens\widetilde{i}_V} V\tens(V\tens V)\xrightarrow{\cA_{V,V,V}} (V\tens V)\tens V\xrightarrow{\mu_V\tens 1_V} V\tens V
\end{equation*}
are equal. Diagrammatically,
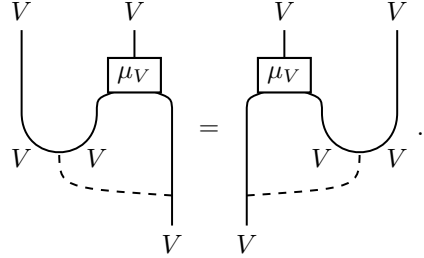
\begin{align*}
 \begin{matrix}
  \begin{tikzpicture}[scale = 1, baseline = {(current bounding box.center)}, line width=0.75pt]
   \draw (2,0) -- (2,1.5) .. controls (2,1.7) .. (1.7,1.8);
   \draw (1.5,2.2) -- (1.5,2.6);
   \draw (1.3,1.8) .. controls (1,1.7) .. (1,1.5) .. controls (1,.8) and (0,.8) .. (0,1.5) -- (0,2.6);
   \draw[dashed] (2,.4) .. controls (.5,.5) .. (.5,1);
   \node at (1.5,2) [draw,minimum width=20pt,minimum height=10pt,thick, fill=white] {$\mu_V$};
   \node at (2,-.25) {$V$};
   \node at (1.5,2.85) {$V$};
   \node at (0,2.85) {$V$};
   \node at (0,.9) {$V$};
   \node at (1,.9) {$V$};
  \end{tikzpicture}
 \end{matrix} =
 \begin{matrix}
  \begin{tikzpicture}[scale = 1, baseline = {(current bounding box.center)}, line width=0.75pt]
   \draw (0,0) -- (0,1.5) .. controls (0,1.7) .. (.3,1.8);
   \draw (.5,2.2) -- (.5,2.6);
   \draw (.7,1.8) .. controls (1,1.7) .. (1,1.5) .. controls (1,.8) and (2,.8) .. (2,1.5) -- (2,2.6);
   \draw[dashed] (0,.4) .. controls (1.5,.5) .. (1.5,1);
   \node at (.5,2) [draw,minimum width=20pt,minimum height=10pt,thick, fill=white] {$\mu_V$};
   \node at (0,-.25) {$V$};
   \node at (.5,2.85) {$V$};
   \node at (2,2.85) {$V$};
   \node at (1,.9) {$V$};
   \node at (2,.9) {$V$};
  \end{tikzpicture}
 \end{matrix} .
\end{align*}
 \end{lemma}
 \begin{proof}
 Since $V$ is rigid with dual $V$ and evaluation $\varepsilon_V\mu_V$, $V\tens V$ is also rigid with dual $V\tens V$ and evaluation
 \begin{equation*}
  e_{V\tens V}: (V\tens V)\tens(V\tens V)\rightarrow\vac
 \end{equation*}
given by the composition
\begin{align*}
 (V\tens V)\tens(V\tens V) \xrightarrow{\cA_{V,V,V\tens V}^{-1}} V\tens (V\tens (V\tens V)) & \xrightarrow{1_V\tens\cA_{V,V,V}} (V\tens(V\tens V))\tens V\nonumber\\
 & \xrightarrow{1_V\tens(\varepsilon_V\mu_V\tens 1_V)} V\tens(\vac\tens V)\xrightarrow{1_V\tens l_V} V\tens V\xrightarrow{\varepsilon_V\mu_V} \vac.
\end{align*}
As $V\tens V$ is rigid, the map $\mathrm{Hom}_\mathcal{C}(V,V\tens V)\rightarrow\mathrm{Hom}_{\mathcal{C}}((V\tens V)\tens V,\vac)$ given by $F\mapsto e_{V\tens V}(1_{V\tens V}\tens F)$ is an isomorphism, so letting $F_L$ and $F_R$ denote the morphisms in the statement of the lemma, it is enough to show
\begin{equation*}
 e_{V\tens V}(1_{V\tens V}\tens F_L) = e_{V\tens V}(1_{V\tens V}\tens F_R).
\end{equation*}
In fact, we will show that these two morphisms equal $\varepsilon_V\mu_V(\mu_V\tens 1_V)$, or equivalently $\varepsilon_V\mu_V(1_V\tens\mu_V)\cA_{V,V,V}^{-1}$.

We analyze $e_{V\tens V}(1_{V\tens V}\tens F_L)$ as follows:
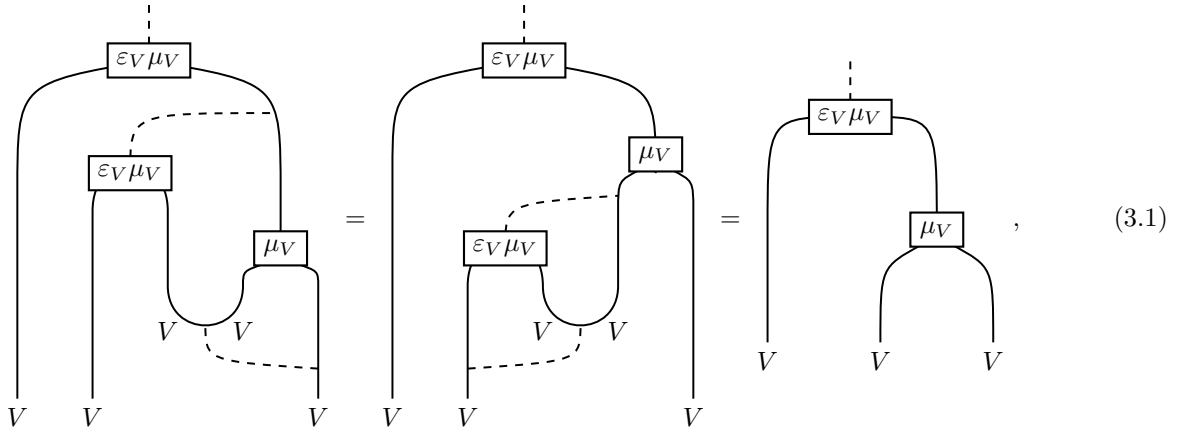
\begin{align}\label{eqn:rigidlike_lemma_1}
   \begin{matrix}
    \begin{tikzpicture}[scale = 1, baseline = {(current bounding box.center)}, line width=0.75pt]
    \draw (0,0) -- (0,2.75) .. controls(0,4.2) .. (1.75,4.5);
    \draw (3.5,2) -- (3.5,2.75) .. controls (3.5,4.2) .. (1.75,4.5);
    \draw[dashed] (1.5,3) .. controls (1.5,3.8) .. (3.4,3.8);
    \draw (1,0) -- (1,2.5) .. controls (1,3.3) and (2,3.3) .. (1.5,3);
    \draw (1.5, 3) .. controls (1,3.3) and (2,3.3) .. (2,2.5) -- (2,1.5) .. controls (2,.8) and (3,.8) .. (3,1.5) .. controls (3,1.7) .. (3.3,1.8);
    \draw (4,0) -- (4,1.5) .. controls (4,1.7) .. (3.7,1.8);
    \draw[dashed] (4,.4) .. controls (2.5,.5) .. (2.5,1);
    \draw[dashed] (1.75,4.5) -- (1.75,5.25);
     \node at (3.5,2) [draw,minimum width=20pt,minimum height=10pt,thick, fill=white] {$\mu_V$};
     \node at (1.5,3) [draw,minimum width=20pt,minimum height=10pt,thick, fill=white] {$\varepsilon_V\mu_V$};
     \node at (1.75,4.5) [draw,minimum width=20pt,minimum height=10pt,thick, fill=white] {$\varepsilon_V\mu_V$};
     \node at (0,-.25) {$V$};
     \node at (1,-.25) {$V$};
     \node at (4,-.25) {$V$};
     \node at (2,.9) {$V$};
     \node at (3,.9) {$V$};
    \end{tikzpicture}
\end{matrix} =
\begin{matrix}
    \begin{tikzpicture}[scale = 1, baseline = {(current bounding box.center)}, line width=0.75pt]
    \draw (0,0) -- (0,3) .. controls(0,4.2) .. (1.75,4.5);
    \draw (3.5,3) .. controls (3.5,4.2) .. (1.75,4.5);
    \draw[dashed] (1.5,2) .. controls (1.5,2.6) .. (3,2.7);
    \draw (1,0) -- (1,1.5) .. controls (1,2.3) and (2,2.3) .. (1.5,2);
    \draw (1.5, 2) .. controls (1,2.3) and (2,2.3) .. (2,1.5) .. controls (2,.8) and (3,.8) .. (3,1.5)--(3,2.5) .. controls (3,2.9) .. (3.3,3.05);
    \draw (4,0) -- (4,2.5) .. controls (4,2.9) .. (3.7,3.05);
    \draw[dashed] (1,.4) .. controls (2.5,.5) .. (2.5,1);
    \draw[dashed] (1.75,4.5) -- (1.75,5.25);
     \node at (3.5,3.25) [draw,minimum width=20pt,minimum height=10pt,thick, fill=white] {$\mu_V$};
     \node at (1.5,2) [draw,minimum width=20pt,minimum height=10pt,thick, fill=white] {$\varepsilon_V\mu_V$};
     \node at (1.75,4.5) [draw,minimum width=20pt,minimum height=10pt,thick, fill=white] {$\varepsilon_V\mu_V$};
     \node at (0,-.25) {$V$};
     \node at (1,-.25) {$V$};
     \node at (4,-.25) {$V$};
     \node at (2,.9) {$V$};
     \node at (3,.9) {$V$};
    \end{tikzpicture}
\end{matrix} =
\begin{matrix}
 \begin{tikzpicture}[scale = 1, baseline = {(current bounding box.center)}, line width=0.75pt]
 \draw (0,0) -- (0,1.5) .. controls (0,3) .. (1.1,3);
 \draw (1.5, 0) .. controls (1.5,1) .. (2.05,1.3);
 \draw (2.25,1.5) .. controls (2.25,3) .. (1.3,3);
 \draw (3,0) .. controls (3,1) .. (2.45,1.3);
 \draw[dashed] (1.1,3) -- (1.1,3.75);
  \node at (2.25,1.5) [draw,minimum width=20pt,minimum height=10pt,thick, fill=white] {$\mu_V$};
     \node at (1.1,3) [draw,minimum width=20pt,minimum height=10pt,thick, fill=white] {$\varepsilon_V\mu_V$};
     \node at (0,-.25) {$V$};
     \node at (1.5,-.25) {$V$};
     \node at (3,-.25) {$V$};
 \end{tikzpicture}
\end{matrix},
  \end{align}
where we have used the rigidity of $V$ for the second step. On the other hand, $e_{V\tens V}(1_{V\tens V}\tens F_R)$ becomes: 
\begin{align}\label{eqn:rigidlike_lemma_2}
  \begin{matrix}
   \begin{tikzpicture}[scale = 1, baseline = {(current bounding box.center)}, line width=0.75pt]
   \draw (0,0) -- (0,2.5) .. controls (0,4.25) .. (2,4.5);
   \draw (1,0) -- (1,2.25) .. controls (1,2.9) .. (1.75,3);
   \draw (2,0) -- (2,1.5) .. controls (2,1.7) .. (2.3,1.8);
   \draw (2.5,2.2) -- (2.5,2.25) .. controls (2.5,2.9) .. (1.75,3);
   \draw (2.7,1.8) .. controls (3,1.7) ..(3,1.5) .. controls (3,.8) and (4,.8) .. (4,1.5) -- (4,2.5) .. controls (4,4.25) .. (2,4.5);
   \draw[dashed] (2,.4) .. controls (3.5,.5) .. (3.5,1);
   \draw[dashed] (1.75,3) .. controls (1.75,3.6) .. (3.95,3.75);
   \draw[dashed] (2,4.5)--(2,5.25);
    \node at (2.5,2) [draw,minimum width=20pt,minimum height=10pt,thick, fill=white] {$\mu_V$};
    \node at (1.75,3) [draw,minimum width=20pt,minimum height=10pt,thick, fill=white] {$\varepsilon_V\mu_V$};
    \node at (2,4.5) [draw,minimum width=20pt,minimum height=10pt,thick, fill=white] {$\varepsilon_V\mu_V$};
    \node at (0,-.25) {$V$};
    \node at (1,-.25) {$V$};
    \node at (2,-.25) {$V$};
    \node at (3,.9) {$V$};
    \node at (4,.9) {$V$};
   \end{tikzpicture}
  \end{matrix}
   =
   \begin{matrix}
    \begin{tikzpicture}[scale = 1, baseline = {(current bounding box.center)}, line width=0.75pt]
   \draw (0,0) -- (0,2.5) .. controls (0,4.25) .. (2,4.5);
   \draw (1,0) -- (1,1) .. controls (1,1.5) .. (1.3,1.8);
   \draw (2,0) -- (2,1) .. controls (2,1.5) .. (1.7,1.8);
   \draw (1.5,2.2) -- (1.5,2.25) .. controls (1.5,2.9) .. (2.25,3);
   \draw (2.25,3) .. controls (3,2.9) ..(3,2.25) -- (3,1.5) .. controls (3,.8) and (4,.8) .. (4,1.5) -- (4,2.5) .. controls (4,4.25) .. (2,4.5);
   \draw[dashed] (2,.4) .. controls (3.5,.5) .. (3.5,1);
   \draw[dashed] (2.25,3) .. controls (2.25,3.6) .. (3.95,3.75);
   \draw[dashed] (2,4.5)--(2,5.25);
    \node at (1.5,2) [draw,minimum width=20pt,minimum height=10pt,thick, fill=white] {$\mu_V$};
    \node at (2.25,3) [draw,minimum width=20pt,minimum height=10pt,thick, fill=white] {$\varepsilon_V\mu_V$};
    \node at (2,4.5) [draw,minimum width=20pt,minimum height=10pt,thick, fill=white] {$\varepsilon_V\mu_V$};
    \node at (0,-.25) {$V$};
    \node at (1,-.25) {$V$};
    \node at (2,-.25) {$V$};
    \node at (3,.9) {$V$};
    \node at (4,.9) {$V$};     
    \end{tikzpicture}
   \end{matrix}
   =
   \begin{matrix}
    \begin{tikzpicture}[scale = 1, baseline = {(current bounding box.center)}, line width=0.75pt]
   \draw (0,0) -- (0,2.5) .. controls (0,4.25) .. (2,4.5);
   \draw (1,0) .. controls (1,.5) .. (1.3,.8);
   \draw (2,0) .. controls (2,.5) .. (1.7,.8);
   \draw (1.5,1.2) -- (1.5,2.5) .. controls (1.5,2.9) .. (2.25,3);
   \draw (2.25,3) .. controls (3,2.9) .. (3,2.5) .. controls (3,1.8) and (4,1.8) .. (4,2.5) .. controls (4,4.25) .. (2,4.5);
   \draw[dashed] (1.5,1.5) .. controls (3.5,1.6) .. (3.5,2);
   \draw[dashed] (2.25,3) .. controls (2.25,3.6) .. (3.95,3.75);
   \draw[dashed] (2,4.5)--(2,5.25);
    \node at (1.5,1) [draw,minimum width=20pt,minimum height=10pt,thick, fill=white] {$\mu_V$};
    \node at (2.25,3) [draw,minimum width=20pt,minimum height=10pt,thick, fill=white] {$\varepsilon_V\mu_V$};
    \node at (2,4.5) [draw,minimum width=20pt,minimum height=10pt,thick, fill=white] {$\varepsilon_V\mu_V$};
    \node at (0,-.25) {$V$};
    \node at (1,-.25) {$V$};
    \node at (2,-.25) {$V$};
    \node at (3,1.9) {$V$};
    \node at (4,1.9) {$V$};     
    \end{tikzpicture}
   \end{matrix},
\end{align}
which by rigidity reduces to the right side of \eqref{eqn:rigidlike_lemma_1}.
 \end{proof}
 
 \begin{lemma}\label{trace_of_g_lemma}
  For $g\in G$, the composition $\vac\xrightarrow{\widetilde{i}_V} V\tens V\xrightarrow{1_V\tens g} V\tens V\xrightarrow{\mu_V} V$
equals $\vert G\vert\delta_{g,1}\iota_V$.
 \end{lemma}
 
 \allowdisplaybreaks
 
\begin{proof}
Since $V$ is haploid, the morphism in the lemma is a multiple of $\iota_V$, which we denote $\mathrm{Tr}_\cC\,g$. By assumption, $\mathrm{Tr}_\cC\, 1=\dim_\cC V=\vert G\vert$, so we just need to show $\mathrm{Tr}_\cC\,g=0$ for $g\neq 1$. We calculate using the left and right unit properties of $V$, the automorphism property of $g$, the associativity of $\mu_V$, and Lemma \ref{rigidlike_lemma}: 
\begin{align}\label{eqn:Tr_g_zero}
 (\mathrm{Tr}_\mathcal{C}\,g)1_V  & =  \begin{matrix}
    \begin{tikzpicture}[scale = 1, baseline = {(current bounding box.center)}, line width=0.75pt]
    \draw (2,0) -- (2,3.2) .. controls (2,3.6) .. (1.45,3.8);
    \draw (.5,3.2) .. controls (.5,3.6) .. (1.05,3.8);
    \draw (1,2.2) .. controls (1,2.6) .. (.7,2.8);
    \draw (1,1.8) -- (1,1.5) .. controls (1,.8) and (0,.8) .. (0,1.5) -- (0,2.2) .. controls (0,2.6) .. (.3,2.8);
    \draw (1.25,4.2) -- (1.25,4.6);
    \draw[dashed] (2,.4) .. controls (.5,.5) .. (.5,1);
    \node at (1,2) [draw,minimum width=10pt,minimum height=10pt,thick, fill=white] {$g$};
      \node at (.5,3) [draw,minimum width=20pt,minimum height=10pt,thick, fill=white] {$\mu_V$};
      \node at (1.25,4) [draw,minimum width=20pt,minimum height=10pt,thick, fill=white] {$\mu_V$};
      \node at (0,.9) {$V$};
      \node at (1,.9) {$V$};
      \node at (1.25,4.85) {$V$};
      \node at (2,-.25) {$V$};
    \end{tikzpicture}
\end{matrix} =
\begin{matrix}
    \begin{tikzpicture}[scale = 1, baseline = {(current bounding box.center)}, line width=0.75pt]
    \draw (2,0) -- (2,3.2) .. controls (2,3.6) .. (1.45,3.8);
    \draw (.5,3.2) .. controls (.5,3.6) .. (1.05,3.8);
    \draw (0,2.2) .. controls (0,2.6) .. (.3,2.8);
    \draw (0,1.8) -- (0,1.5) .. controls (0,.8) and (1,.8) .. (1,1.5) -- (1,2.2) .. controls (1,2.6) .. (.7,2.8);
    \draw (1.25,4.2) -- (1.25,4.8);
    \draw (1.25,5.1) -- (1.25, 5.6);
    \draw (2,-.9) -- (2,-.3);
    \draw[dashed] (2,.4) .. controls (.5,.5) .. (.5,1);
    \node at (0,2) [draw,minimum width=10pt,minimum height=10pt,thick, fill=white] {$g^{-1}$};
      \node at (.5,3) [draw,minimum width=20pt,minimum height=10pt,thick, fill=white] {$\mu_V$};
      \node at (1.25,4) [draw,minimum width=20pt,minimum height=10pt,thick, fill=white] {$\mu_V$};
      \node at (2,-.2) [draw,minimum width=10pt,minimum height=10pt,thick, fill=white] {$g^{-1}$};
      \node at (1.25,4.9) [draw,minimum width=10pt,minimum height=10pt,thick, fill=white] {$g$};
      \node at (0,.9) {$V$};
      \node at (1,.9) {$V$};
      \node at (1.25,5.85) {$V$};
      \node at (2,-1.15) {$V$};
    \end{tikzpicture}
\end{matrix} =
\begin{matrix}
    \begin{tikzpicture}[scale = 1, baseline = {(current bounding box.center)}, line width=0.75pt]
    \draw (2,0) -- (2,2.2) .. controls (2,2.6) .. (1.7,2.8);
    \draw (1.5,3.2) .. controls (1.5,3.6) .. (.95,3.8);
    \draw (0,2.2) -- (0,3.2) .. controls (0,3.6) .. (.55,3.8);
    \draw (0,1.8) -- (0,1.5) .. controls (0,.8) and (1,.8) .. (1,1.5) -- (1,2.2) .. controls (1,2.6) .. (1.3,2.8);
    \draw (.75,4.2) -- (.75,4.8);
    \draw (.75,5.1) -- (.75, 5.6);
    \draw (2,-.9) -- (2,-.3);
    \draw[dashed] (2,.4) .. controls (.5,.5) .. (.5,1);
    \node at (0,2) [draw,minimum width=10pt,minimum height=10pt,thick, fill=white] {$g^{-1}$};
      \node at (1.5,3) [draw,minimum width=20pt,minimum height=10pt,thick, fill=white] {$\mu_V$};
      \node at (.75,4) [draw,minimum width=20pt,minimum height=10pt,thick, fill=white] {$\mu_V$};
      \node at (2,-.2) [draw,minimum width=10pt,minimum height=10pt,thick, fill=white] {$g^{-1}$};
      \node at (.75,4.9) [draw,minimum width=10pt,minimum height=10pt,thick, fill=white] {$g$};
      \node at (0,.9) {$V$};
      \node at (1,.9) {$V$};
      \node at (.75,5.85) {$V$};
      \node at (2,-1.15) {$V$};
    \end{tikzpicture}
\end{matrix}\nonumber\\
& =
\begin{matrix}
\begin{tikzpicture}[scale = 1, baseline = {(current bounding box.center)}, line width=0.75pt]
    \draw (.5,2.2) -- (0.5,3.2) .. controls (0.5,3.6) .. (1.05,3.8);
    \draw (.7,1.8) .. controls (1,1.7) .. (1,1.5) .. controls (1,.8) and (2,.8) .. (2,1.5) -- (2,3.2) .. controls (2,3.6) .. (1.45,3.8);
    \draw (1.25,4.2) -- (1.25,5.6);
    \draw (0,-.9) -- (0, 1.2) .. controls (0,1.6) .. (.3,1.8);
    \draw[dashed] (0,.4) .. controls (1.5,.5) .. (1.5,1);
    \node at (.5,3) [draw,minimum width=10pt,minimum height=10pt,thick, fill=white] {$g^{-1}$};
      \node at (.5,2) [draw,minimum width=20pt,minimum height=10pt,thick, fill=white] {$\mu_V$};
      \node at (1.25,4) [draw,minimum width=20pt,minimum height=10pt,thick, fill=white] {$\mu_V$};
      \node at (0,-.2) [draw,minimum width=10pt,minimum height=10pt,thick, fill=white] {$g^{-1}$};
      \node at (1.25,5) [draw,minimum width=10pt,minimum height=10pt,thick, fill=white] {$g$};
      \node at (2,.9) {$V$};
      \node at (1,.9) {$V$};
      \node at (1.25,5.85) {$V$};
      \node at (0,-1.15) {$V$};
    \end{tikzpicture}
\end{matrix} =
\begin{matrix}
\begin{tikzpicture}[scale = 1, baseline = {(current bounding box.center)}, line width=0.75pt]
    \draw (0.5,3.2) .. controls (0.5,3.6) .. (1.05,3.8);
    \draw (.7,2.8) .. controls (1,2.6) .. (1,2.2) -- (1,2) .. controls (1,1.3) and (2,1.3) .. (2,2) -- (2,3.2) .. controls (2,3.6) .. (1.45,3.8);
    \draw (1.25,4.2) -- (1.25,4.6);
    \draw (0,-.4) -- (0, 2.2) .. controls (0,2.6) .. (.3,2.8);
    \draw[dashed] (0,.9) .. controls (1.5,1) .. (1.5,1.5);
    \node at (2,3) [draw,minimum width=10pt,minimum height=10pt,thick, fill=white] {$g$};
      \node at (.5,3) [draw,minimum width=20pt,minimum height=10pt,thick, fill=white] {$\mu_V$};
      \node at (1.25,4) [draw,minimum width=20pt,minimum height=10pt,thick, fill=white] {$\mu_V$};
      \node at (0,.3) [draw,minimum width=10pt,minimum height=10pt,thick, fill=white] {$g^{-1}$};
      \node at (2,1.4) {$V$};
      \node at (1,1.4) {$V$};
      \node at (1.25,4.85) {$V$};
      \node at (0,-.65) {$V$};
    \end{tikzpicture}
\end{matrix} =
\begin{matrix}
\begin{tikzpicture}[scale = 1, baseline = {(current bounding box.center)}, line width=0.75pt]
    \draw (1.5,3.7) .. controls (1.5,4.1) .. (.95,4.3);
    \draw (1.3, 3.3) .. controls (1,3.1) .. (1,2.7) -- (1,2) .. controls (1,1.3) and (2,1.3) .. (2,2) -- (2,2.7) .. controls (2,3.1) .. (1.7,3.3);
    \draw (.75,4.7) -- (.75,5.1);
    \draw (0,-.4) -- (0, 3.7) .. controls (0,4.1) .. (.55,4.3);
    \draw[dashed] (0,.9) .. controls (1.5,1) .. (1.5,1.5);
    \node at (2,2.5) [draw,minimum width=10pt,minimum height=10pt,thick, fill=white] {$g$};
      \node at (1.5,3.5) [draw,minimum width=20pt,minimum height=10pt,thick, fill=white] {$\mu_V$};
      \node at (.75,4.5) [draw,minimum width=20pt,minimum height=10pt,thick, fill=white] {$\mu_V$};
      \node at (0,.3) [draw,minimum width=10pt,minimum height=10pt,thick, fill=white] {$g^{-1}$};
      \node at (2,1.4) {$V$};
      \node at (1,1.4) {$V$};
      \node at (.75,5.35) {$V$};
      \node at (0,-.65) {$V$};
    \end{tikzpicture}
\end{matrix} =(\mathrm{Tr}_\mathcal{C}\,g)g^{-1}.
\end{align}
Thus $\mathrm{Tr}_\cC\,g=0$ unless $g$ is the identity.
\end{proof}
 
 Now we begin checking that the $\cC$-morphisms $\pi_g$ (or equivalently, the $\Pi_g$) satisfy the required properties:
 
 \bigskip
 
 \noindent\textbf{1. Each $\Pi_g$ is a morphism in $\rep V$.} We need to show that $\mu_W(1_V\tens\Pi_g)=\Pi_g\mu_W$. The proof goes as indicated by the braid diagrams:
 \begin{align*}
 \begin{matrix}
  \begin{tikzpicture}[scale = 1, baseline = {(current bounding box.center)}, line width=0.75pt]
   \node (w) at (2,-0.3) {$W$};
   \node (g) at (1, 3.5) [draw,minimum width=10pt,minimum height=10pt,thick, fill=white] {$g$};
   \draw(1,2.25) .. controls (1,1.75) and (2,2) .. (2,1.5) -- (2,0);
   \draw[white, double=black, line width = 3pt ] (2,2.25) .. controls (2,1.75) and (1,2) .. (1,1.5) .. controls (1,0.8) and (0,0.8) .. (0,1.5) -- (0, 4.5) .. controls (0,5) .. (0.55,5.28);
   \draw[dashed] (2,0.4) .. controls (1,.4) .. (0.5,1);
   \draw (1,3.28) -- (1,3) .. controls (1,2.5) and (2,2.75) .. (2,2.25);
   \draw[white, double=black, line width = 3pt ] (1.7, 4.3) .. controls (2, 4) .. (2,3) .. controls (2,2.5) and (1,2.75) .. (1,2.25);
   \draw(1, 3.74) .. controls (1, 4) .. (1.3, 4.3);
   \node (mu1) at (1.5, 4.5) [draw,minimum width=20pt,minimum height=10pt,thick, fill=white] {$\mu_W$};
   \node (mu2) at (.75, 5.5) [draw,minimum width=20pt,minimum height=10pt,thick, fill=white] {$\mu_W$};
   \draw (1.5, 4.72) .. controls (1.5,5) .. (.95, 5.28);
   \draw (.75,5.72) .. controls (.75, 6) .. (.1,6.28);
   \node at (-.1,.9) {$V$};
   \node at (1.1,.9) {$V$};
   \node at (-0.1, 6.5) [draw,minimum width=20pt,minimum height=10pt,thick, fill=white] {$\mu_W$};
   \draw (-1, 0) -- (-1, 5.5) .. controls (-1, 5.9) .. (-0.3, 6.28);
   \node at (-1,-.3) {$V$};
   \draw(-0.1, 6.72) -- (-0.1, 7.25);
   \node at (-.1, 7.5) {$W$};
  \end{tikzpicture}
 \end{matrix} = \hspace{-4.35em} 
 \begin{matrix}
  \begin{tikzpicture}[scale = 1, baseline = {(current bounding box.center)}, line width=0.75pt]
   \node (w) at (2,-0.3) {$W$};
   \node (g) at (1, 3.5) [draw,minimum width=10pt,minimum height=10pt,thick, fill=white] {$g$};
   \draw(1,2.25) .. controls (1,1.75) and (2,2) .. (2,1.5) -- (2,0);
   \draw[white, double=black, line width = 3pt ] (2,2.25) .. controls (2,1.75) and (1,2) .. (1,1.5) .. controls (1,0.8) and (0,0.8) .. (0,1.5) -- (0, 3.5) .. controls (0,4) .. (0.3,4.28);
   \draw[dashed] (-1,0.3) .. controls (-3, .8) and (1,.4) .. (0.5,1);
   \draw (1,3.28) -- (1,3) .. controls (1,2.5) and (2,2.75) .. (2,2.25);
   \draw[white, double=black, line width = 3pt ] (1.45, 5.3) .. controls (2, 5) .. (2,4.5)--(2,3) .. controls (2,2.5) and (1,2.75) .. (1,2.25);
   \draw(1, 3.74) .. controls (1, 4) .. (0.7, 4.3);
   \node (mu1) at (0.5, 4.5) [draw,minimum width=20pt,minimum height=10pt,thick, fill=white] {$\mu_V$};
   \node (mu2) at (1.25, 5.5) [draw,minimum width=20pt,minimum height=10pt,thick, fill=white] {$\mu_W$};
   \draw (0.5, 4.72) .. controls (0.5,5) .. (1.05, 5.28);
   \draw (1.25,5.72) .. controls (1.25, 6) .. (.3,6.28);
   \node at (-.1,.9) {$V$};
   \node at (1.1,.9) {$V$};
   \node at (0.1, 6.5) [draw,minimum width=20pt,minimum height=10pt,thick, fill=white] {$\mu_W$};
   \draw (-1, 0) -- (-1, 5.5) .. controls (-1, 5.9) .. (-0.1, 6.28);
   \node at (-1,-.3) {$V$};
   \draw(0.1, 6.72) -- (0.1, 7.25);
   \node at (0.1, 7.5) {$W$};
  \end{tikzpicture}
 \end{matrix} = \hspace{-.25em}
 \begin{matrix}
  \begin{tikzpicture}[scale = 1, baseline = {(current bounding box.center)}, line width=0.75pt]
   \node at (1, -0.3) {$V$};
   \node at (2, -0.3) {$W$};
   \node at (1, 3.5) [draw,minimum width=10pt,minimum height=10pt,thick, fill=white] {$g$};
   \draw (1,0) -- (1,0.75) .. controls (1, 1.25) and (0,1) .. (0, 1.5) .. controls (0, 2) and (-1,1.75) .. (-1, 2.25) -- (-1,4.5) .. controls (-1, 5) .. (-.45, 5.28);
   \draw (2,2.25) .. controls (2, 2.75) and (1, 2.5) .. (1, 3) -- (1, 3.28);
   \draw[white, double=black, line width = 3pt ] (2,0) -- (2, 1.5) .. controls (2, 2) and (1, 1.75) .. (1,2.25) .. controls (1, 2.75) and (2, 2.5) .. (2,3) -- (2,5.5) .. controls (2,6) .. (1.1, 6.28);
   \draw[white, double=black, line width = 3pt ] (2,2.25) .. controls (2, 1.75) and (1, 2) .. (1,1.5) .. controls (1, 1) and (0,1.25) .. (0,0.75) .. controls (0, 0.2) and (-1,0.2) .. (-1, 0.75) -- (-1, 1.5) .. controls (-1, 2) and (0, 1.75) .. (0,2.25) -- (0, 3.5) .. controls (0,4) .. (0.3, 4.28);
   \draw(1, 3.72) .. controls (1,4) .. (.7,4.28);
   \draw(0.5, 4.72) .. controls (0.5, 5) .. (-.05,5.28);
   \draw(-.25,5.72) .. controls (-.25, 6) .. (.7,6.28);
   \draw (0.9,6.72) -- (0.9, 7.25);
   \draw[dashed] (1,.2) .. controls (-0.5,0.2) .. (-0.5,0.3);
   \node at (0.5, 4.5) [draw,minimum width=20pt,minimum height=10pt,thick, fill=white] {$\mu_V$};
   \node at (-0.25, 5.5) [draw,minimum width=20pt,minimum height=10pt,thick, fill=white] {$\mu_V$};
   \node at (0.9, 6.5) [draw,minimum width=20pt,minimum height=10pt,thick, fill=white] {$\mu_W$};
   \node at (0.9, 7.5) {$W$};
   \node at (-1.1, 0.3) {$V$};
   \node at (0.2,0.5) {$V$};
  \end{tikzpicture}
 \end{matrix} = \hspace{-.25em}
 \begin{matrix}
  \begin{tikzpicture}[scale = 1, baseline = {(current bounding box.center)}, line width=0.75pt]
   \node at (1, -0.3) {$V$};
   \node at (2, -0.3) {$W$};
   \node at (1, 3.5) [draw,minimum width=10pt,minimum height=10pt,thick, fill=white] {$g$};
   \draw (1,0) -- (1,0.75) .. controls (1, 1.25) and (0,1) .. (0, 1.5) -- (0, 3.5) .. controls (0,4) .. (-.3, 4.28);
   \draw (2,2.25) .. controls (2, 2.75) and (1, 2.5) .. (1, 3) -- (1, 3.28);
   \draw[white, double=black, line width = 3pt ] (2,0) -- (2, 1.5) .. controls (2, 2) and (1, 1.75) .. (1,2.25) .. controls (1, 2.75) and (2, 2.5) .. (2,3) -- (2,5.5) .. controls (2,6) .. (1.3, 6.28);
   \draw[white, double=black, line width = 3pt ] (2,2.25) .. controls (2, 1.75) and (1, 2) .. (1,1.5) .. controls (1, 1) and (0,1.25) .. (0,0.75) .. controls (0, 0.2) and (-1,0.2) .. (-1, 0.75) -- (-1, 3.5) .. controls (-1, 4) .. (-.7, 4.28);
   \draw(1, 3.72) -- (1, 4.5) .. controls (1,5) .. (.45,5.28);
   \draw(-0.5, 4.72) .. controls (-0.5, 5) .. (.05,5.28);
   \draw(.25,5.72) .. controls (.25, 6) .. (.9,6.28);
   \draw (1.1,6.72) -- (1.1, 7.25);
   \draw[dashed] (1,.2) .. controls (-0.5,0.2) .. (-0.5,0.3);
   \node at (-0.5, 4.5) [draw,minimum width=20pt,minimum height=10pt,thick, fill=white] {$\mu_V$};
   \node at (0.25, 5.5) [draw,minimum width=20pt,minimum height=10pt,thick, fill=white] {$\mu_V$};
   \node at (1.1, 6.5) [draw,minimum width=20pt,minimum height=10pt,thick, fill=white] {$\mu_W$};
   \node at (1.1, 7.5) {$W$};
   \node at (-1.1, 0.3) {$V$};
   \node at (0.2,0.5) {$V$};
  \end{tikzpicture}
 \end{matrix}
\end{align*}
\begin{align}\label{eqn:Pi_g_Vhom}
 =
 \begin{matrix}
  \begin{tikzpicture}[scale = 1, baseline = {(current bounding box.center)}, line width=0.75pt]
  \node at (1, -0.3) {$V$};
   \node at (2, -0.3) {$W$};
  \draw(2,2) .. controls (2, 2.3) and (1, 2.2) .. (1,2.5) -- (1, 2.8);
  \draw(1,3.2) -- (1, 3.5) .. controls (1, 3.8) and (0, 3.7) .. (0, 4) .. controls (0,4.15) .. (.3, 4.28);
   \draw[white, double=black, line width = 3pt ] (1,0) -- (1,0.75) .. controls (1, 1.25) and (0,1) .. (0, 1.5) -- (0, 3.5) .. controls (0,3.8) and (1,3.7) .. (1,4) .. controls (1,4.15) .. (.7, 4.28);
  \draw[white, double=black, line width = 3pt ] (2,0) -- (2, 1.5) .. controls (2, 1.8) and (1, 1.7) .. (1,2) .. controls (1, 2.3) and (2, 2.2) .. (2,2.5) -- (2,4.5) .. controls (2,5) .. (1.45, 5.28);
    \draw[white, double=black, line width = 3pt ] (2,2) .. controls (2, 1.7) and (1, 1.8) .. (1,1.5) .. controls (1, 1) and (0,1.25) .. (0,0.75) .. controls (0, 0.2) and (-1,0.2) .. (-1, 0.75) -- (-1, 5.5) .. controls (-1, 6) .. (-.1, 6.28);
    \draw(.5,4.72) .. controls (.5,5) .. (1.05, 5.28);
    \draw(1.25,5.72) .. controls (1.25,6) .. (.3,6.28);
    \draw(0.1, 6.72) -- (.1,7.25);
    \node at (1, 3) [draw,minimum width=10pt,minimum height=10pt,thick, fill=white] {$g$};
    \node at (0.5, 4.5) [draw,minimum width=20pt,minimum height=10pt,thick, fill=white] {$\mu_V$};
   \node at (1.25, 5.5) [draw,minimum width=20pt,minimum height=10pt,thick, fill=white] {$\mu_W$};
   \node at (0.1, 6.5) [draw,minimum width=20pt,minimum height=10pt,thick, fill=white] {$\mu_W$};
   \draw[dashed] (1,.2) .. controls (-0.5,0.2) .. (-0.5,0.3);
   \node at (0.1, 7.5) {$W$};
   \node at (-1.1, 0.3) {$V$};
   \node at (0.2,0.5) {$V$};
  \end{tikzpicture}
 \end{matrix} =
 \begin{matrix}
  \begin{tikzpicture}[scale = 1, baseline = {(current bounding box.center)}, line width=0.75pt]
  \draw(-0.1, 6.72) -- (-.1,7.25);
  \draw(.75,5.72) .. controls (.75,6) .. (.1, 6.28);
  \draw(0,4.72) .. controls (0,5) .. (.55, 5.28);
  \draw(1.5,4.72) .. controls (1.5,5) .. (.95,5.28);
  \draw(0,4.28) -- (0,4) .. controls (0,3.5) and (1,3.75) .. (1,3.25) .. controls (1,2.75) and (2,3) .. (2,2.5);
  \draw[white, double=black, line width = 3pt ] (1,0) -- (1,1) .. controls (1,1.5) and (0,1.25) .. (0,1.75) -- (0,3.25) .. controls (0,3.75) and (1,3.5) .. (1,4) .. controls (1,4.15) .. (1.3,4.28);
  \draw[white, double=black, line width = 3pt ] (2,0) -- (2,1.75) .. controls (2,2.25) and (1,2) .. (1,2.5) .. controls (1,3) and (2,2.75) .. (2,3.25) -- (2,4) .. controls (2,4.15) .. (1.7,4.28);
 \draw[white, double=black, line width = 3pt ] (2,2.5) .. controls (2,2) and (1,2.25) .. (1,1.75) .. controls (1,1.25) and (0,1.5) .. (0,1) .. controls (0,.3) and (-1,.3) .. (-1,1) -- (-1,5.5) .. controls (-1,6) .. (-.3,6.28);
 \draw[dashed] (1,.3) .. controls (-0.5,0.3) .. (-0.5,0.5);
   \node at (0,4.5) [draw,minimum width=10pt,minimum height=10pt,thick, fill=white] {$g$};
   \node at (1.5,4.5) [draw,minimum width=20pt,minimum height=10pt,thick, fill=white] {$\mu_W$};
   \node at (.75, 5.5) [draw,minimum width=20pt,minimum height=10pt,thick, fill=white] {$\mu_W$};
   \node at (-.1,6.5) [draw,minimum width=20pt,minimum height=10pt,thick, fill=white] {$\mu_W$};
   \node at (-.1, 7.5) {$W$};
   \node at (-1.1,.5) {$V$};
   \node at (.1,.5) {$V$};
   \node at (1,-.3) {$V$};
   \node at (2,-.3) {$W$};
  \end{tikzpicture} 
\end{matrix} =
\begin{matrix}
  \begin{tikzpicture}[scale = 1, baseline = {(current bounding box.center)}, line width=0.75pt]
  \node at (1.5,-1.3) {$V$};
   \node (w) at (2.5,-1.3) {$W$};
   \draw(1.5,-1) .. controls (1.5,-.75) .. (1.8,-.48);
   \draw(2.5,-1) .. controls (2.5,-.75) .. (2.2,-.48);
   \node at (2,-.25) [draw,minimum width=20pt,minimum height=10pt,thick, fill=white] {$\mu_W$};
   \node (g) at (1, 3.5) [draw,minimum width=10pt,minimum height=10pt,thick, fill=white] {$g$};
   \draw(1,2.25) .. controls (1,1.75) and (2,2) .. (2,1.5) -- (2,0);
   \draw[white, double=black, line width = 3pt ] (2,2.25) .. controls (2,1.75) and (1,2) .. (1,1.5) .. controls (1,0.8) and (0,0.8) .. (0,1.5) -- (0,4.5) .. controls (0,5) .. (0.55,5.28);
   \draw[dashed] (2,0.4) .. controls (1,.4) .. (0.5,1);
   \draw (1,3.28) -- (1,3) .. controls (1,2.5) and (2,2.75) .. (2,2.25);
   \draw[white, double=black, line width = 3pt ] (1.7, 4.3) .. controls (2, 4) .. (2,3.5) -- (2,3) .. controls (2,2.5) and (1,2.75) .. (1,2.25);
   \draw(1, 3.74) .. controls (1, 4) .. (1.3, 4.3);
   \node (mu1) at (1.5, 4.5) [draw,minimum width=20pt,minimum height=10pt,thick, fill=white] {$\mu_W$};
   \node (mu2) at (.75, 5.5) [draw,minimum width=20pt,minimum height=10pt,thick, fill=white] {$\mu_W$};
   \draw (1.5, 4.72) .. controls (1.5,5) .. (.95, 5.28);
   \draw (.75,5.72) -- (.75,6.25);
   \node at (.75, 6.5) {$W$};
   \node at (-.1,.9) {$V$};
   \node at (1.1,.9) {$V$};
  \end{tikzpicture}
 \end{matrix}
\end{align}
 The third equality uses both the associativity and commutativity of $\mu_V$, and the last step uses naturality of braiding and unit isomorphisms to move the first $\mu_W$.

 \bigskip
 
\noindent\textbf{2. For each $g\in G$, the image $\Pi_g(W)$ is a $g$-twisted $V$-module.} Since $\Pi_g$ is an even morphism in $\rep\,V$, $\Pi_g(W)$ is an object of $\rep\,V$: it is the kernel of the (even) cokernel of $\Pi_g$. Then $\Pi_g(W)$ will be $g$-twisted if
\begin{equation*}
 \mu_W(g\tens 1_W)\cM_{V,W}(1_V\tens\Pi_g)=\mu_W(1_V\tens\Pi_g).
\end{equation*}
By naturality of the monodromy isomorphisms, the left side is $\mu_W(g\tens \Pi_g)\cM_{V,W}$, which we analyze as follows:
\begin{align*}
 \begin{matrix}
  \begin{tikzpicture}[scale = 1, baseline = {(current bounding box.center)}, line width=0.75pt]
  \draw(3,.75) .. controls (3,1.25) and (0,1) .. (0,1.5) -- (0,4.8);
  \draw (2,4.8) -- (2,4.5) .. controls (2,4) and (3,4.25) .. (3,3.75);
   \draw[white, double=black, line width = 3pt ] (3,0) .. controls (3,.5) and (0,.25) .. (0,.75) .. controls (0,1.25) and (3,1) .. (3,1.5) -- (3,3) .. controls (3,3.5) and (2,3.25) .. (2, 3.75) .. controls (2,4.25) and (3,4) .. (3, 4.5) -- (3,5) .. controls (3,5.5) .. (2.7,5.8);
   \draw[white, double=black, line width = 3pt ] (3,3.75) .. controls (3,3.25) and (2,3.5) .. (2,3) .. controls (2,2.28) and (1,2.28) .. (1,3) -- (1,6) .. controls (1,6.5) .. (1.55,6.8);
   \draw[white, double=black, line width = 3pt ] (0,0) .. controls (0,.5) and (3,.25) .. (3,.75);
   \draw (2,5.2) .. controls (2,5.5) .. (2.3,5.8);
   \draw (2.5,6.2) .. controls (2.5,6.5) .. (1.95,6.8);
   \draw (0,5.2) -- (0,7) .. controls (0,7.5) .. (.7,7.8);
   \draw (1.75,7.2) .. controls (1.75,7.5) .. (1.1,7.8);
   \draw (.9,8.2) -- (.9,8.75);
   \node at (0,5) [draw,minimum width=10pt,minimum height=10pt,thick, fill=white] {$g$};
   \node at (2,5) [draw,minimum width=10pt,minimum height=10pt,thick, fill=white] {$g$};
   \node at (2.5,6) [draw,minimum width=20pt,minimum height=10pt,thick, fill=white] {$\mu_W$};
   \node at (1.75,7) [draw,minimum width=20pt,minimum height=10pt,thick, fill=white] {$\mu_W$};
   \node at (0.9,8) [draw,minimum width=20pt,minimum height=10pt,thick, fill=white] {$\mu_W$};
   \node at (0.9,9) {$W$};
   \node at (0,-.25) {$V$};
   \node at (3,-.25) {$W$};
   \node at (.9,2.5) {$V$};
   \node at (2.1,2.5) {$V$};
   \draw[dashed] (3,1.9) .. controls (2,1.9) .. (1.5,2.48);
  \end{tikzpicture}
 \end{matrix} \hspace{1em} = \hspace{-2.5em}
 \begin{matrix}
  \begin{tikzpicture}[scale = 1, baseline = {(current bounding box.center)}, line width=0.75pt]
  \draw(3,.75) .. controls (3,1.25) and (0,1) .. (0,1.5) -- (0,4.3);
  \draw (2,4.3) -- (2,4) .. controls (2,3.7) and (3,3.8) .. (3,3.5);
   \draw[white, double=black, line width = 3pt ] (3,0) .. controls (3,.5) and (0,.25) .. (0,.75) .. controls (0,1.25) and (3,1) .. (3,1.5) -- (3,3) .. controls (3,3.3) and (2,3.2) .. (2, 3.5) .. controls (2,3.8) and (3,3.7) .. (3, 4) -- (3,6) .. controls (3,6.5) .. (2.45,6.8);
   \draw (2,4.6) -- (2,4.8) .. controls (2,5.1) and (1,5) .. (1,5.3) .. controls (1,5.6) .. (1.3,5.8);
   \draw[white, double=black, line width = 3pt ] (3,3.5) .. controls (3,3.2) and (2,3.3) .. (2,3) .. controls (2,2.28) and (1,2.28) .. (1,3) -- (1,4.8) .. controls (1,5.1) and (2,5) .. (2,5.3) .. controls (2,5.6) .. (1.7,5.8);
   \draw[white, double=black, line width = 3pt ] (0,0) .. controls (0,.5) and (3,.25) .. (3,.75);
   \draw (1.5,6.2) .. controls (1.5,6.5) .. (2.15,6.8);
   \draw (0,4.6) -- (0,7) .. controls (0,7.5) .. (.9,7.8);
   \draw (2.25,7.2) .. controls (2.25,7.5) .. (1.3,7.8);
   \draw (1.1,8.2) -- (1.1,8.75);
   \node at (0,4.4) [draw,minimum width=10pt,minimum height=10pt,thick, fill=white] {$g$};
   \node at (2,4.4) [draw,minimum width=10pt,minimum height=10pt,thick, fill=white] {$g$};
   \node at (1.5,6) [draw,minimum width=20pt,minimum height=10pt,thick, fill=white] {$\mu_V$};
   \node at (2.25,7) [draw,minimum width=20pt,minimum height=10pt,thick, fill=white] {$\mu_W$};
   \node at (1.1,8) [draw,minimum width=20pt,minimum height=10pt,thick, fill=white] {$\mu_W$};
   \node at (1.1,9) {$W$};
   \node at (0,-.25) {$V$};
   \node at (3,-.25) {$W$};
   \node at (.9,2.5) {$V$};
   \node at (2.1,2.5) {$V$};
   \draw[dashed] (0,1.8) .. controls (-2, 2.3) and (2,1.9) .. (1.5,2.5);
  \end{tikzpicture}
 \end{matrix} \hspace{1em} = \hspace{1em}
 \begin{matrix}
  \begin{tikzpicture}[scale = 1, baseline = {(current bounding box.center)}, line width=0.75pt]
  \draw (3,.75) .. controls (3,1.25) and (2,1) .. (2,1.5) .. controls (2,2) and (1,1.75) .. (1,2.25) .. controls (1,2.75) and (0,2.5) .. (0,3) -- (0,4.8);
  \draw (3,3) .. controls (3,3.5) and (2,3.25) .. (2,3.75) .. controls (2,4.25) and (1,4) .. (1,4.5) -- (1,4.8);
  \draw[white, double=black, line width = 3pt ] (3,0) .. controls (3,.5) and (2,.25) .. (2,.75) .. controls (2,1.25) and (3,1) .. (3,1.5) -- (3,2.25) .. controls (3,2.75) and (2,2.5) .. (2,3) .. controls (2,3.5) and (3,3.25) .. (3,3.75) -- (3,7) .. controls (3,7.5) .. (2.3,7.8);
  \draw[white, double=black, line width = 3pt ] (2,0) .. controls (2,.5) and (3,.25) .. (3,.75);
  \draw[white, double=black, line width = 3pt ] (3,3) .. controls (3,2.5) and (2,2.75) .. (2,2.25) .. controls (2,1.75) and (1,2) .. (1,1.5) .. controls (1,.8) and (0,.8) .. (0,1.5) -- (0,2.25) .. controls (0,2.75) and (1,2.5) ..(1,3) -- (1,3.75) .. controls (1,4.25) and (2,4) .. (2,4.5) -- (2,6) .. controls (2,6.5) .. (1.45,6.8);
  \draw (0,5.2) .. controls (0,5.5) .. (.3,5.8);
  \draw (1,5.2) .. controls (1,5.5) .. (.7,5.8);
  \draw (.5,6.2) .. controls (.5,6.5) .. (1.05,6.8);
  \draw (1.25,7.2) .. controls (1.25, 7.5) .. (1.9,7.8);
  \draw (2.1,8.2) -- (2.1,8.75);
   \node at (0,5) [draw,minimum width=10pt,minimum height=10pt,thick, fill=white] {$g$};
   \node at (1,5) [draw,minimum width=10pt,minimum height=10pt,thick, fill=white] {$g$};
   \node at (.5,6) [draw,minimum width=20pt,minimum height=10pt,thick, fill=white] {$\mu_V$};
   \node at (1.25,7) [draw,minimum width=20pt,minimum height=10pt,thick, fill=white] {$\mu_V$};
   \node at (2.1,8) [draw,minimum width=20pt,minimum height=10pt,thick, fill=white] {$\mu_W$};
   \node at (2.1,9) {$W$};
   \draw[dashed] (2,0) .. controls (.75,.25) .. (.5,1);
   \node at (2,-.25) {$V$};
   \node at (3,-.25) {$W$};
   \node at (0,.9) {$V$};
   \node at (1,.9) {$V$};
  \end{tikzpicture}
 \end{matrix}
\end{align*}
\begin{align}\label{eqn:Im_of_Pi_g_twisted_1}
 = \hspace{1em} \begin{matrix}
   \begin{tikzpicture}[scale = 1, baseline = {(current bounding box.center)}, line width=0.75pt]
    \draw (3,.75) .. controls (3,1.25) and (2,1) .. (2,1.5) .. controls (2,2) and (1,1.75) .. (1,2.25) -- (1,3.75) .. controls (1,3.9) .. (1.3,4);
  \draw (3,3) .. controls (3,3.5) and (2,3.25) .. (2,3.75) .. controls (2,3.9) .. (1.7,4);
  \draw[white, double=black, line width = 3pt ] (3,0) .. controls (3,.5) and (2,.25) .. (2,.75) .. controls (2,1.25) and (3,1) .. (3,1.5) -- (3,2.25) .. controls (3,2.75) and (2,2.5) .. (2,3) .. controls (2,3.5) and (3,3.25) .. (3,3.75) -- (3,7) .. controls (3,7.5) .. (2.1,7.8);
  \draw[white, double=black, line width = 3pt ] (2,0) .. controls (2,.5) and (3,.25) .. (3,.75);
  \draw (1.5,5.2) -- (1.5,5.5) .. controls (1.5,6) and (0,5.75) .. (0,6.25) .. controls (0,6.55) .. (.55,6.8);
  \draw[white, double=black, line width = 3pt ] (3,3) .. controls (3,2.5) and (2,2.75) .. (2,2.25) .. controls (2,1.75) and (1,2) .. (1,1.5) .. controls (1,.8) and (0,.8) .. (0,1.5) -- (0,5.5) .. controls (0,6) and (1.5,5.75) ..(1.5,6.25) .. controls (1.5,6.55) .. (.95,6.8);
  \draw (.75,7.2) .. controls (.75, 7.5) .. (1.7,7.8);
  \draw (1.9,8.2) -- (1.9,8.75);
  \draw (1.5,4.4) -- (1.5,4.8);
   \node at (1.5,5) [draw,minimum width=10pt,minimum height=10pt,thick, fill=white] {$g$};
   \node at (1.5,4.2) [draw,minimum width=20pt,minimum height=10pt,thick, fill=white] {$\mu_V$};
   \node at (.75,7) [draw,minimum width=20pt,minimum height=10pt,thick, fill=white] {$\mu_V$};
   \node at (1.9,8) [draw,minimum width=20pt,minimum height=10pt,thick, fill=white] {$\mu_W$};
   \node at (1.9,9) {$W$};
   \draw[dashed] (2,0) .. controls (.75,.25) .. (.5,1);
   \node at (2,-.25) {$V$};
   \node at (3,-.25) {$W$};
   \node at (0,.9) {$V$};
   \node at (1,.9) {$V$};
   \end{tikzpicture}
  \end{matrix} \hspace{1em} = \hspace{1em}
  \begin{matrix}
   \begin{tikzpicture}[scale = 1, baseline = {(current bounding box.center)}, line width=0.75pt]
    \draw (3,.75) .. controls (3,1.25) and (2,1) .. (2,1.5) .. controls (2,2) and (1,1.75) .. (1,2.25) -- (1,3.75) .. controls (1,4.1) .. (1.3,4.3);
  \draw (3,3) .. controls (3,3.5) and (2,3.25) .. (2,3.75) .. controls (2,4.1) .. (1.7,4.3);
  \draw[white, double=black, line width = 3pt ] (3,0) .. controls (3,.5) and (2,.25) .. (2,.75) .. controls (2,1.25) and (3,1) .. (3,1.5) -- (3,2.25) .. controls (3,2.75) and (2,2.5) .. (2,3) .. controls (2,3.5) and (3,3.25) .. (3,3.75) -- (3,7) .. controls (3,7.5) .. (2.1,7.8);
  \draw[white, double=black, line width = 3pt ] (2,0) .. controls (2,.5) and (3,.25) .. (3,.75);
  \draw (1.5,5.7) .. controls (1.5,6.5) .. (.95,6.8);
  \draw[white, double=black, line width = 3pt ] (3,3) .. controls (3,2.5) and (2,2.75) .. (2,2.25) .. controls (2,1.75) and (1,2) .. (1,1.5) .. controls (1,.8) and (0,.8) .. (0,1.5) -- (0,5.5) .. controls (0,6.5) .. (.55,6.8);
  \draw (.75,7.2) .. controls (.75, 7.5) .. (1.7,7.8);
  \draw (1.9,8.2) -- (1.9,8.75);
  \draw (1.5,4.7) -- (1.5,5.3);
   \node at (1.5,5.5) [draw,minimum width=10pt,minimum height=10pt,thick, fill=white] {$g$};
   \node at (1.5,4.5) [draw,minimum width=20pt,minimum height=10pt,thick, fill=white] {$\mu_V$};
   \node at (.75,7) [draw,minimum width=20pt,minimum height=10pt,thick, fill=white] {$\mu_V$};
   \node at (1.9,8) [draw,minimum width=20pt,minimum height=10pt,thick, fill=white] {$\mu_W$};
   \node at (1.9,9) {$W$};
   \draw[dashed] (2,0) .. controls (.75,.25) .. (.5,1);
   \node at (2,-.25) {$V$};
   \node at (3,-.25) {$W$};
   \node at (0,.9) {$V$};
   \node at (1,.9) {$V$};
   \end{tikzpicture}
  \end{matrix} .
\end{align}
We simplify the braidings here with the Yang-Baxter relation, the commutativity of $\mu_V$, the hexagon axioms, and the naturality of braiding:
\begin{align}\label{eqn:Im_of_Pi_g_twisted_2}
 \begin{matrix}
  \begin{tikzpicture}[scale = 1, baseline = {(current bounding box.center)}, line width=0.75pt]
  \draw (2,.75) .. controls (2,1.25) and (1,1) .. (1,1.5) .. controls (1,2) and (0,1.75) .. (0,2.25) -- (0,3.75) .. controls (0,4) .. (.3,4.3);
  \draw (2,3) .. controls (2,3.5) and (1,3.25) .. (1,3.75) .. controls (1,4) .. (.7,4.3);
   \draw[white, double=black, line width = 3pt ] (2, 0) .. controls (2,.5) and (1,.25) .. (1,.75) .. controls (1,1.25) and (2,1) ..(2,1.5) -- (2,2.25) .. controls (2,2.75) and (1,2.5) .. (1,3) .. controls (1,3.5) and (2,3.25) .. (2,3.75) -- (2,5.15);
   \draw[white, double=black, line width = 3pt ] (1,0) .. controls (1,.5) and (2,.25) .. (2,.75);
   \draw[white, double=black, line width = 3pt ] (0,0) -- (0,1.5) .. controls (0,2) and (1,1.75) .. (1,2.25) .. controls (1,2.75) and (2,2.5) .. (2,3);
   \draw (.5,4.7)--(.5,5.15);
   \node at (.5,4.5) [draw,minimum width=20pt,minimum height=10pt,thick, fill=white] {$\mu_V$};
   \node at (0,-.25) {$V$};
   \node at (1,-.25) {$V$};
   \node at (2,-.25) {$W$};
   \node at (.5,5.4) {$V$};
   \node at (2,5.4) {$W$};
  \end{tikzpicture}
  \end{matrix} =
  \begin{matrix}
  \begin{tikzpicture}[scale = 1, baseline = {(current bounding box.center)}, line width=0.75pt]
  \draw (2,.75) -- (2,1.5) .. controls (2,2) and (1,1.75) .. (1,2.25) .. controls (1,2.75) and (0,2.5) .. (0,3) -- (0,3.75).. controls (0,4) .. (.3,4.3);
  \draw (2,3) .. controls (2,3.5) and (1,3.25) .. (1,3.75) .. controls (1,4) .. (.7,4.3);
   \draw[white, double=black, line width = 3pt ] (2, 0) .. controls (2,.5) and (1,.25) .. (1,.75) .. controls (1,1.25) and (0,1) .. (0,1.5) -- (0,2.25) .. controls (0,2.75) and (1,2.5) .. (1,3).. controls (1,3.5) and (2,3.25) .. (2,3.75) -- (2,5.15);
   \draw[white, double=black, line width = 3pt ] (1,0) .. controls (1,.5) and (2,.25) .. (2,.75);
   \draw[white, double=black, line width = 3pt ] (0,0) -- (0,.75) .. controls (0,1.25) and (1,1) .. (1,1.5) .. controls (1,2) and (2,1.75) .. (2,2.25) -- (2,3);
   \draw (.5,4.7)--(.5,5.15);
   \node at (.5,4.5) [draw,minimum width=20pt,minimum height=10pt,thick, fill=white] {$\mu_V$};
   \node at (0,-.25) {$V$};
   \node at (1,-.25) {$V$};
   \node at (2,-.25) {$W$};
   \node at (.5,5.4) {$V$};
   \node at (2,5.4) {$W$};
  \end{tikzpicture}
  \end{matrix} =
  \begin{matrix}
  \begin{tikzpicture}[scale = 1, baseline = {(current bounding box.center)}, line width=0.75pt]
  \draw (2,.75) -- (2,2.25) .. controls (2,2.75) and (1,2.5) .. (1,3) .. controls (1,3.5) and (0,3.25) .. (0,3.75) .. controls (0,4) .. (.3,4.3);
  \draw[white, double=black, line width = 3pt ] (1,1.5) .. controls (1,2) and (0,1.75) .. (0,2.25) -- (0,3) .. controls (0,3.5) and (1,3.25) .. (1,3.75) .. controls (1,4) .. (.7,4.3);
   \draw[white, double=black, line width = 3pt ] (2, 0) .. controls (2,.5) and (1,.25) .. (1,.75) .. controls (1,1.25) and (0,1) .. (0,1.5) .. controls (0,2) and (1,1.75) .. (1,2.25).. controls (1,2.75) and (2,2.5) .. (2,3) -- (2,5.15);
   \draw[white, double=black, line width = 3pt ] (1,0) .. controls (1,.5) and (2,.25) .. (2,.75);
   \draw[white, double=black, line width = 3pt ] (0,0) -- (0,.75) .. controls (0,1.25) and (1,1) .. (1,1.5);
   \draw (.5,4.7)--(.5,5.15);
   \node at (.5,4.5) [draw,minimum width=20pt,minimum height=10pt,thick, fill=white] {$\mu_V$};
   \node at (0,-.25) {$V$};
   \node at (1,-.25) {$V$};
   \node at (2,-.25) {$W$};
   \node at (.5,5.4) {$V$};
   \node at (2,5.4) {$W$};
  \end{tikzpicture}
  \end{matrix} =
  \begin{matrix}
  \begin{tikzpicture}[scale = 1, baseline = {(current bounding box.center)}, line width=0.75pt]
  \draw (2,.75) -- (2,2.25) .. controls (2,2.75) and (1,2.5) .. (1,3) .. controls (1,3.25) .. (.7,3.55);
  \draw[white, double=black, line width = 3pt ] (1,1.5) .. controls (1,2) and (0,1.75) .. (0,2.25) -- (0,3) .. controls (0,3.25) .. (.3,3.55);
   \draw[white, double=black, line width = 3pt ] (2, 0) .. controls (2,.5) and (1,.25) .. (1,.75) .. controls (1,1.25) and (0,1) .. (0,1.5) .. controls (0,2) and (1,1.75) .. (1,2.25).. controls (1,2.75) and (2,2.5) .. (2,3) -- (2,4.5);
   \draw[white, double=black, line width = 3pt ] (1,0) .. controls (1,.5) and (2,.25) .. (2,.75);
   \draw[white, double=black, line width = 3pt ] (0,0) -- (0,.75) .. controls (0,1.25) and (1,1) .. (1,1.5);
   \draw (.5,3.95)--(.5,4.5);
   \node at (.5,3.75) [draw,minimum width=20pt,minimum height=10pt,thick, fill=white] {$\mu_V$};
   \node at (0,-.25) {$V$};
   \node at (1,-.25) {$V$};
   \node at (2,-.25) {$W$};
   \node at (.5,4.75) {$V$};
   \node at (2,4.75) {$W$};
  \end{tikzpicture}
  \end{matrix} =
  \begin{matrix}
  \begin{tikzpicture}[scale = 1, baseline = {(current bounding box.center)}, line width=0.75pt]
  \draw (0,0) .. controls (0,.25) .. (.3,.55);
  \draw (1,0) .. controls (1,.25) .. (.7,.55);
  \draw (2,2.25) .. controls (2,2.75) and (.5,2.5) .. (.5,3);
   \draw[white, double=black, line width = 3pt ] (2,0) -- (2,1.5) .. controls (2,2) and (.5,1.75) .. (.5,2.25) .. controls (.5,2.75) and (2,2.5) .. (2,3); 
   \draw[white, double=black, line width = 3pt ] (.5,.95) -- (.5,1.5) .. controls (.5,2) and (2,1.75) .. (2,2.25);
   \node at (.5,.75) [draw,minimum width=20pt,minimum height=10pt,thick, fill=white] {$\mu_V$};
   \node at (0,-.25) {$V$};
   \node at (1,-.25) {$V$};
   \node at (2,-.25) {$W$};
   \node at (.5,3.25) {$V$};
   \node at (2,3.25) {$W$};
  \end{tikzpicture}
  \end{matrix}
  \end{align}
Finally, we insert \eqref{eqn:Im_of_Pi_g_twisted_2} back into \eqref{eqn:Im_of_Pi_g_twisted_1} and apply Lemma \ref{rigidlike_lemma}:
\begin{align}\label{eqn:Im_of_Pi_g_twisted_3}
 \begin{matrix}
    \begin{tikzpicture}[scale = 1, baseline = {(current bounding box.center)}, line width=0.75pt]
    \draw[white, double=black, line width = 3pt ] (3,3.75) .. controls (3,4.25) and (1.5,4) .. (1.5,4.5) -- (1.5,4.8);
     \draw[white, double=black, line width = 3pt ] (3,0) -- (3,3) .. controls (3,3.5) and (1.5,3.25) .. (1.5,3.75) .. controls (1.5,4.25) and (3,4) .. (3,4.5) -- (3,6) .. controls (3,6.5) .. (2.1,6.8);
     \draw[white, double=black, line width = 3pt ] (1.5,2.45) -- (1.5,3) .. controls (1.5,3.5) and (3,3.25) .. (3,3.75);
     \draw (2,0) -- (2,1.5) .. controls (2,1.8) .. (1.7,2.05);
     \draw (1.3,2.05) .. controls (1,1.8) .. (1,1.5) .. controls (1,.8) and (0,.8) .. (0,1.5) -- (0,5) .. controls (0,5.5) .. (.55,5.8);
     \draw (1.5,5.2) .. controls (1.5,5.5) .. (.95,5.8);
     \draw (.75,6.2) .. controls (.75,6.5) .. (1.7,6.8);
     \draw (1.9,7.2) -- (1.9,7.7);
     \draw[dashed] (2,.4) .. controls (.5,.5) .. (.5,1);
     \node at (1.9,7) [draw,minimum width=20pt,minimum height=10pt,thick, fill=white] {$\mu_W$};
     \node at (.75,6) [draw,minimum width=20pt,minimum height=10pt,thick, fill=white] {$\mu_V$};
     \node at (1.5,2.25) [draw,minimum width=20pt,minimum height=10pt,thick, fill=white] {$\mu_V$};
     \node at (1.5,5) [draw,minimum width=10pt,minimum height=10pt,thick, fill=white] {$g$};
     \node at (1.9,7.95) {$W$};
     \node at (2,-.25) {$V$};
     \node at (3,-.25) {$W$};
     \node at (0,.9) {$V$};
     \node at (1,.9) {$V$};
    \end{tikzpicture}
   \end{matrix} =
   \begin{matrix}
    \begin{tikzpicture}[scale = 1, baseline = {(current bounding box.center)}, line width=0.75pt]
    \draw[white, double=black, line width = 3pt ] (3,3.75) .. controls (3,4.25) and (2,4) .. (2,4.5) -- (2,4.8);
     \draw[white, double=black, line width = 3pt ] (3,0) -- (3,3) .. controls (3,3.5) and (2,3.25) .. (2,3.75) .. controls (2,4.25) and (3,4) .. (3,4.5) -- (3,6) .. controls (3,6.5) .. (2.3,6.8);
     \draw[white, double=black, line width = 3pt ] (.7,2.05) .. controls (1,1.8) .. (1,1.5) .. controls (1,.8) and (2,.8) .. (2,1.5) -- (2,3) .. controls (2,3.5) and (3,3.25) .. (3,3.75);
     \draw (0,0) -- (0,1.5) .. controls (0,1.8) .. (.3,2.05);
     \draw (.5,2.45) -- (.5,5) .. controls (.5,5.5) .. (1.05,5.8);
     \draw (2,5.2) .. controls (2,5.5) .. (1.45,5.8);
     \draw (1.25,6.2) .. controls (1.25,6.5) .. (1.9,6.8);
     \draw (2.1,7.2) -- (2.1,7.7);
     \draw[dashed] (0,.4) .. controls (1.5,.5) .. (1.5,1);
     \node at (2.1,7) [draw,minimum width=20pt,minimum height=10pt,thick, fill=white] {$\mu_W$};
     \node at (1.25,6) [draw,minimum width=20pt,minimum height=10pt,thick, fill=white] {$\mu_V$};
     \node at (.5,2.25) [draw,minimum width=20pt,minimum height=10pt,thick, fill=white] {$\mu_V$};
     \node at (2,5) [draw,minimum width=10pt,minimum height=10pt,thick, fill=white] {$g$};
     \node at (2.1,7.95) {$W$};
     \node at (0,-.25) {$V$};
     \node at (3,-.25) {$W$};
     \node at (1,.9) {$V$};
     \node at (2,.9) {$V$};
    \end{tikzpicture}
   \end{matrix} = 
   \begin{matrix}
    \begin{tikzpicture}[scale = 1, baseline = {(current bounding box.center)}, line width=0.75pt]
    \draw[white, double=black, line width = 3pt ] (3,2.75) .. controls (3,3.25) and (2,3) .. (2,3.5) -- (2,3.8);
     \draw[white, double=black, line width = 3pt ] (3,0) -- (3,2) .. controls (3,2.5) and (2,2.25) .. (2,2.75) .. controls (2,3.25) and (3,3) .. (3,3.5) -- (3,6) .. controls (3,6.5) .. (2.3,6.8);
     \draw[white, double=black, line width = 3pt ] (.7,4.8) .. controls (1,4.5) .. (1,4) -- (1,1.5) .. controls (1,.8) and (2,.8) .. (2,1.5) -- (2,2) .. controls (2,2.5) and (3,2.25) .. (3,2.75);
     \draw (0,0) -- (0,4) .. controls (0,4.5) .. (.3,4.8);
     \draw (.5,5.2) .. controls (.5,5.5) .. (1.05,5.8);
     \draw (2,4.2) -- (2,5) .. controls (2,5.5) .. (1.45,5.8);
     \draw (1.25,6.2) .. controls (1.25,6.5) .. (1.9,6.8);
     \draw (2.1,7.2) -- (2.1,7.7);
     \draw[dashed] (3,.4) .. controls (1.5,.5) .. (1.5,1);
     \node at (2.1,7) [draw,minimum width=20pt,minimum height=10pt,thick, fill=white] {$\mu_W$};
     \node at (1.25,6) [draw,minimum width=20pt,minimum height=10pt,thick, fill=white] {$\mu_V$};
     \node at (.5,5) [draw,minimum width=20pt,minimum height=10pt,thick, fill=white] {$\mu_V$};
     \node at (2,4) [draw,minimum width=10pt,minimum height=10pt,thick, fill=white] {$g$};
     \node at (2.1,7.95) {$W$};
     \node at (0,-.25) {$V$};
     \node at (3,-.25) {$W$};
     \node at (1,.9) {$V$};
     \node at (2,.9) {$V$};
    \end{tikzpicture}
   \end{matrix} = 
   \begin{matrix}
    \begin{tikzpicture}[scale = 1, baseline = {(current bounding box.center)}, line width=0.75pt]
    \draw[white, double=black, line width = 3pt ] (3,2.75) .. controls (3,3.25) and (2,3) .. (2,3.5) -- (2,3.8);
     \draw[white, double=black, line width = 3pt ] (3,0) -- (3,2) .. controls (3,2.5) and (2,2.25) .. (2,2.75) .. controls (2,3.25) and (3,3) .. (3,3.5) -- (3,4) .. controls (3,4.5) .. (2.7,4.8);
     \draw[white, double=black, line width = 3pt ] (1.55,5.8) .. controls (1,5.5) .. (1,5) -- (1,1.5) .. controls (1,.8) and (2,.8) .. (2,1.5) -- (2,2) .. controls (2,2.5) and (3,2.25) .. (3,2.75);
     \draw (0,0) -- (0,6) .. controls (0,6.5) .. (.7,6.8);
     \draw (2.5,5.2) .. controls (2.5,5.5) .. (1.95,5.8);
     \draw (2,4.2) .. controls (2,4.5) .. (2.3,4.8);
     \draw (1.75,6.2) .. controls (1.75,6.5) .. (1.1,6.8);
     \draw (.9,7.2) -- (.9,7.7);
     \draw[dashed] (3,.4) .. controls (1.5,.5) .. (1.5,1);
     \node at (.9,7) [draw,minimum width=20pt,minimum height=10pt,thick, fill=white] {$\mu_W$};
     \node at (1.75,6) [draw,minimum width=20pt,minimum height=10pt,thick, fill=white] {$\mu_W$};
     \node at (2.5,5) [draw,minimum width=20pt,minimum height=10pt,thick, fill=white] {$\mu_W$};
     \node at (2,4) [draw,minimum width=10pt,minimum height=10pt,thick, fill=white] {$g$};
     \node at (.9,7.95) {$W$};
     \node at (0,-.25) {$V$};
     \node at (3,-.25) {$W$};
     \node at (1,.9) {$V$};
     \node at (2,.9) {$V$};
    \end{tikzpicture}
   \end{matrix},
\end{align}
which is $\mu_W(1_V\tens\Pi_g)$.

\bigskip

\noindent\textbf{3. For all $g,h\in G$, $\Pi_g\Pi_h=\vert G\vert\delta_{g,h}\Pi_h$, and $\sum_{g\in G} \pi_g= 1_W$.} Since we have just shown that $\Pi_h(W)$ is an $h$-twisted module for any $h\in G$, for the first relation it is enough to prove that
\begin{equation*}
 \Pi_g=\vert G\vert\delta_{g,h} 1_W
\end{equation*}
when $W$ is an $h$-twisted $V$-module. In fact, when $W$ is $h$-twisted, $\Pi_g$ is given by 
\begin{align}\label{eqn:PiG_PiH}
 \begin{matrix}
    \begin{tikzpicture}[scale = 1, baseline = {(current bounding box.center)}, line width=0.75pt]
    \draw[white, double=black, line width = 3pt ] (3,2.25) .. controls (3,2.75) and (2,2.5) .. (2,3) -- (2,3.3);
     \draw[white, double=black, line width = 3pt ] (3,0) -- (3,1.5) .. controls (3,2) and (2,1.75) .. (2,2.25) .. controls (2,2.75) and (3,2.5) .. (3,3) -- (3,3.5) .. controls (3,4) .. (2.7,4.3);
     \draw[white, double=black, line width = 3pt ] (1.55,5.3) .. controls (1,5) .. (1,4.5) -- (1,1.5) .. controls (1,.8) and (2,.8) .. (2,1.5) .. controls (2,2) and (3,1.75) .. (3,2.25);
     \draw (2.5,4.7) .. controls (2.5,5.1) .. (1.95,5.3);
     \draw (2,3.7) .. controls (2,4) .. (2.3,4.3);
     \draw (1.75,5.7) -- (1.75,6.1);
     \draw[dashed] (3,.4) .. controls (1.5,.5) .. (1.5,1);
     \node at (1.75,5.5) [draw,minimum width=20pt,minimum height=10pt,thick, fill=white] {$\mu_W$};
     \node at (2.5,4.5) [draw,minimum width=20pt,minimum height=10pt,thick, fill=white] {$\mu_W$};
     \node at (2,3.5) [draw,minimum width=10pt,minimum height=10pt,thick, fill=white] {$g$};
     \node at (3,-.25) {$W$};
     \node at (1,.9) {$V$};
     \node at (2,.9) {$V$};
     \node at (1.75,6.35) {$W$};
    \end{tikzpicture}
   \end{matrix} =
   \begin{matrix}
    \begin{tikzpicture}[scale = 1, baseline = {(current bounding box.center)}, line width=0.75pt]
    \draw (2.5,4.7) .. controls (2.5,5.1) .. (1.95, 5.3);
    \draw[white, double=black, line width = 3pt ] (3,3.25) .. controls (3,3.75) and (2,3.5) .. (2,4) .. controls (2,4.2) .. (2.3,4.3);
     \draw[white, double=black, line width = 3pt ] (3,0) -- (3,2.5) .. controls (3,3) and (2,2.75) .. (2,3.25) .. controls (2,3.75) and (3,3.5) .. (3,4) .. controls (3,4.2) .. (2.7,4.3);
     \draw[white, double=black, line width = 3pt ] (2,2.2) -- (2,2.5) .. controls (2,3) and (3,2.75) .. (3,3.25);
     \draw[white, double=black, line width = 3pt ] (1.55,5.3) .. controls (1,5) .. (1,4.5) -- (1,1.5) .. controls (1,.8) and (2,.8) .. (2,1.5) -- (2,1.8);
     \draw (1.75,5.7) -- (1.75,6.1);
     \draw[dashed] (3,.4) .. controls (1.5,.5) .. (1.5,1);
     \node at (1.75,5.5) [draw,minimum width=20pt,minimum height=10pt,thick, fill=white] {$\mu_W$};
     \node at (2.5,4.5) [draw,minimum width=20pt,minimum height=10pt,thick, fill=white] {$\mu_W$};
     \node at (2,2) [draw,minimum width=10pt,minimum height=10pt,thick, fill=white] {$g$};
     \node at (3,-.25) {$W$};
     \node at (1,.9) {$V$};
     \node at (2,.9) {$V$};
     \node at (1.75,6.35) {$W$};
    \end{tikzpicture}
   \end{matrix} =
    \begin{matrix}
    \begin{tikzpicture}[scale = 1, baseline = {(current bounding box.center)}, line width=0.75pt]
    \draw (2.5,4.2) .. controls (2.5,4.6) .. (1.95, 4.8);
     \draw[white, double=black, line width = 3pt ] (3,0) -- (3,3.2) .. controls (3,3.6) .. (2.7,3.8);
     \draw[white, double=black, line width = 3pt ] (1.55,4.8) .. controls (1,4.5) .. (1,4) -- (1,1.5) .. controls (1,.8) and (2,.8) .. (2,1.5) -- (2,3.2) .. controls (2,3.6) .. (2.3,3.8);
     \draw (1.75,5.2) -- (1.75,5.6);
     \draw[dashed] (3,.4) .. controls (1.5,.5) .. (1.5,1);
     \node at (1.75,5) [draw,minimum width=20pt,minimum height=10pt,thick, fill=white] {$\mu_W$};
     \node at (2.5,4) [draw,minimum width=20pt,minimum height=10pt,thick, fill=white] {$\mu_W$};
     \node at (2,2) [draw,minimum width=10pt,minimum height=10pt,thick, fill=white] {$g$};
     \node at (2,3) [draw,minimum width=10pt,minimum height=10pt,thick, fill=white] {$h^{-1}$};
     \node at (3,-.25) {$W$};
     \node at (1,.9) {$V$};
     \node at (2,.9) {$V$};
     \node at (1.75,5.85) {$W$};
    \end{tikzpicture}
   \end{matrix} =
   \begin{matrix}
    \begin{tikzpicture}[scale = 1, baseline = {(current bounding box.center)}, line width=0.75pt]
    \draw (1.5,3.2) .. controls (1.5,3.6) .. (2.05, 3.8);
     \draw[white, double=black, line width = 3pt ] (3,0) -- (3,3.2) .. controls (3,3.6) .. (2.45,3.8);
     \draw[white, double=black, line width = 3pt ] (1.3,2.8) .. controls (1,2.5) .. (1,2) -- (1,1.5) .. controls (1,.8) and (2,.8) .. (2,1.5) -- (2,2) .. controls (2,2.5) .. (1.7,2.8);
     \draw (2.25,4.2) -- (2.25,4.6);
     \draw[dashed] (3,.4) .. controls (1.5,.5) .. (1.5,1);
     \node at (2.25,4) [draw,minimum width=20pt,minimum height=10pt,thick, fill=white] {$\mu_W$};
     \node at (1.5,3) [draw,minimum width=20pt,minimum height=10pt,thick, fill=white] {$\mu_V$};
     \node at (2,2) [draw,minimum width=10pt,minimum height=10pt,thick, fill=white] {$h^{-1}g$};
     \node at (3,-.25) {$W$};
     \node at (1,.9) {$V$};
     \node at (2,.9) {$V$};
     \node at (2.25,4.85) {$W$};
    \end{tikzpicture}
   \end{matrix},
\end{align}
which is $\vert G\vert \delta_{g,h} 1_W$ by Lemma \ref{trace_of_g_lemma} and the unit property of $W$.

 Finally we compute $\sum_{g\in G}\pi_g$ using bilinearity of composition and tensor products of morphisms in a tensor category, the assumption $\frac{1}{\vert G\vert}\sum_{g\in G} g =\iota_V\varepsilon_V$, the triviality of $\cM_{\vac,W}$, and the associativity of $\mu_W$:
 \begin{align}\label{eqn:sum_pi_g_identity}
 \begin{matrix}
    \begin{tikzpicture}[scale = 1, baseline = {(current bounding box.center)}, line width=0.75pt]
    \draw[white, double=black, line width = 3pt ] (4.5,2.25) .. controls (4.5,2.75) and (3,2.5) .. (3,3) -- (3,3.3);
     \draw[white, double=black, line width = 3pt ] (4.5,0) -- (4.5,1.5) .. controls (4.5,2) and (3,1.75) .. (3,2.25) .. controls (3,2.75) and (4.5,2.5) .. (4.5,3) -- (4.5,3.5) .. controls (4.5,4) .. (4.05,4.3);
     \draw[white, double=black, line width = 3pt ] (2.325,5.3) .. controls (1.5,5) .. (1.5,4.5) -- (1.5,1.5) .. controls (1.5,.8) and (3,.8) .. (3,1.5) .. controls (3,2) and (4.5,1.75) .. (4.5,2.25);
     \draw (3.75,4.7) .. controls (3.75,5.1) .. (2.925,5.3);
     \draw (3,3.7) .. controls (3,4) .. (3.45,4.3);
     \draw (2.625,5.7) -- (2.625,6.1);
     \draw[dashed] (4.5,.4) .. controls (2.25,.5) .. (2.25,1);
     \node at (2.625,5.5) [draw,minimum width=20pt,minimum height=10pt,thick, fill=white] {$\mu_W$};
     \node at (3.75,4.5) [draw,minimum width=20pt,minimum height=10pt,thick, fill=white] {$\mu_W$};
     \node at (3,3.5) [draw,minimum width=10pt,minimum height=10pt,thick, fill=white] {$\frac{1}{\vert G\vert}\sum_{g\in G} g$};
     \node at (4,-.25) {$W$};
     \node at (1.5,.9) {$V$};
     \node at (3,.9) {$V$};
     \node at (2.625,6.35) {$W$};
    \end{tikzpicture}
   \end{matrix} =
 \begin{matrix}
    \begin{tikzpicture}[scale = 1, baseline = {(current bounding box.center)}, line width=0.75pt]
    \draw[white, double=black, line width = 3pt ] (3,2.25) .. controls (3,2.75) and (2,2.5) .. (2,3) -- (2,3.3);
     \draw[white, double=black, line width = 3pt ] (3,0) -- (3,1.5) .. controls (3,2) and (2,1.75) .. (2,2.25) .. controls (2,2.75) and (3,2.5) .. (3,3) -- (3,3.5) .. controls (3,4) .. (2.7,4.3);
     \draw[white, double=black, line width = 3pt ] (1.55,5.3) .. controls (1,5) .. (1,4.5) -- (1,1.5) .. controls (1,.8) and (2,.8) .. (2,1.5) .. controls (2,2) and (3,1.75) .. (3,2.25);
     \draw (2.5,4.7) .. controls (2.5,5.1) .. (1.95,5.3);
     \draw (2,3.7) .. controls (2,4) .. (2.3,4.3);
     \draw (1.75,5.7) -- (1.75,6.1);
     \draw[dashed] (3,.4) .. controls (1.5,.5) .. (1.5,1);
     \node at (1.75,5.5) [draw,minimum width=20pt,minimum height=10pt,thick, fill=white] {$\mu_W$};
     \node at (2.5,4.5) [draw,minimum width=20pt,minimum height=10pt,thick, fill=white] {$\mu_W$};
     \node at (2,3.5) [draw,minimum width=10pt,minimum height=10pt,thick, fill=white] {$\iota_V\varepsilon_V$};
     \node at (3,-.25) {$W$};
     \node at (1,.9) {$V$};
     \node at (2,.9) {$V$};
     \node at (1.75,6.35) {$W$};
    \end{tikzpicture}
   \end{matrix} =
 \begin{matrix}
    \begin{tikzpicture}[scale = 1, baseline = {(current bounding box.center)}, line width=0.75pt]
     \draw[white, double=black, line width = 3pt ] (3,0) -- (3,2) .. controls (3,2.5) and (2,2.25) .. (2,2.75) .. controls (2,3.25) and (3,3) .. (3,3.5) .. controls (3,4) .. (2.7,4.3);
     \draw[white, double=black, line width = 3pt ] (1.55,5.3) .. controls (1,5) .. (1,4.5) -- (1,1.5) .. controls (1,.8) and (2,.8) .. (2,1.5) -- (2,1.8);
     \draw[dashed] (2,2) .. controls (2,2.5) and (3,2.25) .. (3,2.75) .. controls (3,3.25) and (2,3) .. (2,3.5);
     \draw (2.5,4.7) .. controls (2.5,5.1) .. (1.95,5.3);
     \draw (2,3.7) .. controls (2,4) .. (2.3,4.3);
     \draw (1.75,5.7) -- (1.75,6.1);
     \draw[dashed] (3,.4) .. controls (1.5,.5) .. (1.5,1);
     \node at (1.75,5.5) [draw,minimum width=20pt,minimum height=10pt,thick, fill=white] {$\mu_W$};
     \node at (2.5,4.5) [draw,minimum width=20pt,minimum height=10pt,thick, fill=white] {$\mu_W$};
     \node at (2,3.7) [draw,minimum width=10pt,minimum height=10pt,thick, fill=white] {$\iota_V$};
     \node at (2,1.8) [draw,minimum width=10pt,minimum height=10pt,thick, fill=white] {$\varepsilon_V$};
     \node at (3,-.25) {$W$};
     \node at (1,.9) {$V$};
     \node at (2,.9) {$V$};
     \node at (1.75,6.35) {$W$};
    \end{tikzpicture}
   \end{matrix} =
   \begin{matrix}
    \begin{tikzpicture}[scale = 1, baseline = {(current bounding box.center)}, line width=0.75pt]
     \draw[white, double=black, line width = 3pt ] (3,0) -- (3,4) .. controls (3,4.5) .. (2.45,4.8);
     \draw[white, double=black, line width = 3pt ] (1.3,3.8) .. controls (1,3.5) .. (1,3) -- (1,1.5) .. controls (1,.8) and (2,.8) .. (2,1.5) -- (2,1.8);
     \draw[dashed] (2,2) -- (2,3);
     \draw (1.5,4) .. controls (1.5,4.5) .. (2.05,4.8);
     \draw (2,3) .. controls (2,3.5) .. (1.7,3.8);
     \draw (2.25,5.2) -- (2.25,5.6);
     \draw[dashed] (3,.4) .. controls (1.5,.5) .. (1.5,1);
     \node at (2.25,5) [draw,minimum width=20pt,minimum height=10pt,thick, fill=white] {$\mu_W$};
     \node at (1.5,4) [draw,minimum width=20pt,minimum height=10pt,thick, fill=white] {$\mu_V$};
     \node at (2,3) [draw,minimum width=10pt,minimum height=10pt,thick, fill=white] {$\iota_V$};
     \node at (2,2) [draw,minimum width=10pt,minimum height=10pt,thick, fill=white] {$\varepsilon_V$};
     \node at (3,-.25) {$W$};
     \node at (1,.9) {$V$};
     \node at (2,.9) {$V$};
     \node at (2.25,5.85) {$W$};
    \end{tikzpicture}
   \end{matrix},
   \end{align}
 which is the identity on $W$ by Lemma \ref{iota_lemma} and the unit property of $W$. This completes the proof of the theorem.

\begin{rema}
 When $g=1$, $\pi_g$ projects $W$ onto its maximal untwisted submodule. This projection is defined for general rigid commutative algebra objects in braided tensor categories (see \cite[Lemma 4.3]{KO}).
\end{rema}

\section{Twisted modules for vertex operator superalgebras}\label{sec:VOAs}

Here we interpret the categorical results of the preceding sections as theorems for vertex operator (super)algebras.

\subsection{Definitions}\label{subsec:VOAdefs}

There are several slightly variant notions of vertex operator superalgebra (see for example \cite{DL, Xu, Li3, CKL}); we will use the following definition:
\begin{defi}
 A \textit{vertex operator superalgebra} is a $\frac{1}{2}\ZZ$-graded superspace $V=\bigoplus_{n\in\frac{1}{2}\ZZ} V_{(n)}$ equipped with an even \textit{vertex operator} map
 \begin{align*}
  Y: V\otimes V &\rightarrow{V}[[x,x^{-1}]]\nonumber\\
  u\otimes v\hspace{.25em} & \mapsto Y(u,x)v=\sum_{n\in\ZZ} u_n v\,x^{-n-1}
 \end{align*}
and two distinguished vectors $\mathbf{1}\in V_{(0)}\cap V^\even$ called the \textit{vacuum} and $\omega\in V_{(2)}\cap V^\even$ called the \textit{conformal vector}. The data satisfy the following axioms:
\begin{enumerate}

\item \textit{Grading compatibility}: For $i\in\ZZ/2\ZZ$, $V^i=\bigoplus_{n\in\frac{1}{2}\ZZ} V_{(n)}\cap V^i$.

 \item The \textit{grading restriction conditions}: For each $n\in\frac{1}{2}\ZZ$, $V_{(n)}$ is finite dimensional, and for $n\in\frac{1}{2}\ZZ$ sufficiently negative, $V_{(n)}=0$.
 
 \item \textit{Lower truncation}: For any $u, v\in V$, $Y(u,x)v\in V((x))$, that is, $u_n v=0$ for $n$ sufficiently negative.
 
 \item The \textit{vacuum property}: $Y(\mathbf{1}, x)=1_V$.
 
 \item The \textit{creation property}: For any $v\in V$, $Y(v,x)\mathbf{1}\in V[[x]]$ with constant term $v$.
 
 \item The \textit{Jacobi identity}: For any parity-homogeneous $u,v\in V$,
 \begin{align*}
  x_0^{-1}\delta\left(\dfrac{x_1-x_2}{x_0}\right) Y(u,x_1)Y(v,x_2) & -(-1)^{\vert u\vert \vert v\vert}x_0^{-1}\delta\left(\dfrac{-x_2+x_1}{x_0}\right)Y(v,x_2)Y(u,x_1)\nonumber\\
  & =x_2^{-1}\delta\left(\dfrac{x_1-x_0}{x_2}\right) Y(Y(u,x_0)v,x_2).
 \end{align*}

 \item The \textit{Virasoro algebra properties}: If $Y(\omega, x)=\sum_{n\in\ZZ} L(n)\,x^{-n-2}$, then
 \begin{equation*}
  [L(m), L(n)]=(m-n)L(m+n)+\dfrac{m^3-m}{12}\delta_{m+n,0} c 1_V,
 \end{equation*}
where $c\in\CC$ is the \textit{central charge} of $V$. Moreover, for any $n\in\frac{1}{2}\ZZ$, $V_{(n)}$ is the eigenspace for $L(0)$ with eigenvalue $n$; for $v\in V_{(n)}$, we say that $n$ is the \textit{(conformal) weight} of $v$.

\item The \textit{$L(-1)$-derivative property}: For any $v\in V$,
\begin{equation*}
 Y(L(-1)v,x)=\dfrac{d}{dx}Y(v,x).
\end{equation*}
\end{enumerate}
\end{defi}
\begin{rema}
 Some definitions of vertex operator superalgebra require $V^{\bar{i}}=\bigoplus_{n\in\frac{i}{2}+\ZZ} V_{(n)}$ for $i=0,1$, but this is too restrictive. For example, vertex operator superalgebras based on affine Lie superalgebras are $\ZZ$-graded.
\end{rema}

Next we recall the definition of automorphism of a vertex operator superalgebra:
\begin{defi}
 An \textit{automorphism} of a vertex operator superalgebra $(V, Y, \vac,\omega)$ is an even linear automorphism $g$ of $V$ such that $g\cdot\mathbf{1}=\mathbf{1}$, $g\cdot\omega=\omega$, and for any $v\in V$,
 \begin{equation*}
  g\cdot Y(v,x) = Y(g\cdot v, x)g.
 \end{equation*}
\end{defi}
\begin{rema}
 An automorphism $g$ preserves all the vertex operator superalgebra structure of $V$, including both gradings: the $\ZZ/2\ZZ$ grading because $g$ is even and the $\frac{1}{2}\ZZ$-grading because
 \begin{equation*}
  g Y(\omega,x)=Y(g\cdot\omega,x)g=Y(\omega,x)g
 \end{equation*}
implies $g$ commutes with $L(0)$.
\end{rema}
\begin{rema}
 Since $Y$, $\mathbf{1}$, and $\omega$ are all even in a vertex operator superalgebra $V$, the parity automorphism $P_V=1_{V^\even}\oplus(-1_{V^\odd})$ is an automorphism of the vertex operator superalgebra structure.
\end{rema}

Let $g$ be any even grading-preserving linear automorphism of a vertex operator superalgebra $V$. Then since $V_{(n)}$ is finite dimensional for any $n\in\frac{1}{2}\ZZ$, $g\vert_{V_{(n)}}$ is the exponential of some even linear endomorphism of $V_{(n)}$. Thus $g=e^{2\pi i\gamma}$ where $\gamma$ is a (non-unique) even grading-preserving linear endomorphism of $V$. For concreteness, we choose a specific $\gamma$, following \cite{Ba, HY}: On each $V_{(n)}$, $g$ decomposes uniquely as the product of commuting semisimple and unipotent parts, and the unipotent part is the exponential of a nilpotent endomorphism. Putting these parts together, $g=\sigma e^{2\pi i\mathcal{N}}$ where $\sigma$ is semisimple and $\mathcal{N}$ is locally nilpotent. On any generalized $g$-eigenspace of $V$, $\sigma$ equals a constant $e^{2\pi i\alpha}$ for a unique $\alpha\in\CC$ such that $0\leq\mathrm{Re}\,\alpha<1$. We define $\gamma=\alpha+\mathcal{N}$ on such a generalized eigenspace. As $V$ is the direct sum of its generalized $g$-eigenspaces, this completely specifies $\gamma$.
%

For any grading-preserving linear endomorphism $\gamma$ of $V$, necessarily decomposable as a sum of commuting grading-preserving semisimple and locally nilpotent parts $\gamma_S$ and $\gamma_N$, we define the operator
\begin{equation*}
 x^\gamma: V\rightarrow V[\log x]\{ x\}
\end{equation*}
such that if $v\in V$ is a generalized eigenvector for $\gamma$ with generalized eigenvalue $\alpha$,
\begin{equation*}
 x^\gamma\cdot v= x^\alpha e^{(\log x)\gamma_N}\cdot v,
\end{equation*}
where the exponential sum truncates because $\gamma_N$ is locally nilpotent. Now we can define twisted modules associated to an automorphism of a vertex operator superalgebra:

\begin{defi}\label{def:VoaTwistMods}
 Let $g=e^{2\pi i\gamma}$ be an automorphism of a vertex operator superalgebra $V$ with $\gamma$ chosen as above. A \textit{strongly-graded generalized $g$-twisted $V$-module} is a $\CC$-graded superspace $W=\bigoplus_{h\in\CC} W_{[h]}$ equipped with an even \textit{vertex operator} map
 \begin{align*}
  Y_W: V\otimes W & \rightarrow W[\log x]\{ x\}\nonumber\\
   v\otimes w & \mapsto Y_W(v,x)w=\sum_{h\in\CC} \sum_{k\in\NN} v_{h; k} w\,x^{-h-1} (\log x)^k
 \end{align*}
satisfying the following properties:
\begin{enumerate}
 \item \textit{Grading compatibility}: For $i\in\ZZ/2\ZZ$, $W^i=\bigoplus_{h\in\CC} W_{[h]}\cap W^i$,

 \item The \textit{grading restriction conditions}: For any $h\in\CC$, $W_{[h]}$ is finite dimensional and $W_{[h+r]}=0$ for $r\in\RR$ sufficiently negative.

 \item \textit{Lower truncation}: For any $v\in V$, $w\in W$, and $h\in\CC$, $v_{h+n; k} w = 0$ for $n\in\NN$ sufficiently large, independently of $k\in\NN$.
 
 \item The \textit{$g$-equivariance property}: For any $v\in V$, $Y_W(g\cdot v, e^{2\pi i} x) = Y_W(v,x)$. 
 
 \item The \textit{vacuum property}: $Y_W(\mathbf{1}, x)=1_W$.
 
 \item The \textit{Jacobi identity}: For any parity-homogeneous $u, v\in V$,
 \begin{align*}
  x_0^{-1}\delta\left(\dfrac{x_1-x_2}{x_0}\right) Y_W(u,x_1) & Y_W(v,x_2)  -(-1)^{\vert u\vert \vert v\vert}x_0^{-1}\delta\left(\dfrac{-x_2+x_1}{x_0}\right)Y_W(v,x_2)Y_W(u,x_1)\nonumber\\
  & =x_1^{-1}\delta\left(\dfrac{x_2+x_0}{x_1}\right) Y_W\left(Y\left(\left(\dfrac{x_2+x_0}{x_1}\right)^{\gamma}\cdot u,x_0\right)v,x_2\right).
 \end{align*}

\item If $Y_W(\omega, x)=\sum_{n\in\ZZ} L_W(n) x^{-n-2}$, then for any $h\in\CC$, $W_{[h]}$ is the generalized eigenspace of $L_W(0)$ with generalized eigenvalue $h$.

\item The \textit{$L(-1)$-derivative property}: For any $v\in V$,
\begin{equation*}
 Y_W(L(-1)v,x)=\dfrac{d}{dx}Y_W(v,x).
\end{equation*}
\end{enumerate}
\end{defi}

\begin{rema}
We will typically refer to a strongly-graded generalized $g$-twisted $V$-module simply as a $g$-twisted $V$-module. Note that this term sometimes refers on which $L_W(0)$ acts semisimply. The $g=1_V$ case of Definition \ref{def:VoaTwistMods} is the definition of \textit{(strongly-graded generalized) $V$-module}.
\end{rema}

\begin{rema}
 Although the logarithm of $g$ is not unique, Definition \ref{def:VoaTwistMods} at least does not depend on the choice of semisimple part of $\gamma$. If $\gamma'$ is another choice of logarithm with locally nilpotent part $\mathcal{N}$ (that is, we lift the restriction on the real part of the eigenvalues of $\gamma'$), then for any $v\in V$,
 \begin{equation*}
  x^{\gamma'}\cdot v=\sum x^{\gamma+n_i}\cdot v_i
 \end{equation*}
for some integers $n_i$ and vectors $v_i$ such that $v=\sum v_i$. Then in the Jacobi identity, the extra factors of $\left(\frac{x_2+x_0}{x_1}\right)^{n_i}$ absorb into the delta function. Our specific choice of the semisimple part of $\gamma$ was chosen for simplicity and for consistency with \cite{H,Ba,HY}.
\end{rema}


\begin{rema}
When $g$ has infinite order, it follows from \cite[Theorem 5.2]{Ba} and \cite[Theorem 2.7]{HY} that the Jacobi identity in Definition \ref{def:VoaTwistMods} is equivalent to the duality property in Huang's definition of twisted module \cite{H} (see also \cite[Theorem 3.8]{Huang-TwistOps}). In fact, the only difference between Definition \ref{def:VoaTwistMods} and the definition of twisted module given in \cite{HY} is that here we do not assume a linear automorphism $g_W$ of a $g$-twisted $V$-module $W$ such that
\begin{equation}\label{g_W_condition}
 g_W\cdot Y_W(v, x) w = Y_W(g\cdot v, x)g_W\cdot w.
\end{equation}
Actually, we automatically have such a $g_W$ in some cases. When $V$ is $\ZZ$-graded, we can take $g_W=e^{-2\pi i L_W(0)}$ and when $V^{\bar{i}}=\bigoplus_{n\in\frac{i}{2}+\ZZ}$ for $i=0,1$, we can take $g_W=P_W e^{-2\pi i L_W(0)}$. To show that such $g_W$ satisfy \eqref{g_W_condition}, one uses the evenness of $Y_W$ and the $L_W(0)$-conjugation formula
\begin{equation}\label{L0-conj}
 e^{h L_W(0)} Y_W(v,x)e^{-h L_W(0)} =Y_W(e^{h L(0)}\cdot v, e^h x)
\end{equation}
for $h\in\CC$, $v\in V$ (see for instance \cite[Proposition 3.36(b)]{HLZ2}, which applies because $Y_W$ is an intertwining operator among modules for the vertex operator subalgebra of even $g$-fixed points in $V$).
\end{rema}

%
%
%

The following equivalent form of the $g$-equivariance property of a $g$-twisted $V$-module will be useful:
\begin{lemma}\label{lem:g-equiv}
 The lower truncation and $g$-equivariance properties of Definition \ref{def:VoaTwistMods} are equivalent to the condition that for any $v\in V$ and $w\in W$,
 \begin{equation*}
  Y_W(x^\gamma\cdot v, x)w\in W((x)).
 \end{equation*}
\end{lemma}
\begin{proof}
If $Y_W$ satisfies lower truncation and $g$-equivariance, then for $v\in V$, the $g$-equivariance property implies
\begin{equation*}
 Y_W((e^{2\pi i} x)^\gamma\cdot v, e^{2\pi i} x) = Y(g x^{\gamma}\cdot v, e^{2\pi i} x)= Y(x^\gamma\cdot v,x). 
\end{equation*}
But any $f(x)\in(\mathrm{End}\,W)[\log x]\lbrace x\rbrace$ that satisfies $f(e^{2\pi i} x)=f(x)$ must be a Laurent series. To show this, suppose  $f(x)=\sum_{h\in\CC, k\in\NN} f_{h,k}\,x^h (\mathrm{log}\,x)^k$, so that 
 \begin{equation}\label{fmono}
  f(e^{2\pi i} x)=\sum_{h\in\CC, k\in\NN} e^{2\pi i h}f_{h,k}\, x^h(\mathrm{log}\,x+2\pi i)^k.
 \end{equation}
If $h\in\CC$ satisfies $f_{h,k}\neq 0$ for some $k\in\NN$, let $K$ be maximal so that $f_{h,K}\neq 0$. We must show that $h\in\ZZ$ and $K=0$. Since $f(e^{2\pi i } x)=f(x)$, the coefficient of $x^h (\mathrm{log}\,x)^K$ in $f(e^{2\pi i} x)$ equals the coefficient of $x^h (\mathrm{log}\,x)^K$ in $f(x)$. So \eqref{fmono} implies $e^{2\pi i h} f_{h,K}=f_{h,K}$. Since $f_{h,K}\neq 0$, this meanss $e^{2\pi i h}=1$, or $h\in\ZZ$.

Next, if $K>0$, we compare coefficients of $x^h (\mathrm{log}\,x)^{K-1}$ in $f(e^{2\pi i} x)$ and $f(x)$ and find
\begin{equation*}
 e^{2\pi i h}(f_{h,K-1}+2\pi i K f_{h,K})=f_{h,K-1}.
\end{equation*}
We already know $e^{2\pi i h}=1$, so $2\pi i K f_{h,K}=0$. This is a contradiction since $f_{h,K}\neq 0$, so $K=0$.

Now for $w\in W$, lower truncation implies $Y_W(x^\gamma\cdot v,x)w$ is also lower-truncated, that is, $Y_W(x^\gamma\cdot v,x)w\in W((x))$.

Conversely, assume $Y_W$ satisfies $Y_W(x^\gamma\cdot v,x)w\in W((x))$ for $v\in V$, $w\in W$. To show lower truncation, assume without loss of generality that $v$ is a generalized eigenvector for $\gamma$ with generalized eigenvalue $\alpha$. Then
\begin{equation*}
 Y_W(v,x)w =Y_W(x^\gamma x^{-\gamma}\cdot v, x)w =\sum_{i\geq 0} \frac{(-1)^i}{i!} x^{-\alpha} (\log x)^i Y_W(x^\gamma\cdot\mathcal{N}^i v, x)w.
\end{equation*}
As the sum over $i$ is finite, lower truncation follows because each $Y_W(x^\gamma\cdot\mathcal{N}^i v, x)w\in W((x))$. Moreover, since
$$Y_W((e^{2\pi i}x)^\gamma\cdot\mathcal{N}^i v, e^{2\pi i}x)= Y_W(x^\gamma\cdot\mathcal{N}^i v,x)$$
for each $i$, we get
\begin{align*}
 Y_W(v, x) = Y_W(x^\gamma x^{-\gamma}\cdot v,x) = Y_W((e^{2\pi i}x)^\gamma x^{-\gamma}\cdot v,e^{2\pi i} x)= Y_W(e^{2\pi i\gamma} x^\gamma x^{-\gamma}\cdot v, e^{2\pi i}x)=Y_W(g\cdot v,e^{2\pi i}x),
\end{align*}
which is the $g$-equivariance property.
\end{proof}

\subsection{General theorems}\label{subsec:VOAthms}

Let $G$ be an automorphism group of a vertex operator superalgebra $V$ that includes $P_V$. Then the $G$-fixed points
\begin{equation*}
 V^G=\lbrace v\in V\,\vert\,g\cdot v=v\,\,\mathrm{for\,\,all}\,\,g\in G\rbrace
\end{equation*}
form a vertex operator subalgebra of $V^\even$. If $\mathcal{C}$ is a category of strongly-graded generalized $V^G$-modules that includes $V$ and admits vertex tensor category structure as constructed in \cite{HLZ1}-\cite{HLZ8}, then by \cite[Theorem 3.2]{HKL}, \cite[Theorem 3.13]{CKL}, $V$ is a superalgebra in the braided tensor category $\mathcal{C}$ and we have the monoidal supercategory $\rep\,V$ as in Section \ref{sec:BGXSC}. In this setting, we need to verify that the definition of $g$-twisted $V$-module for $g\in G$ from the previous subsection agrees with the categorical definition of Section \ref{sec:TwistMods}. To accomplish this, we first recall from \cite{CKM} how to characterize modules in $\rep\,V$ in terms of intertwining operators:
\begin{propo}\label{propo:RepV_using_intwops}\cite[Proposition 3.46]{CKM}
 Let $G$ be an automorphism group of a vertex operator superalgebra $V$ that includes $P_V$ and let $\mathcal{C}$ be a category of strongly-graded generalized $V^G$-modules that includes $V$ and admits vertex tensor category structure. Then an object of $\rep\,V$ is precisely a $V^G$-module $W$ in $\cC$ equipped with a $V^G$-module intertwining operator $Y_W$ of type $\binom{W}{V\,W}$ satisfying the following two properties:
 \begin{enumerate}
  \item Unit: $Y_W(\vac, x) = 1_W$.
  \item Associativity: For $v_1, v_2\in V$, $w\in W$, and $w'\in W'=\bigoplus_{h\in\CC} W_{[h]}^*$, the multivalued analytic functions
  \begin{equation*}
  P(z_1,z_2) = \langle w', Y_W(v_1,z_1)Y_W(v_2,z_2)w\rangle
  \end{equation*}
on the region $\vert z_1\vert>\vert z_2\vert>0$ and
\begin{equation*}
I(z_1,z_2) = \langle w', Y_W(Y(v_1,z_1-z_2)v_2,z_2)w\rangle
\end{equation*}
on the region $\vert z_2\vert>\vert z_1-z_2\vert>0$ have equal restrictions to their common domain. Specifically, the equality
\begin{equation*}
 \langle w', Y_W(v_1, e^{\ln r_1})Y_W(v_2,e^{\ln r_2})w\rangle =\langle w', Y_W(Y(v_1, r_1-r_2)v_2, e^{\ln r_2})w\rangle
\end{equation*}
of single-valued branches holds on the simply-connected region $r_1>r_2>r_1-r_2>0$ of $(\RR_+)^2$, where the notation means the real-valued branch $\ln$ of logarithm on $\RR_+$ is used to evaluate powers and logarithms.
 \end{enumerate}
\end{propo}
\begin{rema}
 The associativity property of a module in $\rep\,V$ is stated somewhat differently in \cite[Proposition 3.46]{CKM}, using a simply-connected open region of $(\CC^\times)^2$ containing the region $r_1>r_2>r_1-r_2>0$ in its boundary. However, this difference is irrelevant in light of Proposition 3.18 and Remark 3.19 of \cite{CKM}.
\end{rema}

Next, the relationship between the intertwining operator $Y_W$ and the morphism $\mu_W: V\tens W\rightarrow W$ for a module $W$ in $\rep\,V$, given in the proof of \cite[Proposition 3.46]{CKM}, together with \cite[Equation 3.15]{CKM} for the monodromy isomorphism in $\cC$, imply:
\begin{propo}\label{cor:g-equiv_and_mono}
 In the setting of Proposition \ref{propo:RepV_using_intwops}, a module $W$ in $\rep\,V$ is a $g$-twisted $V$-module for $g\in G$ in the sense of Definition \ref{def:CatTwistMods}, that is,
 \begin{equation*}
  \mu_W(g\tens 1_W)\cM_{V,W}=\mu_W,
 \end{equation*}
if and only if $Y_W$ satisfies the $g$-equivariance property of Definition \ref{def:VoaTwistMods}.
\end{propo}

Now we establish the equivalence of Definitions \ref{def:CatTwistMods} and \ref{def:VoaTwistMods}. The proof is technical but uses standard vertex algebraic techniques and has similarities to the proofs of \cite[Theorems 3.6.3 and 4.4.5]{LL}, \cite[Theorem 2.10]{HY}, \cite[Theorem 3.53]{CKM}, and \cite[Lemma 3.2]{DLXY}.
\begin{theo}\label{thm:tw_mod_defs}
 In the setting of Proposition \ref{propo:RepV_using_intwops}, a $V^G$-module in $\cC$ is a $g$-twisted $V$-module for some $g\in G$ in the sense of Definition \ref{def:CatTwistMods} if and only if it is $g$-twisted in the sense of Definition \ref{def:VoaTwistMods}.
\end{theo}
\begin{proof}
If $W$ is a $g$-twisted $V$-module in the sense of Definition \ref{def:CatTwistMods}, then by Propositions \ref{propo:RepV_using_intwops} and \ref{cor:g-equiv_and_mono}, $W$ satisfies the $g$-equivariance and vacuum properties of Definition \ref{def:VoaTwistMods}. Also, $W$ satisfies all grading conditions in Definition \ref{def:VoaTwistMods} because it is a strongly-graded generalized $V^G$-module, and $Y_W$ satisfies lower truncation and the $L(-1)$-derivative property because it is an intertwining operator among $V^G$-modules. It remains to derive the Jacobi identity from the associativity of $Y_W$.

By \cite[Remark 3.47]{CKM}, $Y_W$ satisfies the following skew-associativity property in addition to associativity: for $w\in W$, $w'\in W'$, and parity-homogeneous $v_1,v_2\in V$, the multivalued analytic functions $I(z_1,z_2)$ on the region $\vert z_2\vert>\vert z_1-z_2\vert>0$ and
\begin{equation*}
 Q(z_1,z_2) = (-1)^{\vert v_1\vert\vert v_2\vert}\langle w', Y_W(v_2,z_2)Y_W(v_1,z_1)w\rangle
\end{equation*}
on the region $\vert z_2\vert>\vert z_1\vert>0$ have equal restrictions to their common domain. Specifically, the equality
\begin{equation*}
 (-1)^{\vert v_1\vert\vert v_2\vert}\langle w', Y_W(v_2,e^{\ln r_2})Y_W(v_1,e^{\ln r_1})w\rangle = \langle w', Y_W(Y(v_1,r_1-r_2)v_2,e^{\ln r_2})w\rangle
\end{equation*}
of single-valued branches holds on the simply-connected region $r_2>r_1>r_2-r_1>0$ of $(\RR_+)^2$. Then we use \cite[Lemma 4.1]{H-genlratl} to extend the multivalued analytic functions $P(z_1,z_2)$, $Q(z_1,z_2)$, and $I(z_1,z_2)$, which agree on their common domains, to a multivalued analytic function $F(v_1; z_1,z_2)$ defined on $(\CC^\times)^2\setminus\lbrace(z,z)\,\vert\,z\in\CC^\times\rbrace$. (The convergence, associativity, and commutativity properties for intertwining operators among $V^G$-modules required in the proof of this lemma from \cite{H-genlratl} are subsumed under the assumption that these intertwining operators satisfy the sufficient conditions of \cite{HLZ1}-\cite{HLZ8} for vertex tensor category structure on $\cC$.) For $\gamma$ a grading-preserving linear endomorphism of $V$ such that $e^{2\pi i\gamma} =g$, we define a new multivalued analytic function $f(z_1,z_2)=F(z_1^\gamma v_1; z_1, z_2)$.

We now define an (\textit{a priori} multivalued) function of the single variable $z_1$. Fix $r_2\in\RR_+$ and choose $r_1\in\RR_+$ such that $r_2>r_1>r_2-r_1>0$. Then for $z_1\in\CC\setminus\lbrace 0,r_2\rbrace$, define $f_{r_2}(z_1)$ to take all values of $f(z_1,r_2)$ that can be obtained by analytic continuation along continuous paths from $r_1$ to $z_1$ in $\CC\setminus\lbrace 0, r_2\rbrace$, starting from the value 
\begin{align*}
(-1)^{\vert v_1\vert\vert v_2\vert} \langle w',Y_W(v_2, e^{\ln r_2}) Y_W(e^{(\ln r_1)\gamma} v_1, & e^{\ln r_1})w\rangle  = \langle w', Y_W(Y(e^{(\ln r_1)\gamma} v_1, r_1-r_2)v_2, e^{\ln r_2})w\rangle\nonumber\\
& = \left\langle w', Y_W\left(Y\left(e^{(\ln r_2)\gamma}\left(1+\frac{r_1-r_2}{r_2}\right)^\gamma v_1, r_1-r_2\right)v_2, e^{\ln r_2}\right)w\right\rangle
\end{align*}
of $f(r_1,r_2)$. We claim that $f_{r_2}(z_1)$ is actually single-valued, that is, the value of $f_{r_2}(z_1)$ obtained by analytic continuation from $r_1$ to $z_1$ is independent of the path. Equivalently, analytic continuation along any continuous path from $r_1$ to $r_1$ in $\CC\setminus\lbrace 0,r_2\rbrace$ does not change the starting value of $f(r_1,r_2)$. To prove this, note that any continuous path from $r_1$ to itself in $\CC\setminus\lbrace 0, r_2\rbrace$ is homotopic to a sequence of loops based at $r_1$ with each loop encircling either $r_2$ or $0$ and remaining within the region $r_2>\vert z_1-r_2\vert>0$ or $r_2>\vert z_1\vert>0$, respectively. But the value of $f(r_1,r_2)$ does not change going around $r_2$ because the series
\begin{equation*}
 Y\left(e^{(\ln r_2)\gamma}\left(1+\frac{x_0}{r_2}\right)^\gamma v_1, x_0\right)v_2
\end{equation*}
has no monodromy in $x_0$, and the value of $f(r_1,r_2)$ does not change going around $0$ because $Y_W( x_1^\gamma v_1, x_1)w$ has no monodromy in $x_1$ by the $g$-equivariance property (recall Lemma \ref{lem:g-equiv}).

The analytic function $f_{r_2}(z_1)$ has singularities at $0$, $r_2$, and $\infty$. Its Laurent series expansion around $\infty$ is
\begin{equation*}
 P_{r_2}(x_1)\vert_{x_1=z_1}=\langle w', Y_W(x_1^\gamma v_1, x_1)Y_W(v_2,e^{\ln r_2})w\rangle\vert_{x_1=z_1},
\end{equation*}
its Laurent series expansion around $0$ is
\begin{equation*}
 Q_{r_2}(x_1)\vert_{x_1=z_1} =(-1)^{\vert v_1\vert\vert v_2\vert}\langle w', Y_W(v_2, e^{\ln r_2})Y_W( x_1^\gamma v_1,x_1)w\rangle\vert_{x_1=z_1},
\end{equation*}
and its is Laurent series expansion around $r_2$ is
\begin{equation*}
 I_{r_2}(x_0)\vert_{x_0=z_1-r_2} =\left\langle w', Y_W\left(Y\left( e^{(\ln r_2)\gamma}\left(1+\frac{x_0}{r_2}\right)^\gamma v_1, x_0\right)v_2, e^{\ln r_2}\right)w\right\rangle.
\end{equation*}
All singularities are poles because $W$ is strongly graded and $Y_W$ is lower truncated, so $f_{r_2}$ is a rational function:
\begin{equation*}
 f_{r_2}(z_1)=\frac{p_{r_2}(z_1)}{z_1^M (z_1-r_2)^N}
\end{equation*}
where $p_{r_2}(z_1)$ is a polynomial and $M,N\in\NN$. Then
\begin{equation*}
 P_{r_2}(x_1)=\frac{p_{r_2}(x_1)}{x_1^M (x_1-r_2)^N},\qquad Q_{r_2}(x_1)=\frac{p_{r_2}(x_1)}{x_1^M (-r_2+x_1)^N},\qquad I_{r_2}(x_0)=\frac{p_{r_2}(r_2+x_0)}{(r_2+x_0)^M x_0^N},
\end{equation*}
where the binomial terms are expanded in non-negative powers of the second variable.

To get a Jacobi identity involving $r_2$, we multiply both sides of the three-term delta-function identity
\begin{equation*}
 x_0^{-1}\delta\left(\frac{x_1-r_2}{x_0}\right)-x_0^{-1}\delta\left(\frac{-r_2+x_1}{x_0}\right) = r_2^{-1}\delta\left(\frac{x_1-x_0}{r_2}\right)
\end{equation*}
by $p_{r_2}(x_1)/x_1^M x_0^N$. Delta-function substitution properties from \cite[Remark 2.3.25]{LL} then yield
\begin{equation*}
 x_0^{-1}\delta\left(\frac{x_1-r_2}{x_0}\right)P_{r_2}(x_1)-x_0^{-1}\delta\left(\frac{-r_2+x_1}{x_0}\right)Q_{r_2}(x_1) = r_2^{-1}\delta\left(\frac{x_1-x_0}{r_2}\right)I_{r_2}(x_0).
\end{equation*}
Since the $w$ and $w'$ in the definition of $P$, $Q$, and $I$ were arbitrary, we get the Jacobi identity
\begin{align}\label{r_2-Jacobi}
 x_0^{-1}\delta\left(\frac{x_1-r_2}{x_0}\right) & Y_W(x_1^\gamma v_1, x_1)Y_W(v_2, e^{\ln r_2})- (-1)^{\vert v_1\vert\vert v_2\vert}x_0^{-1}\delta\left(\frac{-r_2+x_1}{x_0}\right)Y_W(v_2, e^{\ln r_2})Y_W(x_1^\gamma v_1, x_1)\nonumber\\
 & = r_2^{-1}\delta\left(\frac{x_1-x_0}{r_2}\right)Y_W\left(Y\left(e^{(\ln r_2)\gamma}\left(1+\frac{x_0}{r_2}\right)^\gamma v_1, x_0\right)v_2, e^{\ln r_2}\right).
\end{align}
To replace the real number $r_2$ in the Jacobi identity by the formal variable $x_2$, we use the method of \cite[Proposition 4.8]{HLZ3}. First replace $v_1$ in \eqref{r_2-Jacobi} with $x_1^{-\gamma} (r_2 x_2^{-1})^{L(0)}$, and then make the substitutions $x_0\mapsto x_0 r_2 x_2^{-1}$ and $x_1\mapsto x_1 e^{\ln r_2} x_2^{-1}$. This yields the identity
\begin{align*}
r_2^{-1} x_2 x_0^{-1} & \delta\left(\frac{x_1-x_2}{x_0}\right)  Y_W((r_2 x_2^{-1})^{L(0)} v_1, x_1 e^{\ln r_2} x_2^{-1})Y_W(v_2, e^{\ln r_2})\nonumber\\
& \hspace{2em}-(-1)^{\vert v_1\vert\vert v_2\vert}r_2^{-1} x_2 x_0^{-1}\delta\left(\frac{-x_2+x_1}{x_0}\right)Y_W(v_2, e^{\ln r_2})Y_W((r_2 x_2^{-1})^{L(0)} v_1, x_1 e^{\ln r_2} x_2^{-1})\nonumber\\
 & = r_2^{-1}\delta\left(\frac{x_1-x_0}{x_2}\right)Y_W\left(Y\left(\left(\frac{x_2+x_0}{x_1}\right)^\gamma (r_2 x_2^{-1})^{L(0)} v_1, x_0 r_2 x_2^{-1}\right)v_2, e^{\ln r_2}\right).
\end{align*}
By the $L(0)$-conjugation property of $V^G$-module intertwining operators to $Y_W$ and $Y$, this is equivalent to
\begin{align*}
r_2^{-1} & x_2 x_0^{-1}\delta\left(\frac{x_1-x_2}{x_0}\right)  e^{(\ln r_2) L(0)} x_2^{-L(0)} Y_W(v_1, x_1)Y_W((r_2^{-1} x_2)^{L(0)}v_2, x_2)e^{-(\ln r_2)L(0)} x_2^{L(0)}\nonumber\\
& \hspace{2em}-(-1)^{\vert v_1\vert\vert v_2\vert}r_2^{-1} x_2 x_0^{-1}\delta\left(\frac{-x_2+x_1}{x_0}\right)e^{(\ln r_2) L(0)} x_2^{-L(0)}Y_W((r_2^{-1} x_2)^{L(0)}v_2, x_2)Y_W(v_1, x_1)e^{-(\ln r_2)L(0)} x_2^{L(0)}\nonumber\\
 & = r_2^{-1}\delta\left(\frac{x_1-x_0}{x_2}\right)e^{(\ln r_2)L(0)} x_2^{-L(0)}Y_W\left(Y\left(\left(\frac{x_2+x_0}{x_1}\right)^\gamma v_1, x_0\right)(r_2^{-1} x_2)^{L(0)}v_2, x_2\right)e^{-(\ln r_2)L(0)} x_2^{L(0)}.
\end{align*}
To get the Jacobi identity of Definition \ref{def:VoaTwistMods} from this identity, just multiply both sides by  $r_2 x_2^{-1} e^{-(\ln r_2)L(0)} x_2^{L(0)}$ on the left and $e^{(\ln r_2)L(0)} x_2^{-L(0)}$ on the right, then replace $v_2$ with $(r_2 x_2^{-1})^{L(0)} v_2$. This completes the proof that a $g$-twisted $V$-module in the sense of Definition \ref{def:CatTwistMods} is a $g$-twisted $V$-module in the sense of Definition \ref{def:VoaTwistMods}.

Conversely, suppose a strongly-graded generalized $V^G$-module $(W,Y_W)$ in $\cC$ is a $g$-twisted $V$-module in the sense of the Definition \ref{def:VoaTwistMods}. We just need to show that $Y_W$ satisfies the associativity property of \cite[Proposition 3.46]{CKM}, as the unit property $Y_W(\vac,x)=1_W$ of Proposition \ref{propo:RepV_using_intwops}  is already part of Definition \ref{def:VoaTwistMods} and
\begin{equation*}
 \mu_W(g\tens 1_W)\cM_{V,W}=\mu_W
\end{equation*} 
follows from the $g$-equivariance property by Proposition \ref{cor:g-equiv_and_mono}.

We start by noting the following weak associativity for $u,v\in V$. Replacing $u$ in the Jacobi identity with $x_1^\gamma\cdot u$ and extracting a sufficiently  negative (integer) power of $x_1$, we get
\begin{equation*}
 Y_W((x_0+x_2)^{\gamma+M}\cdot u, x_0+x_2)Y_W(v,x_2)w = Y_W(Y((x_2+x_0)^{\gamma+M}\cdot u, x_0)v,x_2)w
\end{equation*}
as series in $x_0$ and $x_2$ when $M\in\NN$ is sufficiently large (depending on $u$ and $w$). If we further replace $v$ by $x_2^\gamma\cdot v$ and pair with $w'\in W'$, the grading-restriction conditions, lower truncation, and $g$-equivariance show that the series
\begin{equation*}
 \langle w', Y_W((x_0+x_2)^{\gamma+M}\cdot u, x_0+x_2)Y_W(x_2^\gamma\cdot v,x_2)w\rangle\quad\mathrm{and}\quad\langle w', Y_W(Y((x_2+x_0)^{\gamma+M}\cdot u,x_0)x_2^{\gamma}\cdot v,x_2)w\rangle
\end{equation*}
equal a common Laurent polynomial in $x_0$ and $x_2$.

Now take $v_1,v_2\in V$, $w\in W$, and $w'\in W'$, assuming without loss of generality that $v_1$ and $v_2$ are generalized eigenvectors for $\gamma$ with generalized eigenvalues $\alpha_1$ and $\alpha_2$, respectively. Then
\begin{align*}
\langle w',  Y_W(Y(v_1,x_0)v_2,x_2)w\rangle & = \langle w', Y_W(Y((x_2+x_0)^{\gamma+M}(x_2+x_0)^{-\gamma-M}\cdot v_1, x_0)x_2^{\gamma}x_2^{-\gamma}\cdot v_2,x_2)w\rangle\nonumber\\
 & =\sum_{i=0}^I\sum_{j=0}^J \frac{(-1)^{i+j}}{i!j!} (x_2+x_0)^{-\alpha_1-M} x_2^{-\alpha_2}(\log(x_2+x_0))^i(\log x_2)^j\cdot\nonumber\\
 &\hspace{7em}\cdot\langle w', Y_W(Y((x_2+x_0)^{\gamma+M}\cdot\mathcal{N}^i v_1,x_0)x_2^{\gamma}\cdot\mathcal{N}^j v_2,x_2)w\rangle
\end{align*}
for any $M\in\NN$. Since $I$ and $J$ are finite, we use weak associativity to choose $M$ sufficiently large so that each
\begin{equation*}
 \langle w', Y_W(Y((x_2+x_0)^{\gamma+M}\cdot\mathcal{N}^i v_1,x_0)x_2^{\gamma}\cdot\mathcal{N}^j v_2,x_2)w\rangle
\end{equation*}
is a Laurent polynomial $p_{i,j}(x_0,x_2)$. The same argument applied to $\langle w', Y_W(v_1,x_0+x_2)Y_W(v_2,x_2)w_2\rangle$ together with weak associativity shows that
\begin{equation*}
 \langle w', Y_W(v_1,x_0+x_2)Y_W(v_2,x_2)w_2\rangle =\sum_{i=0}^I\sum_{j=0}^J \frac{(-1)^{i+j}}{i!j!} (\log(x_0+x_2))^i(\log x_2)^j\cdot\frac{p_{i,j}(x_0,x_2)}{(x_0+x_2)^{\alpha_1+M} x_2^{\alpha_2}},
\end{equation*}
and thus
\begin{equation*}
 \langle Y_W(v_1,x_1)Y_W(v_2,x_2)w\rangle =\sum_{i=0}^I\sum_{j=0}^J \frac{(-1)^{i+j}}{i! j!}(\log x_1)^i(\log x_2)^j\cdot \frac{p_{i,j}(x_1-x_2,x_2)}{x_1^{\alpha_1+M} x_2^{\alpha_2}}.
\end{equation*}

Now for any $r_1,r_2\in\RR$ such that $r_1>r_2>r_1-r_2>0$, make the substitutions $x_1\mapsto e^{\ln r_1}$, $x_2\mapsto e^{\ln r_2}$ and $x_0\mapsto r_1-r_2$. Using $\log(1+x)$ for a real number $x$ to denote the power series expansion of $\ln(1+x)$ when $\vert x\vert<1$, we get
\begin{align*}
 \langle w', Y_W(Y(v_1, & r_1-r_2)v_2, e^{\ln r_2})w\rangle\nonumber\\
 & = \sum_{i=0}^I\sum_{j=0}^J \frac{(-1)^{i+j}}{i! j!}\left(\ln r_2+\log\left(1+\frac{r_1-r_2}{r_2}\right)\right)^i(\ln r_2)^j\cdot\frac{p_{i,j}(r_1-r_2,r_2)}{\left(1+\frac{r_1-r_2}{r_2}\right)^{\alpha_1+M} e^{(\alpha_1+\alpha_2+M)\ln r_2}}\nonumber\\
 & = \sum_{i=0}^I\sum_{j=0}^J \frac{(-1)^{i+j}}{i! j!} (\ln r_1)^i(\ln r_2)^j\cdot\frac{p_{i,j}(r_1-r_2,r_2)}{e^{(\alpha_1+M)\ln r_1} e^{\alpha_2 \ln r_2}}\nonumber\\
 & =\langle w', Y_W(v_1,e^{\ln r_1})Y_W(v_2, e^{\ln r_2})w\rangle.
\end{align*}
Thus the multivalued functions $\langle w', Y_W(v_1,z_1)Y_W(v_2,z_2)w\rangle $ and $\langle w', Y_W(Y(v_1,z_1-z_2)v_2,z_2)w\rangle$ have equal restrictions to their common domain, with equality of single-valued branches on a simply-connected domain as specified in the associativity property of \cite[Proposition 3.46]{CKM}. This proves that $W$ is a $g$-twisted $V$-module in the sense of Definition \ref{def:CatTwistMods}.
\end{proof}

Now that we have unified the categorical and vertex algebraic definitions of twisted module, we apply the categorical theorems of the previous sections to vertex operator superalgebras. In the next theorem, we verify the conditions of Assumption \ref{mainassum} using results from \cite{DLM} and \cite{McR}:
\begin{theo}\label{thm:mainVOAthm1}
 Let $V$ be a simple vertex operator superalgebra, $G$ a finite automorphism group of $V$ that includes $P_V$, and $\cC$ an abelian category of strongly-graded generalized $V^G$-modules that includes $V$ and admits vertex tensor category structure as in \cite{HLZ1}-\cite{HLZ8}. Then:
 \begin{enumerate}
  \item Every indecomposable object of the monoidal supercategory $\rep\,V$ is a $g$-twisted $V$-module for some $g\in G$.
  
  \item The monoidal supercategory $\rep\,V$ admits the structure of a braided $G$-crossed supercategory.
 \end{enumerate}
\end{theo}
\begin{proof}
 Because of the dictionary between twisted modules for vertex operator superalgebras and twisted modules for superalgebra objects in braided tensor categories provided by \cite[Theorem 3.2]{HKL}, \cite[Theorem 3.13]{CKL}, and Theorem \ref{thm:tw_mod_defs}, the conclusions follow from Theorem \ref{Gcrossedfromtwist} (or Corollary \ref{cor:RepV_G-crossed}) and Theorem \ref{thm:repV=repGV}   once we verify the necessary conditions. The assumption in Theorem \ref{Gcrossedfromtwist} that tensoring functors in $\cC$ are right exact, needed for the construction of the monoidal supercategory structure on $\rep\,V$, follows from \cite[Proposition 4.26]{HLZ3}. It remains to verify the conditions of Assumption \ref{mainassum}.
 
 The first two conditions on $G$ in Assumption \ref{mainassum} hold by assumption and because $\mathbb{F}$ here is $\CC$. For the remaining conditions, we use \cite[Theorem 2.4]{DLM} (see also \cite[Theorem 3.2]{McR} which covers the superalgebra generality) which states that $V$ is a semisimple $G\times V^G$-module:
 \begin{equation}\label{SW-duality}
  V=\bigoplus_{\chi\in\widehat{G}} M_\chi\otimes V_\chi
 \end{equation}
where the $M_\chi$ are irreducible $G$-modules with character $\chi$ and the $V_\chi$ are non-zero, simple, and distinct $V^G$-modules. Since $V^G$ is paired with the one-dimensional trivial character of $G$ in this decomposition, $\mathrm{Hom}_{V^G}(V^G,V)=\CC\iota_V$ where $\iota_V$ is the inclusion, and thus $V$ is haploid. Next we define the $V^G$-module homomorphism $\varepsilon_V: V\rightarrow V^G$ to be projection onto $V^G$ with respect to the decomposition \eqref{SW-duality}. Then $\varepsilon_V\iota_V$ is the identity on $V^G$, while $\iota_V\varepsilon_V$ is projection onto the subspace of $G$-fixed points in $V$ and hence equals $\frac{1}{\vert G\vert}\sum_{g\in G} g$.

To verify the rigidity and dimension conditions of Assumption \ref{mainassum}, we first note that $\cC$ includes each irreducible $V^G$-module $V_\chi$ because $\cC$ is abelian and includes $V$. Then the assumptions of \cite[Corollary 4.8]{McR} hold, so there is a fully faithful braided tensor functor
\begin{equation*}
 \Phi: \rep_{\ZZ/2\ZZ}\,G\rightarrow\cC
\end{equation*}
such that $\Phi(M_\chi^*)\cong V_\chi$ for $\chi\in\widehat{G}$. Here $\rep_{\ZZ/2\ZZ}\,G$ is the tensor category of finite-dimensional $G$-modules with the usual symmetric braiding on $M_\chi\otimes M_\psi$ for $\chi,\psi\in\widehat{G}$ modified by $(-1)^{ij}$ when $M_\chi\otimes V_\chi\subseteq V^{\bar{i}}$ and $M_\psi\otimes V_\psi\subseteq V^{\bar{j}}$ (see \cite[Section 2.2]{McR}). This category is a ribbon tensor category with twist $(-1)^i$ on $M_\chi$ for $\chi\in\widehat{G}$ when $M_\chi\otimes V_\chi\subseteq V^{\bar{i}}$. Because $\Phi$ is fully faithful, it is a braided tensor equivalence from $\rep_{\ZZ/2\ZZ}\,G$ to its image $\cC_V\subseteq\cC$, so that $\cC_V$ inherits the ribbon structure of $\rep_{\ZZ/2\ZZ}\,G$ via $\Phi$. (Note, however, that this ribbon structure does not come from conformal weight gradings unless $V^\even$ is the $\ZZ$-graded part of $V$ and $V^\odd$ is the $(\ZZ+\frac{1}{2})$-graded part.) 

Since $V\cong\bigoplus_{\chi\in\widehat{G}} M_\chi\otimes\Phi(M_\chi^*)$ is an object of $\cC_V$, it is a rigid $V^G$-module. Then since $V$ is simple, \cite[Lemma 1.20]{KO} shows that $V$ is self-dual with evaluation $\varepsilon_V\mu_V: V\tens V\rightarrow V^G$ and some coevaluation $\widetilde{i}_V: V^G\rightarrow V\tens V$. Moreover, we may assume $\widetilde{i}_V$ is even: given a parity decomposition $\widetilde{i}_V = \widetilde{i}_V^\even+\widetilde{i}_V^\odd$, rigidity implies
\begin{align*}
 1_V = r_V(1_V\tens\varepsilon_V\mu_V)\cA_{V,V,V}^{-1}(\widetilde{i}_V^\even\tens 1_V)l_V^{-1} + r_V(1_V\tens\varepsilon_V\mu_V)\cA_{V,V,V}^{-1}(\widetilde{i}_V^\odd\tens 1_V)l_V^{-1}
\end{align*}
where, because $\varepsilon_V\mu_V$ is even, the first and second terms on the right side are the even and odd parts of $1_V$, respectively. Thus the first rigidity axiom holds with $\widetilde{i}_V^\even$ replacing $\widetilde{i}_V$, and similarly for the second rigidity axiom. This verifies the fifth condition of Assumption \ref{mainassum}.

Finally, we need to show $\varepsilon_V\mu_V\widetilde{i}_V=\vert G\vert 1_{V^G}$. It is enough to show that $d_V=\varepsilon_V\mu_V\widetilde{i}_V$ is the categorical dimension of $V$ in the ribbon category $\cC_V$ since by Section 2.2 and Corollary 4.8 of \cite{McR},
\begin{align*}
 \dim_{\cC_V} V & =\sum_{\chi\in\widehat{G}} (\dim_\CC M_\chi)(\dim_{\cC_V} V_\chi) = \sum_{\chi\in\widehat{G}}(\dim_\CC M_\chi)(\dim_{\cC_V} \Phi(M_\chi^*))\nonumber\\
 & = \sum_{\chi\in\widehat{G}}(\dim_\CC M_\chi)(\dim_{\rep_{\ZZ/2\ZZ}\,G} M_\chi^*)=\sum_{\chi\in\widehat{G}}(\dim_\CC M_\chi)(\dim_\CC M_\chi^*)\nonumber\\
 & =\sum_{\chi\in\widehat{G}}\dim_\CC \mathrm{End}\,M_\chi =\dim_\CC \CC[G]=\vert G\vert. 
\end{align*}
Because $\widetilde{i}_V$ and $\mu_V$ are even, we have $d_V=d_{V^\even }+d_{V^\odd }$ where $d_{V^{\bar{i}}}$ is the composition
\begin{equation*}
 V^G\xrightarrow{\widetilde{i}_V} V\tens V\xrightarrow{p^{\bar{i}}} V^{\bar{i}}\tens V^{\bar{i}}\xrightarrow{\mu_V} V^\even\xrightarrow{\varepsilon_V} V^G
\end{equation*}
for $i=0,1$, with $p^{\bar{i}}$ the canonical projection. We need to show that
\begin{equation}\label{dim_identity}
 d_{V^{\bar{i}}} =\dim_{\cC_V} V^{\bar{i}}
\end{equation}
for $i=0,1$, with the categorical dimension defined as usual to be
\begin{equation*}
 V^G\xrightarrow{i_{V^{\bar{i}}}} V^{\bar{i}}\tens V^{\bar{i}}\xrightarrow{\theta_{V^{\bar{i}}}\tens 1_{V^{\bar{i}}}} V^{\bar{i}}\tens V^{\bar{i}} \xrightarrow{\cR_{V^{\bar{i}},V^{\bar{i}}}} V^{\bar{i}}\tens V^{\bar{i}}\xrightarrow{e_{V^{\bar{i}}}} V^G,
\end{equation*}
where $e_{V^{\bar{i}}}$ and $i_{V^{\bar{i}}}$ are an evaluation and coevaluation for $V^{\bar{i}}$, respectively. Because $\widetilde{i}_V$ and $\varepsilon_V\mu_V$ are even, we can take $e_{V^{\bar{i}}}=\varepsilon_V\mu_V\vert_{V^{\bar{i}}\tens V^{\bar{i}}}$ and $i_{V^{\bar{i}}}=p^{\bar{i}} \widetilde{i}_V$. Then \eqref{dim_identity} follows because the twists satisfy $\theta_{V^{\bar{i}}}=(-1)^i$ and supercommutativity of $V$ implies $\mu_V\cR_{V^{\bar{i}},V^{\bar{i}}}=(-1)^i\mu_V\vert_{V^{\bar{i}}\tens V^{\bar{i}}}$.
\end{proof}

Before relating $\cC$ to the equivariantization of the braided $G$-crossed supercategory $\rep\,V$, we discuss why the monoidal structure on $\rep\,V$ is natural from a vertex algebraic point of view. For more details, see \cite[Section 3.5]{CKM}. Given three modules $W_1$, $W_2$, and $W_3$ in $\rep\,V$, we say that an even or odd $V^G$-module intertwining operator $\cY$ of type $\binom{W_3}{W_1\,W_2}$ is a \textit{$V$-intertwining operator} if for any $w_2\in W_2$, $w_3'\in W_3'$, and parity-homogeneous $v\in V$, $w_1\in W_1$, the multivalued analytic functions
\begin{align*}
 (-1)^{\vert\cY\vert\vert v\vert}\langle w_3', Y_{W_3}(v,z_1)\cY(w_1,z_2)w_2\rangle, & \qquad\vert z_1\vert>\vert z_2\vert>0,\nonumber\\
 (-1)^{\vert v\vert\vert w_1\vert}\langle w_3', \cY(w_1,z_2)Y_{W_2}(v,z_1)w_2\rangle, & \qquad \vert z_2\vert >\vert z_1\vert>0,\nonumber\\
 \langle w_3', \cY(Y_{W_1}(v,z_1-z_2)w_1,z_2)w_2\rangle, & \qquad \vert z_2\vert>\vert z_1-z_2\vert>0
\end{align*}
defined on the indicated regions have equal restrictions to their common domains, with specified equalities of certain single-valued branches on certain simply-connected domains. Such intertwining operators correspond precisely to the categorical $\rep\,V$-intertwining operators of type $\binom{W_3}{W_1\,W_2}$
defined in Section \ref{subsec:Superalgebras}. When $W_1$ is $g_1$-twisted, $W_2$ is $g_2$-twisted, and $W_3$ is $g_1g_2$-twisted (recall Proposition \ref{prop:TensProdandGgrading}), it is natural to call $\cY$ a \textit{twisted intertwining operator}.
\begin{rema}
For previous definitions of twisted intertwining operator, see \cite{Xu} (for commuting $g_1$ and $g_2$) and \cite{Huang-TwistIntwOps} (for general $g_1$ and $g_2$). When $g_1$ and $g_2$ commute, the definition here agrees with that of \cite{Xu} (see \cite[Theorem 3.6]{DLXY}, whose proof uses a slight modification of \cite[Theorem 3.53]{CKM}). Whether the definition of twisted intertwining operator given here is equivalent to that of \cite{Huang-TwistIntwOps} for general $g_1$ and $g_2$ is a question we plan to address in a future publication.
\end{rema}

With the definition of twisted intertwining operator given here, Proposition \ref{prop:TensProdandGgrading} and \cite[Proposition 3.50]{CKM} show that the tensor product in $\rep\,V$ satisfies a natural vertex algebraic universal property. If $W_1$ is a $g_1$-twisted $V$-module and $W_2$ is a $g_2$-twisted $V$-module, then the tensor product $W_1\tens_V W_2$ is a $g_1g_2$-twisted $V$-module equipped with a canonical even twisted intertwining operator $\cY_{W_1,W_2}$ of type $\binom{W_1\tens_V W_2}{W_1\,W_2}$ corresponding to the categorical intertwining operator $I_{W_1,W_2}: W_1\tens W_2\rightarrow W_1\tens_V W_2$. Then if $W_3$ is any $g_1g_2$-twisted $V$-module and $\cY$ any twisted intertwining operator of type $\binom{W_3}{W_1\,W_2}$, there is a unique $V$-homomorphism
\begin{equation*}
 f: W_1\tens_V W_2\rightarrow W_3
\end{equation*}
such that $\cY=f\circ\cY_{W_1,W_2}$. This universal property is comparable to the one in \cite[Definition 4.15]{HLZ3} satisfied by the $P(z)$-tensor product of (untwisted) $V$-modules.

We can also naturally describe the unit and associativity isomorphisms in $\rep\,V$ using intertwining operators. From \cite[Section 3.5.4]{CKM}, the left and right unit isomorphisms
\begin{equation*}
 l^V_W: V\tens_V W \rightarrow W,\qquad r^V_W: W\tens_V V\rightarrow W
\end{equation*}
associated to a module $(W, Y_W)$ in $\rep\,V$ are characterized by
\begin{equation*}
 l^V_W(\cY_{V,W}(v,x)w) = Y_W(v,x)w,\qquad r^V_W(\cY_{W,V}(w,x)v)= (-1)^{\vert v\vert\vert w\vert}e^{xL(-1)}Y_W(v, e^{-\pi i} x)w
\end{equation*}
for parity-homogeneous $v\in V$, $w\in W$. Note that for the right unit isomorphisms, we need to specify the branch of $\log(-1)$ used for the substitution $x\mapsto -x$ since $Y_W$ may involve non-integral powers of $x$. 

For three modules $W_1$, $W_2$, and $W_3$ in $\rep\,V$, \cite[Proposition 3.62]{CKM} shows that the associativity isomorphism
\begin{equation*}
 \cA^V_{W_1,W_2,W_3}: W_1\tens_V(W_2\tens_V W_3)\rightarrow (W_1\tens_V W_2)\tens_V W_3
\end{equation*}
is characterized by the equality
\begin{align*}
& \left\langle w', \overline{\cA^V_{W_1,W_2,W_3}} \left(\cY_{W_1, W_2\tens_V W_3}(w_1,e^{\ln r_1})\cY_{W_2,W_3}(w_2,e^{\ln r_2})w_3\right)\right\rangle\nonumber\\
&\qquad\qquad\qquad\qquad=\left\langle w', \cY_{W_1\tens_V  W_2, W_3}\big(\cY_{W_1,W_2}(w_1,e^{\ln(r_1-r_2)})w_2,e^{\ln r_2}\big)w_3\right\rangle 
\end{align*}
for $w_1\in W_1$, $w_2\in W_2$, $w_3\in W_3$, and $w'$ in the contragredient module $\left((W_1\tens_V W_2)\tens_V W_3\right)'$. Here $r_1$ and $r_2$ are any positive real numbers that satisfy $r_1>r_2>r_1-r_2>0$.

We also describe the braided $G$-crossed supercategory structure on $\rep\,V$ of Theorem \ref{Gcrossedfromtwist} with intertwining operators. For the $G$-action on $\rep\,V$,
\begin{equation*}
 T_g(W, Y_W) = (W, Y_W\circ(g^{-1}\otimes 1_W))
\end{equation*}
for $g\in G$. Then as in \cite[Section 3.5.5]{CKM}, for a $g$-twisted $V$-module $W_1$ and any $W_2$ in $\rep\,V$, the braiding isomorphism
\begin{equation*}
 \cR^V_{W_1,W_2}: W_1\tens_V W_2\rightarrow T_g(W_2)\tens_V W_1
\end{equation*}
and its inverse are characterized by
\begin{equation*}
 (\cR^V_{W_1,W_2})^{\pm 1}(\cY_{W_1,W_2}(w_1,x)w_2) = (-1)^{\vert w_1\vert \vert w_2\vert} e^{xL(-1)}\cY_{T_g(W_2),W_1}(w_2, e^{\pm \pi i} x) w_1
\end{equation*}
for parity-homogeneous $w_1\in W_1$, $w_2\in W_2$.

Now we prove the final theorem of this section; for $\cC$ rigid and semisimple, it has appeared as \cite[Theorem 1.5]{KirillovOrbifoldII} and \cite[Theorem 3.12]{Mu2}. Here we assume only the existence of a suitable tensor category of $V^G$-modules.
\begin{theo}\label{thm:mainVOAthm2}
 Let $V$ be a simple vertex operator superalgebra, $G$ a finite automorphism group of $V$ that includes $P_V$, and $\cC$ an abelian category of strongly-graded generalized $V^G$-modules that includes $V$ and admits vertex tensor category structure as in \cite{HLZ1}-\cite{HLZ8}. Then the induction functor $\cF: \cC\rightarrow(\rep\,V)^{G}$ is an equivalence of braided tensor categories.
\end{theo}

The proof requires a generalization of \cite[Lemma 3.1]{DM1}:
\begin{lemma}\label{lem:Dong-Mason_lemma}
 In the setting of Theorem \ref{thm:mainVOAthm2}, in particular assuming $V$ is simple, let $W$ be a module in $\rep\,V$, $\lbrace v^{(i)}\rbrace_{i=1}^I\subseteq V$ a set of linearly-independent $L(0)$-eigenvectors, and $\lbrace w^{(i)}\rbrace_{i=1}^I\subseteq W$ a set of parity-homogeneous (non-zero) $L(0)$-eigenvectors. Then
 \begin{equation*}
  \sum_{i=1}^I Y_W(v^{(i)},x)w^{(i)}\neq 0.
 \end{equation*}
\end{lemma}
\begin{proof}
 We will show that if $\sum_{i=1}^I Y_W(v^{(i)},x)w^{(i)} = 0$ when the $v^{(i)}$ are linearly independent $L(0)$-eigenvectors and the $w^{(i)}$ are parity-homogeneous and contained in $L(0)$-eigenspaces of $W$, then the $w^{(i)}$ must all be zero.
 
 If the sum is zero, then also
\begin{equation*}
 0=\sum_{i=1}^I Y_W(v^{(i)}, e^{\ln r_2})w^{(i)}\in\overline{W}=\prod_{h\in\CC} W_{[h]}
\end{equation*}
for any fixed $r_2\in\RR_+$. We first show that the sum is still zero after replacing each $v^{(i)}$ with $u_n v^{(i)}$ for any $u\in V$ and $n\in\ZZ$. This uses the associativity of $Y_W$ from the proof of Theorem \ref{thm:tw_mod_defs}: for $r_1\in\RR_+$ such that $r_1>r_2>r_1-r_2>0$ and $w'\in W'$,
\begin{align*}
 \sum_{n\in\ZZ}\sum_{i=1}^I \langle w', Y_W(u_n v^{(i)}, e^{\ln r_2})w^{(i)}\rangle (r_1-r_2)^{-n-1} & = \sum_{i=1}^I \langle w', Y_W(Y(u,r_1-r_2)v^{(i)}, e^{\ln r_2})w^{(i)}\rangle \nonumber\\
 & =\left\langle w', Y_W(u, e^{\ln r_1})\sum_{i=1}^I Y_W(v^{(i)},e^{\ln r_2})w^{(i)}\right\rangle = 0.
\end{align*}
Thus the Laurent series $\sum_{n\in\ZZ} \sum_{i=1}^I \langle w', Y_W(u_n v^{(i)}, e^{\ln r_2})w^{(i)}\rangle z_0^{-n-1}$, which converges absolutely in the region $0<\vert z_0\vert<r_2$, is identically zero on a non-empty open interval of the real line, and hence is identically zero on its entire domain. Then each coefficient 
\begin{equation*}
 \sum_{i=1}^I \langle w', Y_W(u_n v^{(i)}, e^{\ln r_2})w^{(i)}
\end{equation*}
of the Laurent series is zero.

We have not yet used the linear indepedence or conformal weight homogeneity of the $v^{(i)}$, so we can iterate this argument to show that, if $A\subseteq\mathrm{End}_\CC\,V$ is the subalgebra generated by the $u_n$ for $u\in V$ and $n\in\ZZ$, then 
\begin{equation*}
 \sum_{i=1}^I \langle w', Y_W(a\cdot v^{(i)}, e^{\ln r_2})w^{(i)}=0
\end{equation*}
for any $a\in A$. Letting $A_0\subseteq A$ denote the subalgebra of conformal-weight-grading-preserving operators,  each conformal weight space $V_{(n)}$ is a (finite-dimensional) irreducible $A_0$-module because $V$ is simple. Moreover, they are inequivalent $A_0$-modules because $L(0)\in A_0$ acts differently on each one. Thus the finitely many $v^{(i)}$ are contained in a finite-dimensional completely-reducible $A_0$-module. Then the Jacobson Density Theorem (see for example \cite[Section 4.3]{Ja}) implies that for any $i\in\lbrace 1,\ldots, I\rbrace$, there is some $a_i\in A_0$ such that $a_i\cdot v^{(j)}=\delta_{i,j} v^{(j)}$ for all $j\in\lbrace 1,\ldots I\rbrace$. In particular,
\begin{equation*}
 Y_W(v^{(i)}, e^{\ln r_2})w^{(i)}=0
\end{equation*}
for each $i$. Now using the assumption that $v^{(i)}$ and $w^{(i)}$ are $L(0)$-eigenvectors, we also have
\begin{align*}
 0  & =\left(\frac{x}{e^{\ln r_2}}\right)^{L(0)}Y_W(v^{(i)}, e^{\ln r_2})w^{(i)}\nonumber\\
 &= Y_W\left(\left(\frac{x}{r_2}\right)^{L(0)} v^{(i)}, x\right)\left(\frac{x}{e^{\ln r_2}}\right)^{L(0)} w^{(i)}\nonumber\\
 & =\left(\frac{x}{e^{\ln r_2}}\right)^{\mathrm{wt}\,v^{(i)}+\mathrm{wt}\,w^{(i)}} Y_W(v^{(i)}, x)w^{(i)},
\end{align*}
so that $Y_W(v^{(i)}, x)w^{(i)} =0$ for each $i$.

Now for each $i$, the annihilator $$\mathrm{Ann}_V(w^{(i)}) =\lbrace v\in V\,\vert\,Y_W(v,x)w^{(i)}=0\rbrace$$ is non-zero, containing $v^{(i)}$. But because each $w^{(i)}$ is parity-homogeneous, $\mathrm{Ann}_V(w^{(i)})$ is a (two-sided) ideal of $V$ by \cite[Lemma 3.73]{CKM}. Since $V$ is simple, this means $\mathrm{Ann}_V(w^{(i)})=V$, forcing
\begin{equation*}
 w^{(i)}=Y_W(\vac, x)w^{(i)}=0
\end{equation*}
for all $i$.
\end{proof}

Now we proceed with the proof of Theorem \ref{thm:mainVOAthm2}:

\begin{proof}
 Since induction is a braided tensor functor by Theorems \ref{thm:F_braided} and \ref{thm:F_even_braided}, we just need to show it is an equivalence of categories. For this we use the $G$-invariants functor from $(\rep\,V)^{G}$ to $\cC$:
\begin{itemize}
 \item For an object $(W, Y_W, \varphi_W)$ in $(\rep\,V)^{G}$, we define
 \begin{equation*}
  W^G = \lbrace w\in W\,\vert\,\varphi_W(g)w=w\,\,\mathrm{for\,all}\,\,g\in G\rbrace.
 \end{equation*}
Since $\varphi_W(g)\circ Y_W =Y_W\circ(g\otimes\varphi_W(g))$ for $g\in G$, each $\varphi_W(g)$ is a $V^G$-module endomorphism. Then $W^G$, as the image of $\frac{1}{\vert G\vert}\sum_{g\in G}\varphi_W(g)$, is an object of $\cC$ because $\cC$ is abelian.

\item For a morphism $f: (W_1,Y_{W_1},\varphi_{W_1})\rightarrow(W_2,Y_{W_2},\varphi_{W_2})$ in $(\rep\,V)^{G}$, we define $f^G=f\vert_{W_1^G}$. Since $f$ intertwines the $G$-actions on $W_1$ and $W_2$, the image of $f^G$ is contained in $W_2^G$. Hence
\begin{equation*}
 f^G: W_1^G\rightarrow W_2^G
\end{equation*}
is a morphism in $\cC$.
\end{itemize}
Now to show that induction is an equivalence of categories, we will find natural isomorphisms $\cF(W)^G\cong W$ for $W$ in $\cC$ and $\cF(W^G)\cong W$ for $(W,Y_W,\varphi_W)$ in $(\rep\,V)^{G}$.

First if $W$ is an object of $\cC$, then $\cF(W)=V\tens W$ and $\varphi_{\cF(W)}(g)=g\tens 1_W$ for $g\in G$. Thus 
\begin{equation*}
 l_W(\varepsilon_V\tens 1_W)\vert_{\cF(W)^G}: \cF(W)^G\rightarrow W
\end{equation*}
is a natural isomorphism, with inverse $(\iota_V\tens 1_W)l_W^{-1}$ because $\varepsilon_V\iota_V=1_{V^G}$ and
\begin{equation*}
 (\iota_V\varepsilon_V\tens 1_W)\vert_{\cF(W)^G}=\frac{1}{\vert G\vert}\sum_{g\in G} g\tens 1_W\vert_{\cF(W)^G} =\frac{1}{\vert G\vert}\sum_{g\in G}\varphi_{\cF(W)}(g)\vert_{\cF(W)^G} = 1_{\cF(W)^G}.
\end{equation*}

Now if $(W, \mu_W,\varphi_W)$ is an object of $(\rep\,V)^{G}$, let $\iota_W: W^G\rightarrow W$ denote the inclusion. Note that $\varphi_W(g)\iota_W=\iota_W$ for all $g\in G$. We take the $V^G$-module homomorphism
\begin{equation*}
 \Psi_W=\mu_W(1_V\tens\iota_W): V\tens W^G\rightarrow W.
\end{equation*}
The associativity of $\mu_W$ implies $\Psi_W$ is a morphism in $\rep\,V$, and $\Psi_W$ is a morphism in $(\rep\,V)^{G}$ because
\begin{align*}
 \varphi_W(g)\Psi_W & =\varphi_W(g)\mu_W(1_V\tens\iota_W)\nonumber\\
 & =\mu_W(g\tens\varphi_W(g))(1_V\tens\iota_W)\nonumber\\
 & =\mu_W(1_V\tens\iota_W)(g\tens 1_{W^G})\nonumber\\
 & = \Psi_W\varphi_{\cF(W^G)}(g)
\end{align*}
for $g\in G$. The homomorphisms $\Psi_W$ are natural because if $f: W_1\rightarrow W_2$ is a morphism in $(\rep\,V)^{G}$, then
\begin{align*}
 \Psi_{W_2}\cF(f^G) &= \mu_{W_2}(1_V\tens \iota_{W_2})(1_V\tens f\vert_{W_1^G})\nonumber\\
 & = \mu_{W_2}(1_V\tens f)(1_V\tens\iota_{W_1})\nonumber\\
 & = f\mu_{W_1}(1_V\tens\iota_{W_1})\nonumber\\
 & = f\Psi_{W_1}.
\end{align*}
We need to show that each $\Psi_W$ is actually an isomorphism.

As a $V^G$-module, $\cF(W^G)=\bigoplus_{\chi\in\widehat{G}} V^\chi\tens W^G$, where $V^\chi=M_\chi\otimes V_\chi$ is the sum of all $G$-modules isomorphic to $M_\chi$ in $V$. Also $W$ is a semisimple $G$-module because it is a $V^G$-module with finite-dimensional $L(0)$-generalized eigenspaces and because $L(0)$ commutes with each $\varphi_W(g)$. So $W=\bigoplus_{\chi\in\widehat{G}} W^\chi$ where $W^\chi$ is the sum of all $G$-submodules of $W$ isomorphic to $M_\chi$. As $\Psi_W$ intertwines the $G$-actions on $\cF(W^G)$ and $W$, it maps each $V^\chi\tens W^G$ to $W^\chi$. Moreover,
\begin{equation*}
 \Psi_W\vert_{V^G\tens W^G}: V^G\tens W^G\rightarrow W^G
\end{equation*}
is an isomorphism, since it amounts to $l_{W^G}$ by the unit property of $\mu_W$. Consequently, the kernel and cokernel of $\Psi_W$ are objects of $(\rep\,V)^{G}$ with no $G$-invariants.

To complete the proof, we show that any object $W$ of $(\rep\,V)^{G}$ with $W^G=0$ is itself $0$; equivalently, if $W\neq 0$, then $W^G\neq 0$ as well. As before, $W=\bigoplus_{\chi\in\widehat{G}} W^\chi$ where $W^\chi$ is the sum of all $G$-submodules of $W$ that are isomorphic to $M_\chi$. If $W\neq 0$, then $W^\chi\neq 0$ for some $\chi$;  let $\chi^*$ to denote the character of $G$ dual to $\chi$. Now choose a basis $\lbrace v^{(i)}\rbrace_{i=1}^I\subseteq V^{\chi^*}_{(n)}$ for some copy of $M_{\chi^*}$ contained in a non-zero homogeneous subspace of $V$. Then choose $\lbrace w^{(i)}\rbrace_{i=1}^I\subseteq W^\chi_{[h]}$ to be a dual basis for some copy of $M_\chi$ contained in some non-zero homogeneous subspace of $W$. Although $L(0)$ might not act semisimply on $W$, the $L(0)$-eigenspace of $W^\chi$ with eigenvalue $h$ will be non-zero, so we may assume the $w^{(i)}$ are $L(0)$-eigenvectors. Moreover, because $\varphi_W(P_V)=P_W$, the $G$-submodule $W^\chi$ is either purely even or purely odd, so the $w^{(i)}$ are parity-homogeneous. We now apply Lemma \ref{lem:Dong-Mason_lemma} to conclude that
\begin{equation*}
 \sum_{i=1}^I Y_W(v^{(i)},x)w^{(i)}\neq 0.
\end{equation*}
But we have chosen the $v^{(i)}$ and $w^{(i)}$ so that $\sum_{i=1}^I v^{(i)}\otimes w^{(i)}\in (V\otimes W)^G$. Thus because each coefficient of $Y_W$ provides a $G$-module homomorphism from $V\otimes W$ to $W$, we have
\begin{equation*}
 \sum_{i=1}^I Y_W(v^{(i)},x)w^{(i)}\in W^G[\log x]\lbrace x\rbrace
\end{equation*}
and $W^G\neq 0$.
%
%
\end{proof}

\subsection{\texorpdfstring{$\ZZ/2\ZZ$}{Z/2Z}-equivariantization for superalgebras}\label{subsec:VOSAs}

Here we discuss the implications of Theorems \ref{thm:mainVOAthm1} and \ref{thm:mainVOAthm2} in perhaps the simplest non-trivial case: $V$ is a vertex operator superalgebra and $G=\langle P_V\rangle\cong\ZZ/2\ZZ$ so that $V^G=V^\even$. Let $V$ be simple and $\cC$ an abelian category of strongly-graded generalized $V^\even$-modules that includes $V$ and admits vertex tensor category structure as in \cite{HLZ1}-\cite{HLZ8}. By Theorem \ref{thm:mainVOAthm2}, $\cC$ is braided tensor equivalent to the $\ZZ/2\ZZ$-equivariantization of $\rep\,V$, which by Theorem \ref{thm:mainVOAthm1} is the category of untwisted and parity-twisted $V$-modules (referred to in the physics literature as the Neveu-Schwarz and Ramond sectors, respectively). Here we explicitly describe $(\rep\,V)^{\ZZ/2\ZZ}$.

\textbf{Objects.} By Theorem \ref{thm:mainVOAthm1}, the objects of $\rep\,V$ are (direct sums of) untwisted and parity-twisted $V$-modules. Then objects of $(\rep\,V)^{\ZZ/2\ZZ}$ are such modules with the additional data of a $\ZZ/2\ZZ$-action; however, $P_V$ must act as $P_W$ on a module $W$ in $(\rep\,V)^{\ZZ/2\ZZ}$, so the additional data is simply the parity decomposition of $W$.

\textbf{Morphisms.} Morphisms in $(\rep\,V)^{\ZZ/2\ZZ}$ are homomorphisms of (twisted) $V$-modules that also preserve parity decompositions, that is, they must be even. This means that $(\rep\,V)^{\ZZ/2\ZZ}$ is the underlying category of the supercategory $\rep\,V$.

\textbf{Tensor product functor.} The tensor product $W_1\tens_V W_2$ of two (twisted) $V$-modules is characterized a universal property: There is an (even) twisted intertwining operator $\cY_{W_1,W_2}$ of type $\binom{W_1\tens_V W_2}{W_1\,W_2}$ such that for any (twisted) $V$-module $W_3$ and (even) twisted intertwining operator $\cY$ of type $\binom{W_3}{W_1\,W_2}$, there is a unique homomorphism
\begin{equation*}
 f: W_1\tens_V W_2\rightarrow W_3
\end{equation*}
such that $f\circ\cY_{W_1,W_2}=\cY$. 


The tensor product of two homomorphisms $f_1: W_1\rightarrow\widetilde{W}_1$ and $f_2: W_2\rightarrow\widetilde{W}_2$ in $(\rep\,V)^{\ZZ/2\ZZ}$ is induced by the intertwining operator $\cY_{\widetilde{W}_1,\widetilde{W}_2}\circ(f_1\otimes f_2)$ of type $\binom{\widetilde{W}_1\tens_V\widetilde{W}_2}{W_1\,W_2}$ and the universal propertyof $W_1\tens_V W_2$.

\textbf{Unit isomorphisms.} The unit object of $(\rep\,V)^{\ZZ/2\ZZ}$ is $V$ and for any module $(W,Y_W)$ in $(\rep\,V)^{\ZZ/2\ZZ}$, the left and right unit isomorphisms are characterized respectively by
\begin{equation*}
 l^V_W(\cY_{V,W}(v,x)w) = Y_W(v,x)w,\qquad r^V_W(\cY_{W,V}(w,x)v)= (-1)^{\vert v\vert\vert w\vert}e^{xL(-1)}Y_W(v, e^{-\pi i} x)w
\end{equation*}
for parity-homogeneous $v\in V$, $w\in W$.

\textbf{Associativity isomorphisms.} For three modules $W_1$, $W_2$, and $W_3$ in $(\rep\,V)^{\ZZ/2\ZZ}$, the associativity isomorphism $\cA^V_{W_1,W_2,W_3}$ is characterized by the equality
\begin{align*}
& \left\langle w', \overline{\cA^V_{W_1,W_2,W_3}} \left(\cY_{W_1, W_2\tens_V W_3}(w_1,e^{\ln r_1})\cY_{W_2,W_3}(w_2,e^{\ln r_2})w_3\right)\right\rangle\nonumber\\
&\qquad\qquad\qquad\qquad=\left\langle w', \cY_{W_1\tens_V W_2, W_3}\big(\cY_{W_1,W_2}(w_1,e^{\ln(r_1-r_2)})w_2,e^{\ln r_2}\big)w_3\right\rangle 
\end{align*}
for $w_1\in W_1$, $w_2\in W_2$, $w_3\in W_3$, and $w'\in\left((W_1\tens_V W_2)\tens_V W_3\right)'$, and $r_1, r_2\in\RR_+$ satisfy $r_1>r_2>r_1-r_2>0$.

\textbf{Braiding isomorphisms.} If $W_1$ is untwisted and $W_2$ is any module in $(\rep\,V)^{\ZZ/2\ZZ}$, the braiding isomorphism $\widetilde{\cR}^V_{W_1,W_2}$ is given by
\begin{equation*}
 \widetilde{\cR}^V_{W_1,W_2}(\cY_{W_1,W_2}(w_1,x)w_2) = (-1)^{\vert w_1\vert \vert w_2\vert} e^{xL(-1)}\cY_{W_2,W_1}(w_2, e^{\pi i} x) w_1
\end{equation*}
for parity-homogeneous $w_1\in W_1$, $w_2\in W_2$. If $W_1$ is parity-twisted, then
\begin{align*}
 \widetilde{\cR}^V_{W_1,W_2}(\cY_{W_1,W_2}(w_1,x)w_2) & = (P_{W_2}\tens_V 1_{W_1})\left(\cR^V_{W_1,W_2}(\cY_{W_1,W_2}(w_1,x)w_2)\right)\nonumber\\
 & =(P_{W_2}\tens_V 1_{W_1})\left((-1)^{\vert w_1\vert\vert w_2\vert} e^{xL(-1)}\cY_{P_V(W_2),W_1}(w_2, e^{\pi i}x)w_1\right)\nonumber\\
 & =(-1)^{\vert w_1\vert\vert w_2\vert} e^{xL(-1)}\cY_{W_2,W_1}(P_{W_2}(w_2),e^{\pi i}x)w_1
\end{align*}
for parity-homogeneous $w_1\in W_1$, $w_2\in W_2$. Recall that $P_V(W_2)$ is the superspace $W_2$ with vertex operator $Y_{W_2}(P_V(\cdot),x)$.

This is a complete description of $\cC$ as a braided tensor category, assuming one understands untwisted and parity-twisted $V$-modules and the twisted intertwining operators among them, since $\cC$ is braided tensor equivalent to $(\rep\,V)^{\ZZ/2\ZZ}$ by Theorem \ref{thm:mainVOAthm2}. For example, the following is a simple consequence of Theorem \ref{thm:mainVOAthm2} in this setting:
\begin{corol}
 Let $V$ be a simple vertex operator superalgebra and $\cC$ an abelian category of strongly-graded generalized $V^\even$ modules that includes $V$ and admits vertex tensor category structure as in \cite{HLZ1}-\cite{HLZ8}. Then every indecomposable $V^\even$-module in $\cC$ is the even summand of an untwisted or parity-twisted $V$-module.
\end{corol}
\begin{proof}
 From the proof of Theorem \ref{thm:mainVOAthm2}, any (indecomposable) $V^\even$-module $W$ in $\cC$ is isomorphic to the even part of $V\tens W$. Since $W$ is indecomposable, this even part cannot be divided between non-zero untwisted and twisted summands of $V\tens W$. So $W$ is the even summand of either an untwisted or parity-twisted $V$-module in $\cC$.
\end{proof}

If $V^\even$ is $C_2$-cofinite and non-negatively graded, with $V^\even_{(0)}=\CC\vac$ (that is, $V^\even$ has positive energy/is CFT-type), then the full category of strongly-graded generalized $V^\even$-modules has vertex tensor category structure \cite{H-cofin}. Thus all our results apply when $V$ is simple, $G=\langle P_V\rangle$, and $V^\even$ is $C_2$-cofinite and positive energy. Examples of such $V$ with non-semisimple modules (that is, they are not rational) include the symplectic fermion vertex operator superalgebras $SF(d)$, $d\in\ZZ_+$, of $d$ pairs of symplectic fermions \cite{Ka, Ab, Ru}. A major motivation for this paper was Runkel's construction \cite{Ru} of a braided tensor category that is conjecturally equivalent to the braided tensor category of strongly-graded generalized $SF(d)^\even$-modules. The category in \cite{Ru} seems to be the equivariantization of the braided $\ZZ/2\ZZ$-crossed supercategory of (twisted) $SF(d)$-modules that we have discussed here. We plan to verify this in future work and thus prove the conjectured equivalence with the category of $SF(d)^\even$-modules. Combined with \cite{GR, FGR}, this would mean that the category of $SF(d)^\even$-modules is braided equivalent to the category of finite-dimensional representations of a quasi-Hopf algebra and is a (non-semisimple) modular tensor category.

\subsection{Application to orbifold rationality}\label{sec:OrbRat}

We say that a vertex operator algebra $V$ is \textit{strongly rational} if it satisfies the following conditions:
\begin{itemize}
 \item $V$ is simple and self-contragredient.
 \item Positive energy: $V_{(n)}=0$ for $n<0$ and $V_{(0)}=\CC\vac$. (Such $V$ is also said to be CFT-type.)
 \item $C_2$-cofiniteness: $\dim V/C_2(V)<\infty$ where $C_2(V)=\mathrm{span}\,\lbrace u_{-2} v\,\vert\,u,v\in V\rbrace$.
 \item Rationality: Every $\NN$-gradable weak $V$-module $W=\bigoplus_{n\in\NN} W(n)$ (where the $W(n)$ could be infinite dimensional) is a direct sum of simple strongly-graded $V$-modules.
\end{itemize}
The orbifold rationality problem asks whether strongly rationality of $V$ implies strong rationality of $V^G$ when $G$ is a finite automorphism group. The answer is yes for $G$ solvable by \cite{CarM}, but the problem has remained open for general finite $G$. Here we show that Theorem \ref{thm:mainVOAthm1} combined with \cite{CarM} reduces the orbifold rationality problem to the question of $C_2$-cofiniteness for $V^G$.

We first show that if $V$ is strongly rational and $G$ is a finite automorphism group of $V$, then categories of $V^G$-modules that admit vertex tensor category structure are semisimple:
\begin{theo}\label{thm:VG_ss}
 Let $V$ be a strongly rational vertex operator algebra and $G$ any finite group of automorphisms of $V$. If $\mathcal{C}$ is an abelian category of strongly-graded generalized $V^G$-modules that includes $V$ and admits vertex and braided tensor category structure as in \cite{HLZ1}-\cite{HLZ8}, then $\mathcal{C}$ is semisimple.
\end{theo}

For a category $\cC$ of $V^G$-modules as in the statement of the theorem, we will use $\rep_\cC\,V$ to denote the braided $G$-crossed category of twisted $V$-modules in $\cC$ because we will soon need to consider twisted $V$-modules in smaller braided tensor categories. The following lemma reduces semisimplicity of $\cC$ to semisimplicity of $\rep_\cC\,V$:
\begin{lemma}
 If $\rep_\cC\,V$ is semisimple, then $\cC$ is also semisimple.
\end{lemma}
\begin{proof}
 We need to show that if $f: W_1\rightarrow W_2$ is a surjection in $\cC$, then there exists $\sigma: W_2\rightarrow W_1$ such that $f\sigma=1_{W_2}$.  Since the functor $V\tens\bullet$ is right exact, $1_V\tens f: V\tens W_1\rightarrow V\tens W_2$ is a surjection in $\rep_\cC\,V$. Then because $\rep_\cC\,V$ is semisimple, there is some $s: V\tens W_2\rightarrow V\tens W_1$ such that $(1_V\tens f)s=1_{V\tens W_2}$. We define $\sigma: W_2\rightarrow W_1$ in $\cC$ to be the composition
 \begin{align*}
  W_2\xrightarrow{l_{W_2}^{-1}} V^G\tens W_2\xrightarrow{\iota_V\tens 1_{W_2}} V\tens W_2\xrightarrow{s} V\tens W_1\xrightarrow{\varepsilon_V\tens 1_{W_1}} V^G\tens W_1\xrightarrow{l_{W_1}} W_1.
 \end{align*}
Then
\begin{align*}
 f\sigma & =f l_{W_1}(\varepsilon_V\tens 1_{W_1})s(\iota_V\tens 1_{W_2})l_{W_2}^{-1}\nonumber\\
 & = l_{W_2}(\varepsilon_V\tens 1_{W_2})(1_V\tens f)s(\iota_V\tens 1_{W_2})l_{W_2}^{-1}\nonumber\\
 & = l_{W_2}(\varepsilon_V\tens 1_{W_2})(\iota_V\tens 1_{W_2})l_{W_2}^{-1}\nonumber\\
 & = l_{W_2}l_{W_2}^{-1}\nonumber\\
 & = 1_{W_2}
\end{align*}
as desired.
\end{proof}
\begin{rema}
 Here is a less elementary proof of the lemma: If $\rep_\cC\,V$ is semisimple, then so is $(\rep_\cC\,V)^G$ by Maschke's Theorem. Then $\cC$ is semisimple by Theorem \ref{thm:mainVOAthm2}.
\end{rema}

Now we prove the theorem by showing that $\rep_\cC\,V$ is semisimple:
\begin{proof}
Since every object of $\rep_\cC\,V$ is a direct sum of twisted modules by Theorem \ref{thm:mainVOAthm1}, it is enough to show that the category $\rep^g\,V$ of $g$-twisted $V$-modules is semisimple for any $g\in G$. Fix $g\in G$ and let $\mathcal{D}$ be the category of strongly-graded generalized $V^{\langle g\rangle}$-modules. By the main theorem of \cite{CarM}, $V^{\langle g\rangle}$ is strongly rational so that $\mathcal{D}$ is a (semisimple) modular tensor category \cite{H-rigidity}. Moreover, Theorem \ref{thm:mainVOAthm1} shows that the subcategory $\rep_\mathcal{D}\,V\subseteq\rep_{\cC}\,V$ consisting of untwisted $V^{\langle g\rangle}$-modules is a braided $\langle g\rangle$-crossed tensor category and
\begin{equation*}
 \rep_{\mathcal{D}}\,V=\bigoplus_{i=0}^{\vert g\vert-1} \rep^{g^i}\,V.
\end{equation*}
Thus it is enough to show $\rep_\mathcal{D}\,V$ is semisimple. This follows from Lemma 1.20 and Theorem 3.3 of \cite{KO}, since $\mathcal{D}$ is semisimple and rigid, $V$ is a simple algebra in $\mathcal{D}$, and $\dim_\mathcal{D}\,V =\vert g\vert\neq 0$ \cite[Proposition 4.15]{McR}.
\end{proof}

As a corollary, we get strong rationality of $V^G$ from strong rationality of $V$ and $C_2$-cofiniteness of $V^G$:
\begin{corol}\label{cor:OrbRat}
 Let $V$ be a strongly rational vertex operator algebra and $G$ any finite group of automorphisms of $V$. If $V^G$ is $C_2$-cofinite, then $V^G$ is strongly rational.
\end{corol}
\begin{proof}
Positive energy for $V^G$ follows immediately from positive energy for $V$. Since $V$ is also simple, $V^G$ is simple by the main theorem of \cite{DLM}. The self-contragrediency and positive energy of $V$ mean that there is a nondegenerate invariant bilinear form
\begin{equation*}
 (\cdot,\cdot): V\times V\rightarrow\CC
\end{equation*}
such that $(\vac,\vac)\neq 0$. This restricts to a non-zero invariant bilinear form on $V^G$, which must be nondegenerate since $V^G$ is simple. Thus $V^G$ is also self-contragredient.

Since $V^G$ has positive energy and is $C_2$-cofinite, Lemma 3.6 and Proposition 3.7 of \cite{CarM} (see also \cite[Proposition 4.16]{McR}) imply that $V^G$ will be rational if its full category of strongly-graded generalized modules is semisimple. But this category admits vertex tensor category structure by Proposition 4.1 and Theorem 4.11 of \cite{H-cofin}, so Theorem \ref{thm:VG_ss} applies.
\end{proof}

\begin{rema}
The recent preprint \cite{Mi} has proposed an argument for proving $C_2$-cofiniteness of $V^G$ for general finite $G$ and positive energy, self-contragredient, $C_2$-cofinite $V$, but unfortunately it seems to have a gap.
\end{rema}

\appendix

\section{Proof of Theorem \ref{Gcrossedfromtwist}}\label{app:BGC_proof}

We use the notation and setting of Section \ref{sec:TwistMods}.  We have already seen that $\repGV$ is an $\mathbb{F}$-additive supercategory with a $G$-grading, and Corollary \ref{g1g2corol} shows that $\repGV$ has a monoidal structure compatible with the grading. It remains to show that the $G$-action and braiding isomorphisms on $\repGV$ discussed in Section \ref{sec:TwistMods} are well defined and satisfy the required properties.
 
 First we show that $T_g$ is a superfunctor on $\repGV$. To show that $(T_g(W), \mu_{T_g(W)})$ is an object of $\rep V$, we first verify associativity:
\begin{align*}
 \mu_{T_g(W)}(1_V\tens\mu_{T_g(W)}) & = \mu_W(1_V\tens\mu_W)(g^{-1}\tens(g^{-1}\tens 1_W))\nonumber\\
 &=\mu_W(\mu_V\tens 1_W)\cA_{V,V,W}(g^{-1}\tens(g^{-1}\tens 1_W))\nonumber\\
 & =\mu_W(g^{-1}\tens 1_V)(\mu_V\tens 1_W)\cA_{V,V,W}\nonumber\\
 & =\mu_{T_g(W)}(\mu_V\tens 1_V)\cA_{V,V,T_g(W)},
\end{align*}
using $g^{-1}\mu_V=\mu_V(g^{-1}\tens g^{-1})$. For the unit property,
\begin{equation*}
 \mu_{T_g(W)}(\iota_V\tens 1_W) l_{T_g(W)}^{-1} =\mu_W(g^{-1}\tens 1_W)(\iota_V\tens 1_W) l_W^{-1}=\mu_W(\iota_V\tens 1_W) l_W^{-1}=1_W=1_{T_g(W)}
\end{equation*}
using $g^{-1}\iota_V=\iota_V$. A morphism $f: W_1\rightarrow W_2$ in $\rep V$, is still a morphism from $T_g(W_1)$ to $T_g(W_2)$ because
\begin{equation*}
 f\mu_{T_g(W_1)}=f\mu_{W_1}(g^{-1}\tens 1_W)=\mu_{W_2}(1_V\tens f)(g^{-1}\tens 1_W)=\mu_{W_2}(g^{-1}\tens 1_W)(1_V\tens f) =\mu_{T_g(W_2)}(1_V\tens f).
\end{equation*}
Because $g^{-1}$ is even, we avoid a sign factor in the third equality here. Clearly $T_g$ induces an even linear map on morphisms, so $T_g$ is a superfunctor on $\rep\,V$. Then $T_g$ restricts to a superfunctor on $\repGV$ because if $(W,\mu_W)$ is an $h$-twisted $V$-module for $h\in G$, then $(T_g(W),\mu_{T_g(W)})$ is a $g h g^{-1}$-twisted $V$-module. Indeed,
\begin{align*}
 \mu_{T_g(W)}(ghg^{-1}\tens 1_{T_g(W)})\cM_{V,T_g(W)} & =\mu_W(g^{-1}\tens 1_W)(ghg^{-1}\tens 1_W)\cM_{V,W}\nonumber\\
 & =\mu_W(h\tens 1_W)\cM_{V,W}(g^{-1}\tens 1_W)\nonumber\\
 & =\mu_W(g^{-1}\tens 1_W) =\mu_{T_g(W)},
\end{align*}
using the naturality of the monodromy isomorphisms in the second equality.

Next we construct the even natural isomorphism $\tau_g: T_g\circ\tens_V\rightarrow\tens_V\circ(T_g\times T_g)$. For objects $W_1$, $W_2$ in $\rep V$, recall that $(W_1\tens_V W_2, I_{W_1,W_2})$ is the cokernel of the morphism $\mu^{(1)}_{W_1,W_2}-\mu^{(2)}_{W_1,W_2}$ and similarly for $(T_g(W_1)\tens_V T_g(W_2), I_{T_g(W_1),T_g(W_2)})$. We claim that there are unique morphisms
\begin{equation*}
 \tau_{g; W_1, W_2}: W_1\tens_V W_2\rightarrow T_g(W_1)\tens_V T_g(W_2),\hspace{2em} \widetilde{\tau}_{g; W_1,W_2}: T_g(W_1)\tens_V T_g(W_2)\rightarrow W_1\tens_V W_2
\end{equation*}
in $\sC$ such that the diagrams
\begin{equation*}
 \xymatrixcolsep{4pc}
 \xymatrix{
 W_1\tens W_2 \ar[d]^{I_{W_1,W_2}} \ar[rd]^{I_{T_g(W_1),T_g(W_2)}} & \\
 W_1\tens_V W_2 \ar@<.6ex>[r]^(.4){\tau_{g; W_1, W_2}} & T_g(W_1)\tens_V T_g(W_2) \ar@<.6ex>[l]^(.6){\widetilde{\tau}_{g; W_1, W_2}}\\
 }
\end{equation*}
commute. This follows from the universal properties of the cokernels and the equalities
\begin{align*}
 I_{T_g(W_1),T_g(W_2)} & (\mu_{W_1,W_2}^{(1)}-\mu_{W_1,W_2}^{(2)})\nonumber\\
 & =  I_{T_g(W_1),T_g(W_2)}(\mu_{W_1,W_2}^{(1)}-\mu_{W_1,W_2}^{(2)})(g^{-1}\tens 1_{W_1\tens W_2})(g\tens 1_{W_1\tens W_2})\nonumber\\
 & =I_{T_g(W_1),T_g(W_2)}(\mu_{T_g(W_1),T_g(W_2)}^{(1)}-\mu_{T_g(W_1),T_g(W_2)}^{(2)})(g\tens 1_{W_1\tens W_2})= 0
\end{align*}
and
\begin{align*}
I_{W_1, W_2}(\mu_{T_g(W_1),T_g(W_2)}^{(1)}-\mu_{T_g(W_1),T_g(W_2)}^{(2)}) = I_{W_1,W_2}(\mu_{W_1,W_2}^{(1)}-\mu_{W_1, W_2}^{(2)})(g^{-1}\tens 1_{W_1\tens W_2}) = 0.
\end{align*}
These equalities use the definitions of the $\mu^{(i)}$, the naturality of associativity and braiding isomorphisms in $\sC$, and the evenness of all morphisms involved. Now $\tau_{g; W_1, W_2}$ and $\widetilde{\tau}_{g; W_1, W_2}$ are mutual inverses: because
\begin{equation*}
 \widetilde{\tau}_{g; W_1, W_2} \tau_{g; W_1, W_2} I_{W_1,W_2} =\widetilde{\tau}_{g; W_1,W_2} I_{T_g(W_1),T_g(W_2)} =I_{W_1,W_2}
\end{equation*}
and $I_{W_1,W_2}$ is surjective, $\widetilde{\tau}_{g; W_1, W_2} \tau_{g; W_1, W_2}= 1_{W_1\tens_V W_2}$, and similarly $\tau_{g; W_1, W_2}\widetilde{\tau}_{g; W_1,W_2}= 1_{T_g(W_1)\tens_V T_g(W_2)}$. Also, $\tau_{g; W_1, W_2}$ is even because $I_{W_1,W_2}$ and $I_{T_g(W_1),T_g(W_1)}$ are even and surjective.

Now we show that $\tau_{g; W_1, W_2}$ is a morphism in $\rep V$ from $T_g(W_1\tens_V W_2)$ to $T_g(W_1)\tens_V T_g(W_2)$. Then its inverse $\widetilde{\tau}_{g; W_1, W_2}: T_g(W_1)\tens_V T_g(W_2)\rightarrow T_g(W_1\tens_V W_2)$ will also be a morphism in $\rep V$. This uses the commutative diagrams
\begin{equation*}
 \xymatrixcolsep{6pc}
 \xymatrix{
 V\tens(W_1\tens W_2) \ar[r]^(.55){\mu^{(i)}_{W_1,W_2}(g^{-1}\tens 1_{W_1\tens W_2})} \ar[d]^{1_V\tens I_{W_1,W_2}} & W_1\tens W_2 \ar[d]^{I_{W_1,W_2}} \ar[rd]^{I_{T_g(W_1),T_g(W_2)}} & \\
 V\tens T_g(W_1\tens_V W_2) \ar[r]^(.55){\mu_{T_g(W_1\tens_V W_2)}} & T_g(W_1\tens_V W_2) \ar[r]^(.45){\tau_{g; W_1, W_2}} & T_g(W_1)\tens_V T_g(W_2)\\ 
 }
\end{equation*}
for $i=1$ or $i=2$ and
\begin{equation*}
 \xymatrixcolsep{6pc}
 \xymatrix{
 V\tens(W_1\tens W_2) \ar[d]^{1_V\tens I_{W_1,W_2}} \ar[rd]^{1_V\tens I_{T_g(W_1),T_g(W_2)}} & &\\
 V\tens T_g(W_1\tens_V W_2) \ar[r]^(.45){1_V\tens \tau_{g; W_1,W_2}} & V\tens(T_g(W_1)\tens_V T_g(W_2)) \ar[r]^(.55){\mu_{T_g(W_1)\tens_V T_g(W_2)}} & T_g(W_1)\tens_V T_g(W_2)
 } .
\end{equation*}
The top compositions in these two diagrams agree because $I_{T_g(W_1),T_g(W_2)}$ is an intertwining operator and because $\mu^{(i)}_{W_1,W_2}(g^{-1}\tens 1_{W_1\tens W_2})=\mu^{(i)}_{T_g(W_1),T_g(W_2)}$ for $i=1,2$. Thus
\begin{equation*}
 \tau_{g; W_1,W_2}\mu_{T_g(W_1\tens_V W_2)}(1_V\tens I_{W_1,W_2}) =\mu_{T_g(W_1)\tens_V T_g(W_2)}(1_V\tens \tau_{g; W_1,W_2})(1_V\tens I_{W_1,W_2})
\end{equation*}
as well. Since $I_{W_1, W_2}$ is a surjective cokernel morphism and $V\tens\bullet$ is right exact, $1_V\tens I_{W_1,W_2}$ is surjective as well and it follows that $\tau_{g; W_1,W_2}$ is a morphism in $\rep V$.

Next we show that the $\tau_{g; W_1, W_2}$ define a natural isomorphism, that is, for morphisms $f_1: W_1\rightarrow\widetilde{W}_1$ and $f_2: W_2\rightarrow\widetilde{W}_2$ in $\rep V$,
\begin{equation*}
 \tau_{g; \widetilde{W}_1,\widetilde{W}_2} T_g(f_1\tens_V f_2)=(T_g(f_1)\tens_V T_g(f_2)) \tau_{g; W_1, W_2}.
\end{equation*}
This follows from the commutative diagrams
\begin{equation*}
 \xymatrixcolsep{4pc}
 \xymatrix{
 W_1\tens W_2 \ar[r]^{f_1\tens f_2} \ar[d]^{I_{W_1,W_2}} & \widetilde{W}_1\tens\widetilde{W}_2 \ar[d]^{I_{\widetilde{W}_1,\widetilde{W}_2}} \ar[rd]^{I_{T_g(\widetilde{W}_1),T_g(\widetilde{W}_2)}} & \\
 W_1\tens_V W_2 \ar[r]^{f_1\tens_V f_2} & \widetilde{W}_1\tens_V\widetilde{W}_2 \ar[r]^(.4){\tau_{g; \widetilde{W}_1,\widetilde{W}_2}} & T_g(\widetilde{W}_1)\tens_V T_g(\widetilde{W}_2) \\
 }
\end{equation*}
and
\begin{equation*}
 \xymatrixcolsep{5pc}
 \xymatrix{
  & W_1\tens W_2 \ar[ld]_{I_{W_1,W_2}} \ar[r]^{f_1\tens f_2} \ar[d]^{I_{T_g(W_1),T_g(W_2)}} & \widetilde{W}_1\tens\widetilde{W}_2 \ar[d]^{I_{T_g(\widetilde{W}_1),T_g(\widetilde{W}_2)}} \\
 W_1\tens_V W_2 \ar[r]^(.45){\tau_{g; W_1,W_2}} & T_g(W_1)\tens_V T_g(W_2) \ar[r]^{T_g(f_1)\tens_V T_g(f_2)} & T_g(\widetilde{W}_1)\tens_V T_g(\widetilde{W}_2) \\
 } ,
\end{equation*}
as well as the surjectivity of $I_{W_1,W_2}$.

The even natural isomorphism $\tau_g$ needs to be compatible with the associativity isomorphisms in the sense that the diagram
\begin{equation*}
 \xymatrixcolsep{7pc}
 \xymatrix{
 T_g(W_1\tens_V(W_2\tens_V W_3)) \ar[r]^{T_g(\cA^V_{W_1,W_2,W_3})} \ar[d]^{\tau_{g; W_1; W_2\tens_V W_3}} & T_g((W_1\tens_V W_2)\tens_V W_3) \ar[d]^{\tau_{g; W_1\tens_V W_2, W_3}} \\
 T_g(W_1)\tens_V T_g(W_2\tens_V W_3) \ar[d]^{1_{T_g(W_1)}\tens_V \tau_{g; W_2,W_3}} & T_g(W_1\tens_V W_2)\tens_V T_g(W_3) \ar[d]^{\tau_{g; W_1,W_2}\tens_V 1_{T_g(W_3)}} \\
 T_g(W_1)\tens_V(T_g(W_2)\tens_V T_g(W_3)) \ar[r]^{\cA^V_{T_g(W_1),T_g(W_2),T_g(W_3)}} & (T_g(W_1)\tens_V T_g(W_2))\tens_V T_g(W_3)
 } 
\end{equation*}
commutes for any objects $W_1$, $W_2$, and $W_3$ in $\rep V$. For the proof, recall that $T_g(W)=W$ as objects of $\sC$ and $T_g(f)=f$ when $(W,\mu_W)$ is an object and $f$ is a morphism in $\rep V$. Consider the composition
\begin{align*}
 W_1\tens (W_2\tens & W_3)  \xrightarrow{1_{W_1}\tens I_{W_2,W_3}} W_1\tens(W_2\tens_V W_3)\xrightarrow{I_{W_1,W_2\tens_V W_3}} W_1\tens_V(W_2\tens_V W_3)\nonumber\\
 & \xrightarrow{\cA^V_{W_1,W_2,W_3}} (W_1\tens_V W_2)\tens_V W_3\xrightarrow{\tau_{g; W_1\tens_V W_2, W_3}} T_g(W_1\tens_V W_2)\tens_V T_g(W_3)\nonumber\\
 &\xrightarrow{\tau_{g; W_1, W_2}\tens_V 1_{T_g(W_3)}} (T_g(W_1)\tens_V T_g(W_2))\tens_V T_g(W_3).
\end{align*}
By the definition of the associativity isomorphisms in $\rep V$, this equals
\begin{align*}
 W_1\tens(W_2\tens & W_3) \xrightarrow{\cA_{W_1,W_2,W_3}} (W_1\tens W_2)\tens W_3\xrightarrow{I_{W_1,W_2}\tens 1_{W_3}} (W_1\tens_V W_2)\tens W_3\nonumber\\
 &\xrightarrow{I_{W_1\tens_V W_2, W_3}} (W_1\tens_V W_2)\tens_V W_3 \xrightarrow{\tau_{g; W_1\tens_V W_2, W_3}} T_g(W_1\tens_V W_2)\tens_V T_g(W_3)\nonumber\\
 &\xrightarrow{\tau_{g; W_1, W_2}\tens_V 1_{T_g(W_3)}} (T_g(W_1)\tens_V T_g(W_2))\tens_V T_g(W_3).
\end{align*}
Then the definition of $\tau_{g; W_1\tens_V W_2, W_3}$ implies that we get
\begin{align}\label{calc2}
  W_1\tens & (W_2\tens W_3)  \xrightarrow{\cA_{W_1,W_2,W_3}} (W_1\tens W_2)\tens W_3\xrightarrow{I_{W_1,W_2}\tens 1_{W_3}} (W_1\tens_V W_2)\tens W_3\nonumber\\
  &\xrightarrow{I_{T_g(W_1\tens_V W_2),T_g(W_3)}} T_g(W_1\tens_V W_2)\tens_V T_g(W_3)\xrightarrow{\tau_{g; W_1, W_2}\tens_V 1_{T_g(W_3)}} (T_g(W_1)\tens_V T_g(W_2))\tens_V T_g(W_3).
\end{align}
From the definition of the tensor product of morphisms in $\rep V$,
\begin{equation*}
 (\tau_{g; W_1, W_2}\tens_V 1_{T_g(W_3)})  I_{T_g(W_1\tens_V W_2),T_g(W_3)}=I_{T_g(W_1)\tens_V T_g(W_2),T_g(W_3)}(\tau_{g; W_1,W_2}\tens 1_{T_g(W_3)}),
\end{equation*}
and then the definition of $\tau_{g; W_1, W_2}$ implies that \eqref{calc2} becomes
\begin{align*}
 W_1\tens(W_2\tens W_3)\xrightarrow{\cA_{W_1,W_2,W_3}} & (W_1\tens W_2)\tens W_3\xrightarrow{I_{T_g(W_1),T_g(W_2)}\tens 1_{W_3}} (T_g(W_1)\tens_V T_g(W_2))\tens T_g(W_3)\nonumber\\
 &\xrightarrow{I_{T_g(W_1)\tens_V T_g(W_2), T_g(W_3)}} (T_g(W_1)\tens_V T_g(W_2))\tens_V T_g(W_3).
\end{align*}
Next, the definition of the associativity isomorphisms in $\rep V$ implies that this composition equals
\begin{align*}
 W_1\tens (W_2\tens W_3)\xrightarrow{1_{W_1}\tens I_{T_g(W_2),T_g(W_3)}} & T_g(W_1)\tens(T_g(W_2)\tens_V T_g(W_3))\nonumber\\
 &\hspace{-5em} \xrightarrow{I_{T_g(W_1),T_g(W_2)\tens_V T_g(W_3)}} T_g(W_1)\tens_V(T_g(W_2)\tens_V T_g(W_3))\nonumber\\
 &\hspace{-5em}\xrightarrow{\cA^V_{T_g(W_1),T_g(W_2),T_g(W_3)}} (T_g(W_1)\tens_V T_g(W_2))\tens_V T_g(W_3).
\end{align*}
We replace $I_{T_g(W_2),T_g(W_3)}$ with $\tau_{g; W_2,W_3} I_{W_2,W_3}$ and use the definition of tensor product of morphisms in $\rep V$:
\begin{align*}
 W_1\tens  (W_2\tens W_3)\xrightarrow{1_{W_1}\tens I_{W_2,W_3}} & W_1\tens(W_2\tens_V W_3)\xrightarrow{I_{T_g(W_1),T_g(W_2\tens_V W_3)}} T_g(W_1)\tens_V T_g(W_2\tens_V W_3)\nonumber\\
 &\xrightarrow{1_{T_g(W_1)}\tens_V \tau_{g; W_2,W_3}} T_g(W_1)\tens_V(T_g(W_2)\tens_V T_g(W_3))\nonumber\\
 &\xrightarrow{\cA^V_{T_g(W_1),T_g(W_2),T_g(W_3)}} (T_g(W_1)\tens_V T_g(W_2))\tens_V T_g(W_3). 
\end{align*}
Finally we use the definition of $\tau_{g; W_1, W_2\tens_V W_3}$ to obtain
\begin{align*}
 W_1\tens & (W_2\tens W_3)\xrightarrow{1_{W_1}\tens I_{W_2,W_3}}   W_1\tens(W_2\tens_V W_3)\xrightarrow{I_{W_1,W_2\tens_V W_3}} W_1\tens_V(W_2\tens_V W_3)\nonumber\\
 &\xrightarrow{\tau_{g; W_1, W_2\tens_V W_3}} T_g(W_1)\tens_V T_g(W_2\tens_V W_3)\xrightarrow{1_{T_g(W_1)}\tens_V \tau_{g; W_2,W_3}}T_g(W_1)\tens_V(T_g(W_2)\tens_V T_g(W_3))\nonumber\\
 & \xrightarrow{\cA^V_{T_g(W_1),T_g(W_2),T_g(W_3)}} (T_g(W_1)\tens_V T_g(W_2))\tens_V T_g(W_3),
\end{align*}
and compatibility follows from the surjectivity of $1_{W_1}\tens I_{W_1,W_2}$ and $I_{W_1,W_2\tens_V W_3}$, and hence of their composition.

Now the even morphism $\varphi_g=g: T_g(V)\rightarrow V$ needs to be is an isomorphism in $\rep V$. In fact,
\begin{equation*}
 g\mu_{T_g(V)}=g\mu_V(g^{-1}\tens 1_V) =\mu_V(1\tens g)
\end{equation*}
because $g$ is an automorphism of $V$. The isomorphism $\varphi_g$ also needs to be compatible with $\tau_g$ and the unit isomorphisms in $\rep V$ in the sense that
\begin{equation}\label{leftcompat}
 l^V_{T_g(W)}(\varphi_g\tens_V 1_{T_g(W)}) \tau_{g; V, W} = T_g(l^V_W): T_g(V\tens_V W)\rightarrow T_g(W)
\end{equation}
and
\begin{equation}\label{rightcompat}
 r^V_{T_g(W)}(1_{T_g(W)}\tens_V \varphi_g) \tau_{g; W, V} =T_g(r^V_W): T_g(W\tens_V V)\rightarrow T_g(W)
\end{equation}
for any object $W$ in $\rep V$. Since $I_{V,W}$ and $I_{W, V}$ are surjective, it is sufficient to show that the equalities in \eqref{leftcompat} and \eqref{rightcompat} hold when both sides are precomposed with
\begin{equation*}
 I_{V,W}: V\tens W\rightarrow V\tens_V W=T_g(V\tens_V W)
\end{equation*}
and
\begin{equation*}
 I_{W,V}: W\tens V\rightarrow W\tens_V V=T_g(W\tens_V W),
\end{equation*}
respectively. 

For \eqref{leftcompat}, we get the composition
\begin{align*}
 V\tens W\xrightarrow{I_{V,W}}T_g(V\tens_V W)\xrightarrow{\tau_{g; V, W}} T_g(V)\tens_V T_g(W)\xrightarrow{g\tens_V 1_{T_g(W)}} V\tens_V T_g(W)\xrightarrow{l^V_{T_g(W)}} T_g(W).
\end{align*}
Using $\tau_{g; V,W} I_{V,W}=I_{T_g(V),T_g(W)}$ and the definition of the tensor product of morphisms in $\rep V$, this becomes
\begin{equation*}
 V\tens W\xrightarrow{g\tens 1_W} V\tens W\xrightarrow{I_{V,T_g(W)}} V\tens_V T_g(W)\xrightarrow{l^V_{T_g(W)}} T_g(W).
\end{equation*}
By the definition of the left unit isomorphism in $\rep V$, the last two arrows above can be replaced with $\mu_{T_g(W)}=\mu_W(g^{-1}\tens 1_W)$, so that in total the composition is simply $\mu_W$. But this is
\begin{equation*}
 l^V_W I_{V,W}=T_g(l^V_W) I_{V,W},
\end{equation*}
as required. Now for \eqref{rightcompat}, we have the composition
\begin{equation*}
 W\tens V\xrightarrow{I_{W,V}} W\tens_V V\xrightarrow{\tau_{g; W, V}} T_g(W)\tens_V T_g(V)\xrightarrow{1_{T_g(W)}\tens_V g} T_g(W)\tens_V V\xrightarrow{r^V_{T_g(W)}} T_g(W).
\end{equation*}
Similar to before, this composition is
\begin{equation*}
 W\tens V\xrightarrow{1_W\tens g} W\tens V\xrightarrow{I_{T_g(W),V}} T_g(W)\tens_V V\xrightarrow{r^V_{T_g(W)}} T_g(W).
\end{equation*}
By definition of $r^V_{T_g(W)}$, this equals
\begin{equation*}
 W\tens V\xrightarrow{1_W\tens g} W\tens V\xrightarrow{\cR_{V,W}^{-1}} V\tens W\xrightarrow{\mu_{T_g(W)}} T_g(W).
\end{equation*}
Since $\mu_{T_g(W)}=\mu_W(g^{-1}\tens 1_W)$, naturality of the braiding isomorphisms in $\sC$ implies that we get
\begin{equation*}
 W\tens V\xrightarrow{\cR_{V,W}^{-1}} V\tens W\xrightarrow{\mu_W} W=T_g(W).
\end{equation*}
By definition, this is $r^V_{W} I_{W,V}=T_g(r^V_W) I_{W,V}$, as desired. This completes the proof that $(T_g, \tau_g, \varphi_g)$ is a tensor endofunctor of $\rep V$, restricting to a tensor endofunctor on $\repGV$.

To finish the construction of the $G$-action on $\repGV$, we need to prove that $g\mapsto(T_g, \tau_g, \varphi_g)$ is a group homomorphism. Note first that $(T_1, \tau_1, \varphi_1)$ is the identity functor on $\rep V$ and $\repGV$, and we also need to show that $(T_{gh}, \tau_{gh}, \varphi_{gh})$ is the composition of $(T_g, \tau_g, \varphi_g)$ and $(T_h, \tau_h, \varphi_h)$ for $g,h\in G$, that is:
\begin{itemize}
 \item $T_g(T_h(W,\mu_W))=T_{gh}(W,\mu_W)$ for any object $(W,\mu_W)$ in $\rep V$, and $T_g(T_h(f))=T_{gh}(f)$ for any morphism in $\rep V$.
 
 \item $\tau_{g; T_h(W_1),T_h(W_2)} T_g(\tau_{h; W_1,W_2})=\tau_{gh; W_1,W_2}$ for all objects $W_1$ and $W_2$ in $\rep V$.
 
 \item $\varphi_g T_g(\varphi_h)=\varphi_{gh}$.
\end{itemize}
The first point is easy because
\begin{equation*}
 \mu_{T_g(T_h(W))}=\mu_{T_h(W)}(g^{-1}\tens 1_W)=\mu_{W}(h^{-1}\tens 1_W)(g^{-1}\tens 1_W)=\mu_W((gh)^{-1}\tens 1_W)=\mu_{T_{gh}(W)}
\end{equation*}
and because $T_g(T_h(f))=f=T_{gh}(f)$. Also, $\varphi_g T_g(\varphi_h)=gh=\varphi_{gh}$. Then because $\tau_{gh; W_1, W_2}$ is the unique morphism such that
\begin{equation*}
 \xymatrixcolsep{4pc}
 \xymatrix{
 W_1\tens W_2 \ar[d]^{I_{W_1,W_2}} \ar[rd]^{I_{T_{gh}(W_1),T_{gh}(W_2)}} & \\
 T_{gh}(W_1\tens_V W_2) \ar[r]^(.45){\tau_{gh; W_1, W_2}} & T_{gh}(W_1)\tens_V T_{gh}(W_2) \\
 }
\end{equation*}
commutes, the commutative diagram
\begin{equation*}
 \xymatrixcolsep{5pc}
 \xymatrix{
  & W_1\tens W_2 \ar[ld]_{I_{W_1,W_2}} \ar[d]^{I_{T_h(W_1),T_h(W_2)}} \ar[rd]^{\hspace{2em}I_{T_g(T_h(W_1)),T_g(T_h(W_2))}} & \\
  T_g(T_h(W_1\tens_V W_2))\ar[r]_(.47){T_g(\tau_{h; W_1,W_2})} & T_g(T_h(W_1)\tens_V T_h(W_2)) \ar[r]_(.48){\tau_{g; T_h(W_1),T_h(W_2)}} & T_g(T_h(W_1))\tens_V T_g(T_h(W_2))\\
}
  \end{equation*}
shows that the second point holds as well.

Having constructed the $G$-action on $\repGV$, we now construct the braiding isomorphisms. For objects $W_1$, $W_2$ in $\repGV$ with $W_1$ a $g$-twisted $V$-module for some $g\in G$, we will show that there are unique morphisms $\cR^V_{W_1,W_2}$ and $(\cR^V_{W_1,W_2})^{-1}$ such that
\begin{equation*}
 \xymatrixcolsep{4pc}
 \xymatrix{
 W_1\tens W_2 \ar[d]^{I_{W_1,W_2}} \ar[r]^{\cR_{W_1,W_2}} & W_2\tens W_1 \ar[d]^{I_{T_g(W_2), W_1}} \\
 W_1\tens_V W_2 \ar[r]^(.45){\cR^V_{W_1,W_2}} & T_g(W_2)\tens_V W_1\\
 } \hspace{2em}\mathrm{and}\hspace{2em}
 \xymatrix{
 W_2\tens W_1 \ar[d]^{I_{T_g(W_2),W_1}} \ar[r]^{\cR_{W_1,W_2}^{-1}} & W_1\tens W_2 \ar[d]^{I_{W_1,W_2}} \\
 T_g(W_2)\tens_V W_1 \ar[r]^(.55){(\cR^V_{W_1,W_2})^{-1}} & W_1\tens_V W_2\\
 }
\end{equation*}
commute. Such morphisms would be mutual inverses by the surjectivity of $I_{W_1,W_2}$ and $I_{T_g(W_2),W_1}$, so it remains to show their existence as morphisms in $\sC$ and that $\cR^V_{W_1,W_2}$ is a morphism in $\rep V$.

The existence and uniqueness of the morphisms $\cR^V_{W_1,W_2}$ and $(\cR^V_{W_1,W_2})^{-1}$ in $\sC$ will follow from the universal properties of the cokernels $(W_1\tens_V W_2, I_{W_1,W_2})$ and $(T_g(W_2)\tens_V W_1, I_{T_g(W_2), W_1})$ provided we can show:
\begin{align}
 I_{T_g(W_2), W_1}\cR_{W_1,W_2}\mu^{(1)}_{W_1,W_2}=I_{T_g(W_2),W_1}\cR_{W_1,W_2}\mu^{(2)}_{W_1,W_2}\label{Rexist}\\
 I_{W_1,W_2}\cR_{W_1,W_2}^{-1}\mu^{(1)}_{T_g(W_2), W_1}=I_{W_1,W_2}\cR_{W_1,W_2}^{-1}\mu^{(2)}_{T_g(W_2), W_1}\label{Rinvexist}
\end{align}
To verify \eqref{Rexist}, we start with $I_{T_g(W_2), W_1}\cR_{W_1,W_2}\mu^{(2)}_{W_1,W_2}$, which is the composition
\begin{align*}
 V\tens & (W_1\tens W_2)\xrightarrow{\cA_{V,W_1,W_2}} (V\tens W_1)\tens W_2\xrightarrow{\cR_{V,W_1}\tens 1_{W_2}} (W_1\tens V)\tens W_2\nonumber\\
& \xrightarrow{\cA_{W_1,V,W_2}^{-1}} W_1\tens(V\tens W_2)\xrightarrow{1_{W_1}\tens\mu_{W_2}} W_1\tens W_2\xrightarrow{\cR_{W_1,W_2}} W_2\tens W_1\xrightarrow{I_{T_g(W_2)\tens W_1}} T_g(W_2)\tens_V W_1.
\end{align*}
By the naturality of the braiding isomorphisms and the hexagon axiom in $\sC$, this equals
\begin{align*}
 & V\tens(W_1\tens W_2)\xrightarrow{\cA_{V,W_1,W_2}} (V\tens W_1)\tens W_2\xrightarrow{\cM_{V,W_1}\tens 1_{W_2}} (V\tens W_1)\tens W_2\xrightarrow{\cA_{V,W_1,W_2}^{-1}} V\tens (W_1\tens W_2)\nonumber\\
 &\xrightarrow{1_{V}\tens\cR_{W_1,W_2}} V\tens(W_2\tens W_1)\xrightarrow{\cA_{V,W_2,W_1}} (V\tens W_2)\tens W_1\xrightarrow{\mu_{W_2}\tens 1_{W_1}} W_2\tens W_1\xrightarrow{I_{T_g(W_2), W_1}}T_g(W_2)\tens_V W_1.
\end{align*}
We replace $\mu_{W_2}$ with $\mu_{T_g(W_2)}(g\tens 1_{W_2})$ and then use the intertwining operator property of $I_{T_g(W_2), W_1}$ and naturality of the associativity isomorphisms:
\begin{align*}
 & V\tens(W_1\tens W_2) \xrightarrow{\cA_{V,W_1,W_2}} (V\tens W_1)\tens W_2\xrightarrow{\cM_{V,W_1}\tens 1_{W_2}} (V\tens W_1)\tens W_2\xrightarrow{(g\tens 1_{W_1})\tens 1_{W_2}}(V\tens W_1)\tens W_2\nonumber\\
 &\xrightarrow{\cA_{V,W_1,W_2}^{-1}} V\tens(W_1\tens W_2)\xrightarrow{1_{V}\tens\cR_{W_1,W_2}} V\tens(W_2\tens W_1)\xrightarrow{\cA_{V,W_2,W_1}} (V\tens W_2)\tens W_1\nonumber\\
 &\xrightarrow{\cR_{V,W_2}\tens 1_{W_1}} (W_2\tens V)\tens W_1\xrightarrow{\cA_{W_2,V,W_1}^{-1}} W_2\tens(V\tens W_1)\xrightarrow{1_{W_2}\tens\mu_{W_1}} W_2\tens W_1\xrightarrow{I_{T_g(W_2),W_1}} T_g(W_2)\tens_V W_1.
\end{align*}
Now we apply the hexagon axiom and naturality of the braiding in $\sC$ to reduce this composition to
\begin{align*}
 V\tens(W_1\tens W_2) \xrightarrow{\cA_{V,W_1,W_2}} & (V\tens W_1)\tens W_2\xrightarrow{\cM_{V,W_1}\tens 1_{W_2}} (V\tens W_1)\tens W_2\xrightarrow{(g\tens 1_{W_1})\tens 1_{W_2}}(V\tens W_1)\tens W_2\nonumber\\
 &\xrightarrow{\mu_{W_1}\tens 1_{W_2}} W_1\tens W_2\xrightarrow{\cR_{W_1,W_2}} W_2\tens W_1\xrightarrow{I_{T_g(W_2),W_1}} T_g(W_2)\tens_V W_1.
\end{align*}
We replace $\mu_{W_1}(g\tens 1_{W_1})\cM_{V,W_1}$ with $\mu_{W_1}$ since $W_1$ is a $g$-twisted $V$-module, and the resulting composition is $I_{T_g(W_2),W_1}\cR_{W_1,W_2}\mu^{(1)}_{W_1,W_2}$, as desired.

Now to prove \eqref{Rinvexist}, we start with $I_{W_1,W_2}\cR_{W_1,W_2}^{-1}\mu^{(1)}_{T_g(W_2),W_1}$, which is the composition
\begin{align*}
 V\tens(W_2\tens W_1)\xrightarrow{\cA_{V,W_2,W_1}} & (V\tens W_2)\tens W_1\xrightarrow{(g^{-1}\tens 1_{W_2})\tens 1_{W_1}} (V\tens W_2)\tens W_1\nonumber\\
 &\xrightarrow{\mu_{W_2}\tens 1_{W_1}} W_2\tens W_1\xrightarrow{\cR_{W_1,W_2}^{-1}} W_1\tens W_2\xrightarrow{I_{W_1,W_2}} W_1\tens_V W_2.
\end{align*}
Using naturality of the braiding isomorphisms and the hexagon axiom in $\sC$, this becomes
\begin{align*}
 V\tens & (W_2\tens W_1)\xrightarrow{1_V\tens\cR_{W_1,W_2}^{-1}} V\tens(W_1\tens W_2)\xrightarrow{\cA_{V,W_1,W_2}} (V\tens W_1)\tens W_2\xrightarrow{\cR_{W_1,V}^{-1}\tens 1_{W_2}} (W_1\tens V)\tens W_2\nonumber\\
 &\xrightarrow{\cA^{-1}_{W_1,V,W_2}} W_1\tens (V\tens W_2)\xrightarrow{1_{W_1}\tens(g^{-1}\tens 1_{W_2})} W_1\tens(V\tens W_2)\xrightarrow{1_{W_1}\tens\mu_{W_2}} W_1\tens W_2\xrightarrow{I_{W_1,W_2}} W_1\tens_V W_2.
\end{align*}
Since $I_{W_1,W_2}$ is an intertwining operator,
\begin{equation*}
 I_{W_1,W_2}(1_{W_1}\tens\mu_{W_2})=I_{W_1,W_2}(\mu_{W_1}\tens 1_{W_2})(\cR_{V,W_1}^{-1}\tens 1_{W_2})\cA_{W_1,V,W_2};
\end{equation*}
this leads to the composition
\begin{align*}
 V\tens(W_2\tens W_1)\xrightarrow{1_V\tens\cR_{W_1,W_2}^{-1}} V\tens(W_1\tens W_2)\xrightarrow{\cA_{V,W_1,W_2}} (V\tens W_1)\tens W_2\xrightarrow{\cM_{V,W_1}^{-1}\tens 1_{W_2}} (V\tens W_1)\tens W_2\nonumber\\
 \xrightarrow{(g^{-1}\tens 1_{W_1})\tens 1_{W_2}} (V\tens W_1)\tens W_2\xrightarrow{\mu_{W_1}\tens 1_{W_2}} W_1\tens W_2\xrightarrow{I_{W_1,W_2}} W_1\tens_V W_2.
\end{align*}
Since $W_1$ is a $g$-twisted $V$-module, we can eliminate $(g^{-1}\tens 1_{W_1})\cM_{W_1,W_2}^{-1}$ here and then add associativity and braiding isomorphisms and their inverses to obtain:
\begin{align*}
 V\tens(W_2\tens W_1) & \xrightarrow{\cA_{V,W_2,W_1}} (V\tens W_2)\tens W_1\xrightarrow{\cR_{V,W_2}\tens 1_{W_1}} (W_2\tens V)\tens W_1\xrightarrow{\cR_{V,W_2}^{-1}\tens 1_{W_1}} (V\tens W_2)\tens W_1\nonumber\\
 &\xrightarrow{\cA_{V,W_2,W_1}^{-1}} V\tens(W_2\tens W_1)\xrightarrow{1_V\tens\cR_{W_1,W_2}^{-1}} V\tens(W_1\tens W_2)\xrightarrow{\cA_{V,W_1,W_2}}(V\tens W_1)\tens W_2\nonumber\\
 & \xrightarrow{\mu_{W_1}\tens 1_{W_2}} W_1\tens W_2\xrightarrow{I_{W_1,W_2}} W_1\tens_V W_2.
\end{align*}
By the hexagon axiom and naturality of the braiding isomorphisms, this is
\begin{align*}
 V\tens(W_2\tens W_1) & \xrightarrow{\cA_{V,W_2,W_1}} (V\tens W_2)\tens W_1\xrightarrow{\cR_{V,W_2}\tens 1_{W_1}} (W_2\tens V)\tens W_1\xrightarrow{\cA_{W_2,V,W_1}^{-1}} W_2\tens(V\tens W_1)\nonumber\\
 & \xrightarrow{1_{W_2}\tens \mu_{W_1}} W_2\tens W_1\xrightarrow{\cR_{W_1,W_2}^{-1}} W_1\tens W_2\xrightarrow{I_{W_1,W_2}} W_1\tens_V W_2,
\end{align*}
which is the right side of \eqref{Rinvexist}. We have now proved that $\cR^V_{W_1,W_2}$ exists and is an isomorphism in $\sC$.

Now we prove that $\cR^V_{W_1,W_2}$ is a morphism in $\rep V$ (and thus in $\repGV$). From the commutative diagrams
\begin{equation*}
 \xymatrixcolsep{4pc}
 \xymatrix{
 V\tens(W_1\tens W_2) \ar[d]^{1_V\tens I_{W_1,W_2}} \ar[r]^(.55){\mu^{(i)}_{W_1,W_2}} & W_1\tens W_2 \ar[d]^{I_{W_1,W_2}} \ar[r]^{\cR_{W_1,W_2}} & W_2\tens W_1 \ar[d]^{I_{T_g(W_2),W_1}} \\
 V\tens(W_1\tens_V W_2) \ar[r]^(.55){\mu_{W_1\tens_V W_2}} & W_1\tens_V W_2 \ar[r]^(.45){\cR^V_{W_1,W_2}} & T_g(W_2)\tens_V W_1 \\
 }
\end{equation*}
and
\begin{equation*}
 \xymatrixcolsep{4pc}
 \xymatrix{
 V\tens(W_1\tens W_2) \ar[d]^{1_V\tens I_{W_1,W_2}} \ar[r]^{1_V\tens\cR_{W_1,W_2}} & V\tens(W_2\tens W_1) \ar[d]^{1_V\tens I_{T_g(W_2),W_1}} \ar[r]^(.55){\mu^{(i)}_{T_g(W_2),W_1}} & W_2\tens W_1 \ar[d]^{I_{T_g(W_2),W_1}} \\
 V\tens(W_1\tens_V W_2) \ar[r]^(.45){1_V\tens\cR^V_{W_1,W_2}} & V\tens(T_g(W_2)\tens_V W_1) \ar[r]^(.55){\mu_{T_g(W_2), W_1}} & T_g(W_2)\tens_V W_1 \\
 }
\end{equation*}
for $i=1$ and $i=2$, together with the surjectivity of $1_V\tens I_{W_1,W_2}$, it is sufficient to show
\begin{equation*}
 I_{T_g(W_2),W_1}\cR_{W_1,W_2}\mu^{(1)}_{W_1,W_2}=I_{T_g(W_2),W_1}\mu^{(2)}_{T_g(W_2), W_1}(1_V\tens\cR_{W_1,W_2}).
\end{equation*}
We start with the right side of this equation, which is the composition
\begin{align*}
 V\tens(W_1\tens W_2) & \xrightarrow{1_V\tens\cR_{W_1,W_2}} V\tens(W_2\tens W_1)\xrightarrow{\cA_{V,W_2,W_1}} (V\tens W_2)\tens W_1\xrightarrow{\cR_{V,W_2}\tens 1_{W_1}} (W_2\tens V)\tens W_1\nonumber\\
 & \xrightarrow{\cA_{W_2,V,W_1}^{-1}} W_2\tens(V\tens W_1)\xrightarrow{1_{W_2}\tens\mu_{W_1}} W_2\tens W_1\xrightarrow{I_{T_g(W_2),W_1}} T_g(W_2)\tens_V W_1.
\end{align*}
By the hexagon axioms in $\sC$, this composition simplifies to
\begin{align*}
 V\tens(W_1\tens W_2)\xrightarrow{\cA_{V,W_1,W_2}}  (V\tens W_1)\tens  W_2\xrightarrow{\cR_{V\tens W_1, W_2}} & W_2\tens(V\tens W_1)\nonumber\\
 &\xrightarrow{1_{W_2}\tens\mu_{W_1}} W_2\tens W_1\xrightarrow{I_{T_g(W_2),W_1}} T_g(W_2)\tens_V W_1.
\end{align*}
Then we get $I_{T_g(W_2),W_1}\cR_{W_1,W_2}\mu^{(1)}_{W_1,W_2}$ from the naturality of the braiding isomorphisms.

Next we show that the $\cR^V_{W_1,W_2}$ define an even natural isomorphism from $\boxtimes$ to $\boxtimes\circ(T_g\times 1_{\repgV})\circ\sigma$, that is, for parity-homogeneous morphisms $f_1: W_1\rightarrow\widetilde{W}_1$ in $\repgV$ and $f_2: W_2\rightarrow\widetilde{W}_2$ in $\rep V$,
\begin{equation*}
 \cR^V_{\widetilde{W}_1,\widetilde{W}_2}(f_1\tens_V f_2) =(-1)^{\vert f_1\vert\vert f_2\vert} (T_g(f_2)\tens_V f_1)\cR_{W_1, W_2}.
\end{equation*}
First, $\cR^V_{W_1,W_2}$ is even because $\cR_{W_1,W_2}$, $I_{W_1,W_2}$, and $I_{T_g(W_2),W_1}$ are even. Then from the commutativity of
\begin{equation*}
 \xymatrixcolsep{4pc}
 \xymatrix{
 W_1\tens W_2 \ar[d]^{I_{W_1,W_2}} \ar[r]^{f_1\tens f_2} & \widetilde{W}_1\tens\widetilde{W}_2 \ar[d]^{I_{\widetilde{W}_1,\widetilde{W}_2}} \ar[r]^{\cR_{\widetilde{W}_1,\widetilde{W}_2}} & \widetilde{W}_2\tens\widetilde{W}_1 \ar[d]^{I_{T_g(\widetilde{W}_2),\widetilde{W}_1}} \\
 W_1\tens_V W_2 \ar[r]^{f_1\tens_V f_2} & \widetilde{W}_1\tens_V\widetilde{W}_2 \ar[r]^(.45){\cR^V_{\widetilde{W}_1,\widetilde{W}_2}} & T_g(\widetilde{W}_2)\tens_V\widetilde{W}_1 \\
 }
\end{equation*}
and
\begin{equation*}
 \xymatrixcolsep{4pc}
 \xymatrix{
 W_1\tens W_2 \ar[d]^{I_{W_1,W_2}} \ar[r]^{\cR_{W_1,W_2}} & W_2\tens W_1 \ar[d]^{I_{T_g(W_2), W_1}} \ar[r]^{f_2\tens f_1} & \widetilde{W}_2\tens\widetilde{W}_1 \ar[d]^{I_{T_g(\widetilde{W}_2), \widetilde{W}_1}} \\
 W_1\tens_V W_2 \ar[r]^(.45){\cR^V_{W_1,W_2}} & T_g(W_2)\tens_V W_1 \ar[r]^{T_g(f_2)\tens_V f_1} & T_g(\widetilde{W}_2)\tens_V\widetilde{W}_1
 } ,
\end{equation*}
the surjectivity of $I_{W_1,W_2}$, and the naturality of the braiding in $\sC$, we get the naturality of $\cR^V$.

To complete the proof, we need to check that the braiding $\cR^V$ is compatible with the $G$-action and satisfies the hexagon/heptagon axioms. First, for $g,h\in G$, $W_1$ a $g$-twisted $V$-module, and $W_2$ any object in $\rep V$, we need
\begin{equation*}
 \tau_{h; T_g(W_2),W_1} T_h(\cR^V_{W_1,W_2}) = \cR^V_{T_h(W_1),T_h(W_2)} \tau_{h; W_1,W_2}.
\end{equation*}
This follows from the commutative diagrams
\begin{equation*}
 \xymatrixcolsep{4pc}
 \xymatrix{
 W_1\tens W_2 \ar[d]^{I_{W_1,W_2}} \ar[r]^{\cR_{W_1,W_2}} & W_2\tens W_1 \ar[d]^{I_{T_g(W_2), W_1}} \ar[rd]^{I_{T_{hg}(W_2),T_h(W_1)}} & \\
 T_h(W_1\tens_V W_2) \ar[r]^(.47){T_h(\cR^V_{W_1,W_2})} & T_h(T_g(W_2)\tens_V W_1) \ar[r]^{\tau_{h; T_g(W_2), W_1}} & T_{hg}(W_2)\tens_V T_h(W_1)
 }
\end{equation*}
and
\begin{equation*}
 \xymatrixcolsep{4pc}
 \xymatrix{
  & W_1\tens W_2 \ar[ld]_{I_{W_1,W_2}} \ar[d]^{I_{T_h(W_1),T_h(W_2)}} \ar[r]^{\cR_{W_1,W_2}} & W_2\tens W_1 \ar[d]^{I_{T_{hg}(W_2),T_h(W_1)}} \\
  T_h(W_1\tens_V W_2) \ar[r]_(.45){\tau_{h; W_1,W_2}} & T_h(W_1)\tens_V \cF_h(W_2) \ar[r]_{\cR^V_{T_h(W_1),T_h(W_2)}} & T_{hg}(W_2)\tens_V T_h(W_1) \\
 }
\end{equation*}
as well as the surjectivity of $I_{W_1,W_2}$. In the second diagram here, the image of $\cR^V_{T_h(W_1),T_h(W_2)}$ is indeed $T_{gh}(W_2)\tens_V T_h(W_1)$: because $W_1$ is $g$-twisted, $T_h(W_1)$ is $hgh^{-1}$-twisted, and then $T_{hgh^{-1}}(T_h(W_2))=T_{hg}(W_2)$.

Now suppose $g_1, g_2\in G$, $W_1$ is a $g_1$-twisted $V$-module, $W_2$ is a $g_2$-twisted $V$-module, and $W_3$ is any object of $\rep V$. The first hexagon axiom follows from the commutative diagrams
\begin{equation*}
 \xymatrixcolsep{5.3pc}
 \xymatrix{
 W_1\tens(W_2\tens W_3) \ar[d]^{1_{W_1}\tens\cR_{W_2,W_3}} \ar[r]^{1_{W_1}\tens I_{W_2,W_3}} & W_1\tens(W_2\tens_V W_3) \ar[d]^{1_{W_1}\tens\cR^V_{W_2,W_3}} \ar[r]^{I_{W_1,W_2\tens_V W_3}} & W_1\tens_V(W_2\tens_V W_3) \ar[d]^{1_{W_1}\tens_V \cR^V_{W_2,W_3}} \\
 W_1\tens(W_3\tens W_2) \ar[d]^{\cA_{W_1,W_3,W_2}} \ar[r]^(.45){1_{W_1}\tens I_{T_{g_2}(W_3), W_2}} & W_1\tens(T_{g_2}(W_3)\tens_V W_2) \ar[r]^(.48){I_{W_1, T_{g_2}(W_3)\tens_V W_2}} & W_1\tens_V(T_{g_2}(W_3)\tens_V W_2) \ar[d]^{\cA^V_{W_1,T_{g_2}(W_3), W_2}} \\
 (W_1\tens W_3)\tens W_2 \ar[d]^{\cR_{W_1,W_3}\tens 1_{W_2}} \ar[r]^(.45){I_{W_1,T_{g_2}(W_3)}\tens 1_{W_2}} & (W_1\tens_V T_{g_2}(W_3))\tens W_2 \ar[d]^{\cR^V_{W_1,T_{g_2}(W_3)}\tens 1_{W_2}} \ar[r]^(.48){I_{W_1\tens_V T_{g_2}(W_3), W_2}} & (W_1\tens_V T_{g_2}(W_3))\tens_V W_2 \ar[d]^{\cR^V_{W_1,T_{g_2}(W_3)}\tens_V 1_{W_2}} \\
 (W_3\tens W_1)\tens W_2 \ar[r]^(.45){I_{T_{g_1 g_2}(W_3), W_1}\tens 1_{W_2}} & (T_{g_1 g_2}(W_3)\tens_V W_1)\tens W_2 \ar[r]^{I_{T_{g_1 g_2}(W_3)\tens_V W_1, W_2}} & (T_{g_1 g_2}(W_3)\tens_V W_1)\tens_V W_2 \\
 }
\end{equation*}
and
\begin{equation*}
 \xymatrixcolsep{5.3pc}
 \xymatrix{
 W_1\tens(W_2\tens W_3) \ar[d]^{\cA_{W_1,W_2,W_3}} \ar[r]^{1_{W_1}\tens I_{W_2,W_3}} & W_1\tens(W_2\tens_V W_3) \ar[r]^{I_{W_1,W_2\tens_V W_3}} & W_1\tens_V(W_2\tens_V W_3) \ar[d]^{\cA^V_{W_1,W_2,W_3}} \\
 (W_1\tens W_2)\tens W_3 \ar[d]^{\cR_{W_1\tens W_2, W_3}} \ar[r]^(.48){I_{W_1,W_2}\tens 1_{W_3}} & (W_1\tens_V W_2)\tens W_3 \ar[d]^{\cR_{W_1\tens_V W_2, W_3}} \ar[r]^(.47){I_{W_1\tens_V W_2, W_3}} & (W_1\tens_V W_2)\tens_V W_3 \ar[d]^{\cR^V_{W_1\tens_V W_2, W_3}} \\
 W_3\tens(W_1\tens W_2) \ar[d]^{\cA_{W_3,W_1,W_2}} \ar[r]^(.47){1_{W_3}\tens I_{W_1,W_2}} & W_3\tens(W_1\tens_V W_2) \ar[r]^(.45){I_{T_{g_1 g_2}(W_3), W_1\tens_V W_2}} & T_{g_1 g_2}(W_3)\tens_V(W_1\tens_V W_2) \ar[d]^{\cA^V_{T_{g_1 g_2}(W_3), W_1,W_2}} \\
 (W_3\tens W_1)\tens W_2 \ar[r]^(.45){I_{T_{g_1 g_2}(W_3), W_1}\tens 1_{W_2}} & (T_{g_1 g_2}(W_3)\tens_V W_1)\tens W_2 \ar[r]^{I_{T_{g_1 g_2}(W_3)\tens_V W_1, W_2}} & (T_{g_1 g_2}(W_3)\tens_V W_1)\tens_V W_2 \\
 } ,
\end{equation*}
the surjectivity of $I_{W_1, W_2\tens_V W_3}(1_{W_1}\tens I_{W_2,W_3})$, and the hexagon axioms in $\sC$. For the heptagon, take $g\in G$, a $g$-twisted $V$-module $W_1$, and any objects $W_2$, $W_3$ in $\rep V$. Then the commutative diagrams
\begin{equation*}
\xymatrixcolsep{5.4pc}
\xymatrix{
 (W_1\tens W_2)\tens W_3 \ar[d]^{\cA_{W_1,W_2,W_3}^{-1}} \ar[r]^{I_{W_1,W_2}\tens 1_{W_3}} & (W_1\tens_V W_2)\tens W_3 \ar[r]^{I_{W_1\tens_V W_2, W_3}} & (W_1\tens_V W_2)\tens_V W_3 \ar[d]^{(\cA^V_{W_1,W_2,W_3})^{-1}} \\
 W_1\tens(W_2\tens W_3) \ar[d]^{\cR_{W_1,W_2\tens W_3}} \ar[r]^{1_{W_1}\tens I_{W_2,W_3}} & W_1\tens(W_2\tens_V W_3) \ar[d]^{\cR_{W_1,W_2\tens_V W_3}} \ar[r]^{I_{W_1,W_2\tens_V W_3}} & W_1\tens_V(W_2\tens_V W_3) \ar[d]^{\cR^V_{W_1,W_2\tens_V W_3}} \\
 (W_2\tens W_3)\tens W_1 \ar[d]^{\cA_{W_2,W_3,W_1}^{-1}} \ar[rd]^(.6){I_{T_g(W_2),T_g(W_3)}\tens 1_{W_1}} \ar[r]^{I_{W_2,W_3}\tens 1_{W_1}} & (W_2\tens_V W_3)\tens W_1 \ar[d]^{\tau_{g; W_2,W_3}\tens 1_{W_1}} \ar[r]^{I_{T_g(W_2\tens_V W_3), W_1}} & T_g(W_2\tens_V W_3)\tens_V W_1 \ar[d]^{\tau_{g; W_2,W_3}\tens_V 1_{W_1}} \\
 W_2\tens(W_3\tens W_1) \ar[rd]^(.6){1_{W_2}\tens I_{T_g(W_3), W_1}} & (T_g(W_2)\tens_VT_g(W_3))\tens W_1 \ar[r]^(.49){I_{T_g(W_2)\tens_V T_g(W_3), W_1}} & (T_g(W_2)\tens_V T_g(W_3))\tens_V W_1 \ar[d]_{(\cA^V_{T_g(W_2),T_g(W_3), W_1})^{-1}} \\
  & W_2\tens(T_g(W_3)\tens_V W_1) \ar[r]_(.47){I_{T_g(W_2),T_g(W_3)\tens_V W_1}} & T_g(W_2)\tens_V(T_g(W_3)\tens_V W_1) \\
 }
\end{equation*}
and
\begin{equation*}
 \xymatrixcolsep{5.4pc}
 \xymatrix{
 (W_1\tens W_2)\tens W_3 \ar[d]^{\cR_{W_1,W_2}\tens 1_{W_3}} \ar[r]^{I_{W_1,W_2}\tens 1_{W_3}} & (W_1\tens_V W_2)\tens W_3 \ar[d]^{\cR^V_{W_1,W_2}\tens 1_{W_3}} \ar[r]^{I_{W_1\tens_V W_2, W_3}} & (W_1\tens_V W_2)\tens_V W_3 \ar[d]^{\cR^V_{W_1,W_2}\tens_V 1_{W_3}} \\
 (W_2\tens W_1)\tens W_3 \ar[d]^{\cA_{W_2,W_1,W_3}^{-1}} \ar[r]^(.47){I_{T_g(W_2), W_1}\tens 1_{W_3}} & (T_g(W_2)\tens_V W_1)\tens W_3 \ar[r]^(.49){I_{T_g(W_2)\tens_V W_1, W_3}} & (T_g(W_2)\tens_V W_1)\tens_V W_3 \ar[d]^{(\cA^V_{T_g(W_2), W_1, W_3})^{-1}} \\
 W_2\tens(W_1\tens W_3) \ar[d]^{1_{W_2}\tens\cR_{W_1,W_3}} \ar[r]^{1_{W_2}\tens I_{W_1,W_3}} & W_2\tens(W_1\tens_V W_3) \ar[d]^{1_{W_2}\tens\cR^V_{W_1,W_3}} \ar[r]^(.47){I_{T_g(W_2), W_1\tens_V W_3}} & T_g(W_2)\tens_V (W_1\tens_V W_3) \ar[d]^{1_{W_2}\tens_V\cR^V_{W_1,W_3}} \\
 W_2\tens(W_3\tens W_1) \ar[r]^(.47){1_{W_2}\tens I_{T_g(W_3), W_1}} & W_2\tens(T_g(W_3)\tens_V W_1) \ar[r]^(.48){I_{T_g(W_2),T_g(W_3)\tens_V W_1}} & T_g(W_2)\tens_V(T_g(W_3)\tens_V W_1)\\
 } ,
\end{equation*}
the surjectivity of $I_{W_1\tens_V W_2, W_3}(I_{W_1,W_2}\tens 1_{W_3})$, and the hexagon axiom in $\sC$ complete the proof of the theorem.

\section{Details for Theorem \ref{thm:repV=repGV}}\label{app:mthm_details}

Here we provide detailed calculations for the proofs of Section \ref{sec:MainCatThm}, incorporating all unit and associativity isomorphisms and making heavy use of the triangle, pentagon, and hexagon axioms.

\bigskip

\noindent\textbf{Equations \eqref{eqn:rigidlike_lemma_1} and \eqref{eqn:rigidlike_lemma_2}.} We consider $e_{V\tens V}(1_{V\tens V}\tens F_L)$, which is given by the composition
\begin{align*}
 (V\tens & V)  \tens V\xrightarrow{1_{V\tens V}\tens l_V^{-1}} (V\tens V)\tens(\vac\tens V)\xrightarrow{1_{V\tens V}\tens(\widetilde{i}_V\tens 1_V)} (V\tens V)\tens((V\tens V)\tens V)\nonumber\\
 &\xrightarrow{1_{V\tens V}\tens\cA_{V,V,V}^{-1}} (V\tens V)\tens(V\tens(V\tens V))\xrightarrow{1_{V\tens V}\tens(1_V\tens\mu_V)} (V\tens V)\tens(V\tens V)\xrightarrow{\cA^{-1}_{V,V,V\tens V}} V\tens (V\tens (V\tens V))\nonumber\\
 & \xrightarrow{1_V\tens\cA_{V,V,V}} V\tens((V\tens V)\tens V)\xrightarrow{1_V\tens(\varepsilon_V\mu_V\tens 1_V)} V\tens(\vac\tens V)\xrightarrow{1_V\tens l_V} V\tens V\xrightarrow{\varepsilon_V\mu_V} \vac.
\end{align*}
We move the second associativity isomorphism to the front using its naturality and we move the first $\mu_V$ back using naturality of the associativity and left unit isomorphisms:
\begin{align*}
 (V\tens V) & \tens V\xrightarrow{\cA_{V,V,V}^{-1}} V\tens(V\tens V)\xrightarrow{1_V\tens(1_V\tens l_V^{-1})} V\tens(V\tens(\vac\tens V))\xrightarrow{1_V\tens(1_V\tens(\widetilde{i}_V\tens 1_V))} V\tens(V\tens((V\tens V)\tens V))\nonumber\\
 & \xrightarrow{1_V\tens(1_V\tens\cA_{V,V,V}^{-1})} V\tens(V\tens(V\tens(V\tens V))) \xrightarrow{1_V\tens\cA_{V,V,V\tens V}} V\tens((V\tens V)\tens(V\tens V))\nonumber\\ 
 & \xrightarrow{1_V\tens(\varepsilon_V\mu_V\tens 1_{V\tens V})} V\tens(\vac\tens(V\tens V)) \xrightarrow{1_V\tens l_{V\tens V}} V\tens(V\tens V)\xrightarrow{1_V\tens\mu_V} V\tens V\xrightarrow{\varepsilon_V\mu_V} \vac.
\end{align*}
We rewrite using the triangle axiom and naturality of the associativity isomorphisms:
\begin{align*}
 (V\tens V) & \tens V\xrightarrow{\cA_{V,V,V}^{-1}} V\tens(V\tens V)\xrightarrow{1_V\tens(r_V^{-1}\tens 1_V)} V\tens((V\tens\vac)\tens V) \xrightarrow{1_V\tens((1_V\tens\widetilde{i}_V)\tens 1_V)} V\tens((V\tens(V\tens V))\tens V)\nonumber\\
 & \xrightarrow{1_V\tens\cA_{V,V\tens V,V}^{-1}} V\tens(V\tens((V\tens V)\tens V)) \xrightarrow{1_V\tens(1_V\tens\cA_{V,V,V}^{-1})} V\tens(V\tens(V\tens(V\tens V)))\nonumber\\
 & \xrightarrow{1_V\tens\cA_{V,V,V\tens V}} V\tens((V\tens V)\tens(V\tens V))\xrightarrow{1_V\tens\cA_{V\tens V,V,V}} V\tens(((V\tens V)\tens V)\tens V)\nonumber\\
& \xrightarrow{1_V\tens((\varepsilon_V\mu_V\tens 1_V)\tens 1_V)} V\tens((\vac\tens V)\tens V)\xrightarrow{1_V\tens(l_V\tens 1_V)} V\tens(V\tens V)\xrightarrow{1_V\tens\mu_V} V\tens V\xrightarrow{\varepsilon_V\mu_V} \vac.
\end{align*}
Now we replace the associativity isomorphisms in the second and third lines with $1_V\tens(\cA_{V,V,V}\tens 1_V)$ using the pentagon axiom, and then by rigidity of $V$, the whole composition collapses to
\begin{equation}\label{eqn:two_mult}
 (V\tens V)\tens V\xrightarrow{\cA^{-1}_{V,V,V}} V\tens(V\tens V)\xrightarrow{1_V\tens\mu_V} V\tens V\xrightarrow{\varepsilon_V\mu_V} \vac
\end{equation}
as required.

On the other hand, $e_{V\tens V}(1_{V\tens V}\tens F_R)$ is the composition
\begin{align*}
 (V\tens & V)  \tens V\xrightarrow{1_{V\tens V}\tens r_V^{-1}} (V\tens V)\tens(V\tens\vac)\xrightarrow{1_{V\tens V}\tens(1_V\tens\widetilde{i}_V)} (V\tens V)\tens(V\tens(V\tens V))\nonumber\\
& \xrightarrow{1_{V\tens V}\tens\cA_{V,V,V}} (V\tens V)\tens((V\tens V)\tens V)\xrightarrow{1_{V\tens V}\tens(\mu_V\tens 1_V)} (V\tens V)\tens(V\tens V)\xrightarrow{\cA_{V,V,V\tens V}^{-1}} V\tens (V\tens (V\tens V))\nonumber\\
& \xrightarrow{1_V\tens\cA_{V,V,V}} V\tens((V\tens V)\tens V)\xrightarrow{1_V\tens(\varepsilon_V\mu_V\tens 1_V)} V\tens(\vac\tens V)\xrightarrow{1_V\tens l_V} V\tens V\xrightarrow{\varepsilon_V\mu_V}\vac.
\end{align*}
As before, we move $\cA_{V,V,V\tens V}^{-1}$ forward and the first $\mu_V$ back; we also apply the associativity of $\mu_V$:
\begin{align*}
 (V\tens V) & \tens V\xrightarrow{\cA_{V,V,V}^{-1}} V\tens(V\tens V)\xrightarrow{1_V\tens(1_V\tens r_V^{-1})} V\tens(V\tens(V\tens\vac)) \xrightarrow{1_V\tens(1_V\tens(1_V\tens\widetilde{i}_V))} V\tens(V\tens(V\tens(V\tens V)))\nonumber\\
 & \xrightarrow{1_V\tens(1_V\tens\cA_{V,V,V})} V\tens(V\tens((V\tens V)\tens V))\xrightarrow{1_V\tens\cA_{V,V\tens V,V}} V\tens((V\tens(V\tens V))\tens V)\nonumber\\
 & \xrightarrow{1_V\tens(\cA_{V,V,V}\tens 1_V)} V\tens(((V\tens V)\tens V)\tens V) \xrightarrow{1_V\tens((\mu_V\tens 1_V)\tens 1_V)} V\tens((V\tens V)\tens V)\nonumber\\
 & \xrightarrow{1_V\tens(\varepsilon_V\mu_V\tens 1_V)} V\tens(\vac\tens V) \xrightarrow{1_V\tens l_V} V\tens V\xrightarrow{\varepsilon_V\mu_V} \vac.
\end{align*}
Now we rewrite the associativity isomorphisms in the second and third rows as $\cA_{V\tens V,V,V}\cA_{V,V,V\tens V}$ using the pentagon axiom, and then we apply the naturality of these isomorphisms:
\begin{align*}
 (V & \tens V)  \tens V\xrightarrow{\cA_{V,V,V}^{-1}} V\tens(V\tens V)\xrightarrow{1_V\tens(1_V\tens r_V^{-1})} V\tens(V\tens(V\tens\vac)) \xrightarrow{1_V\tens\cA_{V,V\vac}} V\tens((V\tens V)\tens\vac)\nonumber\\
 & \xrightarrow{1_V\tens(1_{V\tens V}\tens\widetilde{i}_V)} V\tens((V\tens V)\tens(V\tens V)) \xrightarrow{1_V\tens(\mu_V\tens 1_{V\tens V})} V\tens(V\tens(V\tens V)) \xrightarrow{1_V\tens\cA_{V,V,V}} V\tens((V\tens V)\tens V)\nonumber\\
 & \xrightarrow{1_V\tens(\varepsilon_V\mu_V\tens 1_V)} V\tens(\vac\tens V)\xrightarrow{1_V\tens l_V} V\tens V\xrightarrow{\varepsilon_V\mu_V} \vac.
\end{align*}
Next we use the identity $\cA_{V,V,\vac}(1_V\tens r_V^{-1}) =r_{V\tens V}^{-1}$ and the naturality of the right unit isomorphisms:
\begin{align*}
 (V  \tens V)  \tens V\xrightarrow{\cA_{V,V,V}^{-1}} V\tens & (V\tens V)\xrightarrow{1_V\tens\mu_V} V\tens V\xrightarrow{1_V\tens r_V^{-1}} V\tens(V\tens\vac)\xrightarrow{1_V\tens(1_V\tens\widetilde{i}_V)} V\tens(V\tens(V\tens V))\nonumber\\
 & \xrightarrow{1_V\tens\cA_{V,V,V}} V\tens((V\tens V)\tens V)\xrightarrow{1_V\tens(\varepsilon_V\mu_V\tens 1_V)} V\tens(\vac\tens V)\xrightarrow{1_V\tens l_V} V\tens V\xrightarrow{\varepsilon_V\mu_V} \vac.
\end{align*}
Finally, this composition collapses to \eqref{eqn:two_mult} by the rigidity of $V$.

\bigskip

\noindent\textbf{Equation \eqref{eqn:Tr_g_zero}.} By the left unit property of $V$, $(\mathrm{Tr}_\cC\,g)1_V$ is the composition
\begin{equation*}
 V\xrightarrow{l_V^{-1}}\vac\tens V\xrightarrow{\widetilde{i}_V\tens 1_V} (V\tens V)\tens V\xrightarrow{(1_V\tens g)\tens 1_V} (V\tens V)\tens V\xrightarrow{\mu_V\tens 1_V} V\tens V\xrightarrow{\mu_V} V.
\end{equation*}
Because $g$ is an automorphism of $V$, this agrees with
\begin{equation*}
 V\xrightarrow{l_V^{-1}} \vac\tens V\xrightarrow{\widetilde{i}_V\tens 1_V} (V\tens V)\tens V\xrightarrow{(g^{-1}\tens 1_V)\tens g^{-1}} (V\tens V)\tens V\xrightarrow{\mu_V\tens 1_V} V\tens V\xrightarrow{\mu_V} V\xrightarrow{g} V.
\end{equation*}
We then use associativity of $\mu_V$ and naturality of associativity and unit isomorphisms to rewrite as
\begin{equation*}
 V\xrightarrow{g^{-1}} V\xrightarrow{l_V^{-1}} \vac\tens V\xrightarrow{\widetilde{i}_V\tens 1_V} (V\tens V)\tens V\xrightarrow{\cA_{V,V,V}^{-1}} V\tens(V\tens V)\xrightarrow{1_V\tens\mu_V} V\tens V\xrightarrow{g^{-1}\tens 1_V} V\tens V\xrightarrow{\mu_V} V\xrightarrow{g} V.
\end{equation*}
Next we use Lemma \ref{rigidlike_lemma} and the automorphism property of $g$ to obtain
\begin{equation*}
 V\xrightarrow{g^{-1}} V\xrightarrow{r_V^{-1}} V\tens\vac\xrightarrow{1_V\tens\widetilde{i}_V} V\tens(V\tens V)\xrightarrow{\cA_{V,V,V}} (V\tens V)\tens V\xrightarrow{\mu_V\tens 1_V} V\tens V\xrightarrow{1_V\tens g} V\tens V\xrightarrow{\mu_V} V.
\end{equation*}
Naturality of the associativity isomorphisms and one more application of the associativity of $\mu_V$ then yields
\begin{equation*}
 V\xrightarrow{g^{-1}} V\xrightarrow{r_V^{-1}} V\tens\vac\xrightarrow{1_V\tens\widetilde{i}_V} V\tens(V\tens V)\xrightarrow{1_V\tens(1_V\tens g)} V\tens(V\tens V)\xrightarrow{1_V\tens\mu_V} V\tens V\xrightarrow{\mu_V} V,
\end{equation*}
which is $(\mathrm{Tr}_\cC\,g)g^{-1}$ by the right unit property of $V$.

\bigskip

\noindent\textbf{Equation \eqref{eqn:Pi_g_Vhom}.} We start with $\mu_W(1_V\tens\Pi_g)$, which is the composition
 \begin{align*}
  V\tens W & \xrightarrow{1_V\tens l_W^{-1}} V\tens(\vac\tens W)\xrightarrow{1_V\tens(\widetilde{i}_V\tens 1_W)} V\tens((V\tens V)\tens W)\xrightarrow{1_V\tens\cA_{V,V,W}^{-1}} V\tens(V\tens(V\tens W))\nonumber\\
  &\xrightarrow{1_V\tens(1_V\tens[(g\tens 1_W)\cM_{V,W}])} V\tens(V\tens(V\tens W))\xrightarrow{1_V\tens(1_V\tens\mu_W)} V\tens(V\tens W)\xrightarrow{1_V\tens\mu_W} V\tens W\xrightarrow{\mu_W} W.
 \end{align*}
 We replace the first arrow using the triangle axiom and the triviality of $\cR_{\vac,V}$ and rewrite the last two arrows using associativity of $\mu_V$ and $\mu_W$:
 \begin{align*}
  V & \tens W  \xrightarrow{l_V^{-1}\tens 1_W} (\vac\tens V)\tens W\xrightarrow{\cR_{\vac,V}\tens 1_W} (V\tens\vac)\tens W\xrightarrow{\cA_{V,\vac,W}^{-1}} V\tens(\vac\tens W)\xrightarrow{1_V\tens(\widetilde{i}_V\tens 1_W)} V\tens((V\tens V)\tens W)\nonumber\\
  &\xrightarrow{1_V\tens\cA_{V,V,W}^{-1}} V\tens(V\tens(V\tens W))
  \xrightarrow{1_V\tens(1_V\tens[(g\tens 1_W)\cM_{V,W}])} V\tens(V\tens(V\tens W))\xrightarrow{1_V\tens\cA_{V,V,W}} V\tens((V\tens V)\tens W)\nonumber\\
  &\xrightarrow{1_V\tens(\mu_V\tens 1_W)} V\tens(V\tens W)\xrightarrow{1_V\tens\mu_W} V\tens W\xrightarrow{\mu_W} W.
 \end{align*}
Now we write $l_V^{-1}\tens 1_W=\cA_{\vac,V,W} l_{V\tens W}^{-1}$ and apply naturality of associativity and braiding isomorphisms to $\widetilde{i}_V$; meanwhile we rewrite the last three arrows using associativity again and naturality of the associativity isomorphisms:
\begin{align*}
 V & \tens W  \xrightarrow{l_{V\tens W}^{-1}} \vac\tens(V\tens W)\xrightarrow{\widetilde{i}_V\tens 1_{V\tens W}} (V\tens V)\tens(V\tens W)\xrightarrow{\cA_{V\tens V,V,W}} ((V\tens V)\tens V)\tens W\nonumber\\
 & \xrightarrow{\cR_{V\tens V,V}\tens 1_W} (V\tens(V\tens V))\tens W\xrightarrow{\cA_{V,V\tens V,W}^{-1}} V\tens((V\tens V)\tens W)\xrightarrow{1_V\tens\cA_{V,V,W}^{-1}} V\tens(V\tens(V\tens W))\nonumber\\
  & \xrightarrow{1_V\tens(1_V\tens[(g\tens 1_W)\cM_{V,W}])} V\tens(V\tens(V\tens W))\xrightarrow{1_V\tens\cA_{V,V,W}} V\tens((V\tens V)\tens W) \xrightarrow{\cA_{V,V\tens V,W}} (V\tens(V\tens V))\tens W\nonumber\\
  & \xrightarrow{(1_V\tens\mu_V)\tens 1_W} (V\tens V)\tens W\xrightarrow{\mu_V\tens 1_W} V\tens W\xrightarrow{\mu_W} W.
\end{align*}
Next we apply the hexagon and pentagon axioms to the arrows in the second line above; we also apply the associativity $\mu_V$ and the pentagon axiom towards the end of the composition:
\begin{align*}
 V & \tens W  \xrightarrow{l_{V\tens W}^{-1}} \vac\tens(V\tens W)\xrightarrow{\widetilde{i}_V\tens 1_{V\tens W}} (V\tens V)\tens(V\tens W)\xrightarrow{\cA_{V\tens V,V,W}} ((V\tens V)\tens V)\tens W\nonumber\\
  & \xrightarrow{\cA_{V,V,V}^{-1}\tens 1_W} (V\tens(V\tens V))\tens W\xrightarrow{(1_V\tens\cR_{V,V})\tens 1_W} (V\tens(V\tens V))\tens W\xrightarrow{\cA_{V,V,V}\tens 1_W} ((V\tens V)\tens V)\tens W\nonumber\\
  & \xrightarrow{(\cR_{V,V}\tens 1_V)\tens 1_W} ((V\tens V)\tens V)\tens W\xrightarrow{\cA_{V\tens V,V,W}^{-1}} (V\tens V)\tens(V\tens W)\xrightarrow{\cA_{V,V,V\tens W}^{-1}} V\tens(V\tens(V\tens W))\nonumber\\
   &\xrightarrow{1_V\tens(1_V\tens[(g\tens 1_W)\cM_{V,W}])} V\tens(V\tens(V\tens W))\xrightarrow{\cA_{V,V,V\tens W}} (V\tens V)\tens(V\tens W)\xrightarrow{\cA_{V\tens V,V,W}} ((V\tens V)\tens V)\tens W\nonumber\\
   & \xrightarrow{(\mu_V\tens 1_V)\tens 1_W} (V\tens V)\tens W\xrightarrow{\mu_V\tens 1_W} V\tens W\xrightarrow{\mu_W} W.
\end{align*}
We use naturality of the associativity isomorphisms to cancel $\cA_{V,V,V\tens W}$ and its inverse here. With this done, we move the second $\cR_{V,V}$ using naturality of the associativity isomorphisms, in order to cancel it against the first $\mu_V$ using commutativity of $\mu_V$. Then we begin rewriting the fifth line using associativity again:
\begin{align*}
 V & \tens W  \xrightarrow{l_{V\tens W}^{-1}} \vac\tens(V\tens W)\xrightarrow{\widetilde{i}_V\tens 1_{V\tens W}} (V\tens V)\tens(V\tens W)\xrightarrow{\cA_{V\tens V,V,W}} ((V\tens V)\tens V)\tens W\nonumber\\
  & \xrightarrow{\cA_{V,V,V}^{-1}\tens 1_W} (V\tens(V\tens V))\tens W\xrightarrow{(1_V\tens\cR_{V,V})\tens 1_W} (V\tens(V\tens V))\tens W\xrightarrow{\cA_{V,V,V}\tens 1_W} ((V\tens V)\tens V)\tens W\nonumber\\
  &\xrightarrow{\cA_{V\tens V,V,W}^{-1}} (V\tens V)\tens(V\tens W)\xrightarrow{1_{V\tens V}\tens[(g\tens 1_W)\cM_{V,W}]} (V\tens V)\tens(V\tens W)\xrightarrow{\cA_{V\tens V,V,W}} ((V\tens V)\tens V)\tens W\nonumber\\
  &\xrightarrow{\cA_{V,V,V}^{-1}\tens 1_W} (V\tens(V\tens V))\tens W\xrightarrow{(1_V\tens\mu_V)\tens 1_W} (V\tens V)\tens W\xrightarrow{\mu_V\tens 1_W} V\tens W\xrightarrow{\mu_W} W.
\end{align*}
We now rewrite the last five arrows using associativity of $\mu_W$, naturality of the associativity isomorphisms, and the pentagon axiom. Then we use commutativity to insert an $\cR_{V,V}$ in front of $\mu_V$:
\begin{align*}
 V & \tens W  \xrightarrow{l_{V\tens W}^{-1}} \vac\tens(V\tens W)\xrightarrow{\widetilde{i}_V\tens 1_{V\tens W}} (V\tens V)\tens(V\tens W)\xrightarrow{\cA_{V\tens V,V,W}} ((V\tens V)\tens V)\tens W\nonumber\\
  & \xrightarrow{\cA_{V,V,V}^{-1}\tens 1_W} (V\tens(V\tens V))\tens W\xrightarrow{(1_V\tens\cR_{V,V})\tens 1_W} (V\tens(V\tens V))\tens W\xrightarrow{\cA_{V,V,V}\tens 1_W} ((V\tens V)\tens V)\tens W\nonumber\\
  &\xrightarrow{\cA_{V\tens V,V,W}^{-1}} (V\tens V)\tens(V\tens W)\xrightarrow{1_{V\tens V}\tens[(g\tens 1_W)\cM_{V,W}]} (V\tens V)\tens(V\tens W)\xrightarrow{\cA_{V,V,V\tens W}^{-1}} V\tens(V\tens(V\tens W))\nonumber\\
  & \xrightarrow{1_V\tens\cA_{V,V,W}} V\tens((V\tens V)\tens W)\xrightarrow{1_V\tens(\cR_{V,V}\tens 1_W)} V\tens((V\tens V)\tens W)\xrightarrow{1_V\tens(\mu_V\tens 1_W)} V\tens(V\tens W)\nonumber\\
  &\xrightarrow{1_V\tens\mu_W} V\tens W\xrightarrow{\mu_W} W.
\end{align*}
Next we use naturality of the associativity isomorphisms to move $\cA_{V,V,V\tens W}^{-1}$, the pentagon axiom in the second and third rows, the associativity of $\mu_W$, and naturality of the associativity and braiding isomorphisms to move $g$:
\begin{align*}
 V & \tens W  \xrightarrow{l_{V\tens W}^{-1}} \vac\tens(V\tens W)\xrightarrow{\widetilde{i}_V\tens 1_{V\tens W}} (V\tens V)\tens(V\tens W)\xrightarrow{\cA_{V\tens V,V,W}} ((V\tens V)\tens V)\tens W\nonumber\\
  & \xrightarrow{\cA_{V,V,V}^{-1}\tens 1_W} (V\tens(V\tens V))\tens W\xrightarrow{(1_V\tens\cR_{V,V})\tens 1_W} (V\tens(V\tens V))\tens W\xrightarrow{\cA_{V,V\tens V,W}^{-1}} V\tens((V\tens V)\tens W)\nonumber\\
  &\xrightarrow{1_V\tens\cA_{V,V,W}^{-1}} V\tens(V\tens(V\tens W))\xrightarrow{1_V\tens(1_V\tens\cM_{V,W})} V\tens(V\tens(V\tens W))\xrightarrow{1_V\tens\cA_{V,V,W}} V\tens((V\tens V)\tens W)\nonumber\\
  &\xrightarrow{1_V\tens(\cR_{V,V}\tens 1_W)} V\tens((V\tens V)\tens W)\xrightarrow{1_V\tens\cA_{V,V,W}^{-1}} V\tens(V\tens(V\tens W))\xrightarrow{1_V\tens(1_V\tens\mu_W)} V\tens(V\tens W)\nonumber\\
  & \xrightarrow{1_V\tens(g\tens 1_W)} V\tens(V\tens W)\xrightarrow{1_V\tens\mu_W} V\tens W\xrightarrow{\mu_W} W.
\end{align*}
Now we use naturality of the associativity isomorphisms to move the first $\cR_{V,V}$, and then we use the pentagon axiom to rewrite the first three associativity isomorphisms and the hexagon axiom to rewrite all braiding isomorphisms:
\begin{align*}
 V & \tens W  \xrightarrow{l_{V\tens W}^{-1}} \vac\tens(V\tens W)\xrightarrow{\widetilde{i}_V\tens 1_{V\tens W}} (V\tens V)\tens(V\tens W)\xrightarrow{\cA_{V,V,V\tens W}^{-1}} V\tens(V\tens(V\tens W))\nonumber\\
 &\xrightarrow{1_V\tens\cA_{V,V,W}} V\tens((V\tens V)\tens W)\xrightarrow{1_V\tens\cA_{V,V,W}^{-1}} V\tens(V\tens(V\tens W))\xrightarrow{1_V\tens\cR_{V,V\tens W}} V\tens((V\tens W)\tens V)\nonumber\\
 &\xrightarrow{1_V\tens\cA_{V,W,V}^{-1}} V\tens(V\tens(W\tens V))\xrightarrow{1_V\tens\cA_{V,W,V}} V\tens((V\tens W)\tens V)\xrightarrow{1_V\tens\cR_{V\tens W,V}} V\tens(V\tens(V\tens W))\nonumber\\
 &\xrightarrow{1_V\tens\cA_{V,V,W}} V\tens((V\tens V)\tens W)\xrightarrow{1_V\tens\cA_{V,V,W}^{-1}} V\tens(V\tens(V\tens W))\xrightarrow{1_V\tens(1_V\tens\mu_W)} V\tens(V\tens W)\nonumber\\
  & \xrightarrow{1_V\tens(g\tens 1_W)} V\tens(V\tens W)\xrightarrow{1_V\tens\mu_W} V\tens W\xrightarrow{\mu_W} W.
\end{align*}
We cancel all pairs of associativity isomorphisms and their inverses and then apply naturality of the associativity, braiding, and unit isomorphisms to the first $\mu_W$ to finally obtain
\begin{align*}
 V\tens W\xrightarrow{\mu_W} W\xrightarrow{l_W^{-1}}\vac\tens W & \xrightarrow{\widetilde{i}_V\tens 1_W}  (V\tens V)\tens W  \xrightarrow{\cA_{V,V,W}^{-1}} V\tens  (V\tens W)\nonumber\\
 & \xrightarrow{1_V\tens\cM_{V,W}} V\tens(V\tens W)\xrightarrow{1_V\tens(g\tens 1_W)} V\tens(V\tens W)\xrightarrow{1_V\tens\mu_W} V\tens W\xrightarrow{\mu_W} W,
\end{align*}
which is $\Pi_g\mu_W$.

\bigskip

\noindent\textbf{Equations \eqref{eqn:Im_of_Pi_g_twisted_1} through \eqref{eqn:Im_of_Pi_g_twisted_3}.} The morphism $\mu_W(g\tens\Pi_g)\cM_{V,W}$ is the composition 
\begin{align}\label{PiGtwisted}
 V\tens W & \xrightarrow{\cM_{V,W}} V\tens W\xrightarrow{1_V\tens l_W^{-1}} V\tens(\vac\tens W)\xrightarrow{1_V\tens(\widetilde{i}_V\tens 1_W)} V\tens((V\tens V)\tens W)\xrightarrow{1_V\tens\cA_{V,V,W}^{-1}} V\tens(V\tens(V\tens W))\nonumber\\
 & \xrightarrow{g\tens(1_V\tens[(g\tens 1_W)\cM_{V,W}])} V\tens(V\tens(V\tens W))\xrightarrow{1_V\tens(1_V\tens\mu_W)} V\tens(V\tens W)\xrightarrow{1_V\tens\mu_W} V\tens W\xrightarrow{\mu_W} W.
\end{align}
We begin by rewriting the second, third, and fourth arrows:
\begin{align}\label{PiGtwisted2}
 (1_V\tens\cA_{V,V,W}^{-1}) & (1_V\tens(\widetilde{i}_V\tens 1_W))\cA_{V,\vac,W}^{-1}(r_V^{-1}\tens 1_W)\nonumber\\
 & =(1_V\tens\cA_{V,V,W}^{-1})\cA_{V,V\tens V, W}^{-1}((1_V\tens\widetilde{i}_V)\tens 1_W)(\cR_{\vac, V}\tens 1_W)(l_V^{-1}\tens 1_W)\nonumber\\
 & =(1_V\tens\cA_{V,V,W}^{-1})\cA_{V,V\tens V, W}^{-1}(\cR_{V\tens V, V}\tens 1_W)((\widetilde{i}_V\tens 1_V)\tens 1_W)\cA_{\vac,V,W} l_{V\tens W}^{-1}\nonumber\\
 & =\cA_{V,V,V\tens W}^{-1}\cA_{V\tens V,V,W}^{-1}((\cR_{V,V}\tens 1_V)\tens 1_W)(\cA_{V,V,V}\tens 1_W)((1_V\tens\cR_{V,V})\tens 1_W)\circ\nonumber\\
 &\hspace{6em}\circ(\cA_{V,V,V}^{-1}\tens 1_W)\cA_{V\tens V,V,W}(\widetilde{i}_V\tens 1_{V\tens W}) l_{V\tens W}^{-1},
\end{align}
where the last equality uses both the hexagon and pentagon axioms. We also rewrite the last three arrows of \eqref{PiGtwisted} using the associativity and commutativity of $\mu_W$ and $\mu_V$ as well as the pentagon axiom:
\begin{align}\label{PiGtwisted3}
 \mu_W(1_V & \tens\mu_W)  (1_V\tens(1_V\tens\mu_W)) = \mu_W(1_V\tens\mu_W)(1_V\tens(\mu_V\tens 1_W))(1_V\tens\cA_{V,V,W})\nonumber\\
 & =\mu_W(\mu_V\tens 1_W)\cA_{V,V,W}(1_V\tens(\mu_V\tens 1_W))(1_V\tens(\cR_{V,V}\tens 1_W))(1_V\tens\cA_{V,V,W})\nonumber\\
 & =\mu_W(\mu_V\tens 1_W)((1_V\tens\mu_V)\tens 1_W)((1_V\tens\cR_{V,V})\tens 1_W)\cA_{V,V\tens V,W}(1_V\tens\cA_{V,V,W})\nonumber\\
 & =\mu_W(\mu_V\tens 1_W)((\mu_V\tens 1_V)\tens 1_W)(\cA_{V,V,V}\tens 1_W)((1_V\tens\cR_{V,V})\tens 1_W)(\cA_{V,V,V}^{-1}\tens 1_W)\cA_{V\tens V,V,W}\cA_{V,V,V\tens W}.
\end{align}
We insert \eqref{PiGtwisted2} and \eqref{PiGtwisted3} into \eqref{PiGtwisted}, canceling $\cA_{V,V,V\tens W}$ with its inverse:
\begin{align*}
 V & \tens W  \xrightarrow{\cM_{V,W}} V\tens W\xrightarrow{l_{V\tens W}^{-1}} \vac\tens(V\tens W)\xrightarrow{\widetilde{i}_V\tens 1_{V\tens W}} (V\tens V)\tens(V\tens W)\xrightarrow{\cA_{V\tens V, V,W}} ((V\tens V)\tens V)\tens W\nonumber\\
 & \xrightarrow{\cA_{V,V,V}^{-1}\tens 1_W} (V\tens(V\tens V))\tens W\xrightarrow{(1_V\tens\cR_{V,V})\tens 1_W} (V\tens(V\tens V))\tens W\xrightarrow{\cA_{V,V,V}\tens 1_W} ((V\tens V)\tens V)\tens W\nonumber\\
 & \xrightarrow{(\cR_{V,V}\tens 1_V)\tens 1_W} ((V\tens V)\tens V)\tens W\xrightarrow{\cA_{V\tens V,V,W}^{-1}} (V\tens V)\tens(V\tens W)\xrightarrow{(g\tens 1_V)\tens[(g\tens 1_W)\cM_{V,W}]} (V\tens V)\tens(V\tens W)\nonumber\\
 & \xrightarrow{\cA_{V\tens V,V,W}} ((V\tens V)\tens V)\tens W\xrightarrow{\cA_{V,V,V}^{-1}\tens 1_W} (V\tens(V\tens V))\tens W\xrightarrow{(1_V\tens\cR_{V,V})\tens 1_W} (V\tens(V\tens V))\tens W\nonumber\\
 &\xrightarrow{\cA_{V,V,V}\tens 1_W} ((V\tens V)\tens V)\tens W\xrightarrow{(\mu_V\tens 1_V)\tens 1_W} (V\tens V)\tens W\xrightarrow{\mu_V\tens 1_W} V\tens W\xrightarrow{\mu_W} W.
\end{align*}
Next we apply naturality of the left unit isomorphisms to the first two arrows and naturality of the associativity and braiding isomorphisms to $g$. Then we use the automorphism property of $g$ and finally apply naturality of associativity to the second $\cR_{V,V}$:
\begin{align}\label{eqn:calc3}
 V & \tens W  \xrightarrow{l_{V\tens W}^{-1}} \vac\tens(V\tens W)\xrightarrow{\widetilde{i}_V\tens 1_{V\tens W}} (V\tens V)\tens(V\tens W)\xrightarrow{1_{V\tens V}\tens\cM_{V,W}} (V\tens V)\tens(V\tens W)\nonumber\\
 &\xrightarrow{\cA_{V\tens V, V,W}} ((V\tens V)\tens V)\tens W\xrightarrow{\cA_{V,V,V}^{-1}\tens 1_W} (V\tens(V\tens V))\tens W\xrightarrow{(1_V\tens\cR_{V,V})\tens 1_W} (V\tens(V\tens V))\tens W\nonumber\\
 &\xrightarrow{\cA_{V,V,V}\tens 1_W} ((V\tens V)\tens V)\tens W\xrightarrow{\cA_{V\tens V,V,W}^{-1}} (V\tens V)\tens(V\tens W)\xrightarrow{1_{V\tens V}\tens\cM_{V,W}} (V\tens V)\tens(V\tens W)\nonumber\\
 & \xrightarrow{\cA_{V\tens V,V,W}} ((V\tens V)\tens V)\tens W\xrightarrow{(\cR_{V,V}\tens 1_V)\tens 1_W} ((V\tens V)\tens V)\tens W\xrightarrow{\cA_{V,V,V}^{-1}\tens 1_W} (V\tens(V\tens V))\tens W\nonumber\\
 &\xrightarrow{(1_V\tens\cR_{V,V})\tens 1_W} (V\tens(V\tens V))\tens W\xrightarrow{\cA_{V,V,V}\tens 1_W} ((V\tens V)\tens V)\tens W\xrightarrow{(\mu_V\tens 1_V)\tens 1_W} (V\tens V)\tens W\nonumber\\
 &\xrightarrow{(g\tens 1_V)\tens 1_W} (V\tens V)\tens W\xrightarrow{\mu_V\tens 1_W} V\tens W\xrightarrow{\mu_W} W.
\end{align}
Now we use the hexagon axiom and commutativity of $\mu_V$ to simplify the penultimate seven arrows:
\begin{align*}
 \mu_V(g\tens 1_V)(\mu_V\tens 1_V)\cA_{V,V,V}(1_V\tens\cR_{V,V}) & \cA_{V,V,V}^{-1}(\cR_{V,V}\tens 1_V) = \mu_V(g\tens 1_V)(\mu_V\tens 1_V)\cR_{V,V\tens V}\cA_{V,V,V}^{-1}\nonumber\\
 & =\mu_V\cR_{V,V}(1_V\tens g)(1_V\tens\mu_V)\cA_{V,V,V}^{-1} =\mu_V(1_V\tens g)(1_V\tens\mu_V)\cA_{V,V,V}^{-1}.
\end{align*}
We also rewrite associativity isomorphisms in the second and third lines of \eqref{eqn:calc3} using the pentagon axiom:
\begin{align*}
 V & \tens W  \xrightarrow{l_{V\tens W}^{-1}} \vac\tens(V\tens W)\xrightarrow{\widetilde{i}_V\tens 1_{V\tens W}} (V\tens V)\tens(V\tens W)\xrightarrow{1_{V\tens V}\tens\cM_{V,W}} (V\tens V)\tens(V\tens W)\nonumber\\
 &\xrightarrow{\cA_{V,V,V\tens W}^{-1}} V\tens(V\tens(V\tens W)\xrightarrow{1_V\tens\cA_{V,V,W}} V\tens((V\tens V)\tens W)\xrightarrow{\cA_{V,V\tens V,W}} (V\tens(V\tens V))\tens W\nonumber\\
 & \xrightarrow{(1_V\tens\cR_{V,V})\tens 1_W} (V\tens(V\tens V))\tens W\xrightarrow{\cA_{V,V\tens V, W}^{-1}} V\tens((V\tens V)\tens W)\xrightarrow{1_V\tens\cA_{V,V,W}^{-1}} V\tens(V\tens(V\tens W))\nonumber\\
 & \xrightarrow{\cA_{V,V,V\tens W}} (V\tens V)\tens(V\tens W)\xrightarrow{1_{V\tens V}\tens\cM_{V,W}} (V\tens V)\tens(V\tens W)\xrightarrow{\cA_{V\tens V,V,W}} ((V\tens V)\tens V)\tens W\nonumber\\
 & \xrightarrow{\cA_{V,V,V}^{-1}\tens 1_W} (V\tens(V\tens V))\tens W\xrightarrow{(1_V\tens\mu_V)\tens 1_W} (V\tens V)\tens W\xrightarrow{(1_V\tens g)\tens 1_W} (V\tens V)\tens W\xrightarrow{\mu_V\tens 1_W} V\tens W\xrightarrow{\mu_W} W.
\end{align*}
Again using naturality of associativity and the pentagon axiom, we get:
\begin{align}\label{PiGtwisted4}
 V & \tens W  \xrightarrow{l_{V\tens W}^{-1}} \vac\tens(V\tens W)\xrightarrow{\widetilde{i}_V\tens 1_{V\tens W}} (V\tens V)\tens(V\tens W)\xrightarrow{\cA_{V,V,V\tens W}^{-1}} V\tens(V\tens(V\tens W))\nonumber\\
 &\xrightarrow{1_V\tens(1_V\tens\cM_{V,W})} V\tens(V\tens(V\tens W))\xrightarrow{1_V\tens\cA_{V,V,W}} V\tens((V\tens V)\tens W)\xrightarrow{1_V\tens(\cR_{V,V}\tens 1_W)} V\tens((V\tens V)\tens W)\nonumber\\
 &\xrightarrow{1_V\tens\cA_{V,V,W}^{-1}} V\tens(V\tens (V\tens W))\xrightarrow{1_V\tens(1_V\tens\cM_{V,W})} V\tens(V\tens(V\tens W))\xrightarrow{1_V\tens\cA_{V,V,W}} V\tens((V\tens V)\tens W)\nonumber\\
 &\xrightarrow{\cA_{V,V\tens V,W}} (V\tens(V\tens V))\tens W\xrightarrow{(1_V\tens\mu_V)\tens 1_W} (V\tens V)\tens W\xrightarrow{(1_V\tens g)\tens 1_W} (V\tens V)\tens W\xrightarrow{\mu_V\tens 1_W} V\tens W\xrightarrow{\mu_W} W.
\end{align}

We now analyze the isomorphism $V\tens(V\tens W)\rightarrow(V\tens V)\tens W$ in the fourth through ninth arrows. We apply the Yang-Baxter relation once to get
\begin{align*}
 V\tens(V\tens W) & \xrightarrow{1_V\tens\cR_{V,W}} V\tens(W\tens V)\xrightarrow{\cA_{V,W,V}} (V\tens W)\tens V\xrightarrow{\cR_{V,W}\tens 1_V} (W\tens V)\tens V\xrightarrow{\cA_{W,V,V}^{-1}} W\tens(V\tens V)\nonumber\\
 & \xrightarrow{1_W\tens\cR_{V,V}} W\tens(V\tens V)\xrightarrow{\cA_{W,V,V}} (W\tens V)\tens W\xrightarrow{\cR_{W,V}\tens 1_V} (V\tens W)\tens V\xrightarrow{\cA_{V,W,V}^{-1}} V\tens(W\tens V)\nonumber\\
 & \xrightarrow{1_V\tens\cR_{W,V}} V\tens(V\tens W)\xrightarrow{\cA_{V,V,W}} (V\tens V)\tens W
\end{align*}
and a second time to obtain
\begin{align*}
 V\tens(V\tens W) & \xrightarrow{1_V\tens\cR_{V,W}} V\tens(W\tens V)\xrightarrow{\cA_{V,W,V}} (V\tens W)\tens V\xrightarrow{\cR_{V,W}\tens 1_V} (W\tens V)\tens V\xrightarrow{\cR_{W,V}\tens 1_V} (V\tens W)\tens V\nonumber\\
 & \xrightarrow{\cA_{V,W,V}^{-1}} V\tens(W\tens V)\xrightarrow{1_V\tens\cR_{W,V}} V\tens(V\tens W)\xrightarrow{\cA_{V,V,W}} (V\tens V)\tens W\xrightarrow{\cR_{V,V}\tens 1_W} (V\tens V)\tens W.
\end{align*}
By the hexagon axiom, this equals
\begin{align*}
 V\tens(V\tens W) & \xrightarrow{\cA_{V,V,W}} (V\tens V)\tens W\xrightarrow{\cR_{V\tens V, W}} W\tens(V\tens V)\xrightarrow{\cA_{W,V,V}}(W\tens V)\tens V\nonumber\\
 &\xrightarrow{\cA_{W,V,V}^{-1}} W\tens(V\tens V)\xrightarrow{\cR_{W,V\tens V}} (V\tens V)\tens W\xrightarrow{\cR_{V,V}\tens 1_W} (V\tens V)\tens W.
\end{align*}
We cancel associativity isomorphisms and insert this composition into \eqref{PiGtwisted4}:
\begin{align*}
 V & \tens W  \xrightarrow{l_{V\tens W}^{-1}} \vac\tens(V\tens W)\xrightarrow{\widetilde{i}_V\tens 1_{V\tens W}} (V\tens V)\tens(V\tens W)\xrightarrow{\cA_{V,V,V\tens W}^{-1}} V\tens(V\tens(V\tens W))\nonumber\\
 & \xrightarrow{1_V\tens\cA_{V,V,W}} V\tens((V\tens V)\tens W)\xrightarrow{1_V\tens\cM_{V\tens V,W}} V\tens((V\tens V)\tens W)\xrightarrow{1_V\tens(\cR_{V,V}\tens 1_W)} V\tens((V\tens V)\tens W\nonumber\\
 &\xrightarrow{\cA_{V,V\tens V,W}} (V\tens(V\tens V))\tens W\xrightarrow{(1_V\tens\mu_V)\tens 1_W} (V\tens V)\tens W\xrightarrow{(1_V\tens g)\tens 1_W} (V\tens V)\tens W\xrightarrow{\mu_V\tens 1_W} V\tens W\xrightarrow{\mu_W} W.
\end{align*}
By the naturality of associativity and monodromy, the commutativity of $\mu_V$, and properties of the left unit isomorphism, we now get
\begin{align*}
  V & \tens W \xrightarrow{l_V^{-1}\tens 1_W} (\vac\tens V)\tens W\xrightarrow{(\widetilde{i}_V\tens 1_V)\tens 1_W} ((V\tens V)\tens V)\tens W\xrightarrow{\cA_{V\tens V,V,W}^{-1}} (V\tens V)\tens(V\tens W)\nonumber\\
 & \xrightarrow{\cA_{V,V,V\tens W}^{-1}} V\tens(V\tens(V\tens W)) \xrightarrow{1_V\tens\cA_{V,V,W}} V\tens((V\tens V)\tens W)\xrightarrow{1_V\tens(\mu_V\tens 1_W)} V\tens(V\tens W)\nonumber\\
 &\xrightarrow{1_V\tens\cM_{V,W}} V\tens(V\tens W)\xrightarrow{\cA_{V,V,W}} (V\tens V)\tens W\xrightarrow{(1_V\tens g)\tens 1_W} (V\tens V)\tens W\xrightarrow{\mu_V\tens 1_W} V\tens W\xrightarrow{\mu_W} W.
\end{align*}
By the pentagon axiom, the third through fifth arrows equal $\cA_{V,V\tens V, W}^{-1}(\cA_{V,V,V}^{-1}\tens 1_W)$. We also use naturality of the associativity isomorphisms and associativity of $\mu_V$ to end up with:
\begin{align*}
  V & \tens W \xrightarrow{l_V^{-1}\tens 1_W} (\vac\tens V)\tens W\xrightarrow{(\widetilde{i}_V\tens 1_V)\tens 1_W} ((V\tens V)\tens V)\tens W\xrightarrow{\cA_{V,V,V}^{-1}\tens 1_W} (V\tens(V\tens V))\tens W\nonumber\\
 & \xrightarrow{(1_V\tens\mu_V)\tens 1_W} (V\tens V)\tens W\xrightarrow{\cA_{V,V,W}^{-1}} V\tens(V\tens W)\xrightarrow{1_{V}\tens[(g\tens 1_W)\cM_{V,W}]} V\tens(V\tens W)\xrightarrow{1_V\tens\mu_W} V\tens W\xrightarrow{\mu_W} W.
\end{align*}

At this point, we use Lemma \ref{rigidlike_lemma} to replace the morphism $V\rightarrow V\tens V$ in the first four arrows with
\begin{equation*}
 (\mu_V\tens 1_V)\cA_{V,V,V}(1_V\tens\widetilde{i}_V)r_V^{-1}.
\end{equation*}
Inserting this into the above composition and using the triangle axiom, we get
\begin{align*}
 V & \tens W\xrightarrow{1_V\tens l_W^{-1}} V\tens(\vac\tens W)\xrightarrow{\cA_{V,\vac, W}} (V\tens\vac)\tens W\xrightarrow{(1_V\tens\widetilde{i}_V)\tens 1_W} (V\tens(V\tens V))\tens W\nonumber\\
 &\xrightarrow{\cA_{V,V,V}\tens 1_W} ((V\tens V)\tens V)\tens W\xrightarrow{(\mu_V\tens 1_V)\tens 1_W} (V\tens V)\tens W\xrightarrow{\cA_{V,V,W}^{-1}} V\tens(V\tens W)\nonumber\\
 &\xrightarrow{1_{V}\tens[(g\tens 1_W)\cM_{V,W}]} V\tens(V\tens W)\xrightarrow{1_V\tens\mu_W} V\tens W\xrightarrow{\mu_W} W.
\end{align*}
Using naturality of the associativity and the pentagon, we obtain
\begin{align*}
 V  \tens W\xrightarrow{1_V\tens l_W^{-1}} & V\tens(\vac\tens W)\xrightarrow{1_V\tens(\widetilde{i}_V\tens 1_W)} V\tens((V\tens V)\tens W)\xrightarrow{1_V\tens\cA_{V,V,W}^{-1}} V\tens(V\tens(V\tens W))\nonumber\\
 & \xrightarrow{\cA_{V,V,V\tens W}} (V\tens V)\tens(V\tens W)\xrightarrow{1_{V\tens V}\tens[(g\tens 1_W)\cM_{V,W}]} (V\tens V)\tens(V\tens W)\nonumber\\
 &\xrightarrow{1_{V\tens V}\tens\mu_W} (V\tens V)\tens W\xrightarrow{\mu_V\tens 1_W} V\tens W\xrightarrow{\mu_W} W.
\end{align*}
Finally, by naturality of the associativity isomorphisms and the associativity of $\mu_V$ and $\mu_W$, this equals
\begin{align*}
 V  \tens W\xrightarrow{1_V\tens l_W^{-1}} & V\tens(\vac\tens W)\xrightarrow{1_V\tens(\widetilde{i}_V\tens 1_W)} V\tens((V\tens V)\tens W)\xrightarrow{1_V\tens\cA_{V,V,W}^{-1}} V\tens(V\tens(V\tens W))\nonumber\\
 & \xrightarrow{1_{V}\tens(1_V\tens[(g\tens 1_W)\cM_{V,W}])} V\tens(V\tens(V\tens W))\xrightarrow{1_V\tens(1_V\tens\mu_W)} V\tens(V\tens W)\xrightarrow{1_V\tens \mu_W} V\tens W\xrightarrow{\mu_W} W,
\end{align*}
which is $\mu_W(1_V\tens\Pi_g)$.

\bigskip

\noindent\textbf{Equations \eqref{eqn:PiG_PiH} and \eqref{eqn:sum_pi_g_identity}.} When $W$ is $h$-twisted, $\Pi_g$ is the composition
\begin{equation*}
 W\xrightarrow{l_W^{-1}} \vac\tens W\xrightarrow{\widetilde{i}_V\tens 1_W} (V\tens V)\tens W\xrightarrow{\cA_{V,V,W}^{-1}} V\tens(V\tens W)\xrightarrow{1_V\tens(h^{-1}g\tens 1_W)} V\tens(V\tens W)\xrightarrow{1_V\tens\mu_W} V\tens W\xrightarrow{\mu_W} W.
\end{equation*}
Naturality of the associativity isomorphisms and associativity of $\mu_W$ imply this is
\begin{equation*}
 W\xrightarrow{l_W^{-1}} \vac\tens W\xrightarrow{\widetilde{i}_V\tens 1_W} (V\tens V)\tens W\xrightarrow{(1_V\tens h^{-1}g)\tens 1_W} (V\tens V)\tens W\xrightarrow{\mu_V\tens 1_W} V\tens W\xrightarrow{\mu_W} W,
\end{equation*}
which is the right side of \eqref{eqn:PiG_PiH}. Finally, $\sum_{g\in G}\pi_g$ is the composition
 \begin{align*}
 W\xrightarrow{l_W^{-1}}\vac\tens W\xrightarrow{\widetilde{i}_V\tens 1_W} (V\tens V)\tens W\xrightarrow{\cA_{V,V,W}^{-1}} V\tens( & V\tens W)  \xrightarrow{1_V\tens\cM_{V,W}} V\tens(V\tens W)\nonumber\\
  &\xrightarrow{1_V\tens(\frac{1}{\vert G\vert}\sum_{g\in G} g\tens 1_V)} V\tens(V\tens W)\xrightarrow{1_V\tens\mu_W} V\tens W\xrightarrow{\mu_W} W.
 \end{align*}
 Using $\frac{1}{\vert G\vert}\sum_{g\in G} g=\iota_V\varepsilon_V$ and the unit property of $W$, we get
 \begin{align*}
  W\xrightarrow{l_W^{-1}}\vac\tens W\xrightarrow{\widetilde{i}_V\tens 1_W} (V\tens V)\tens W\xrightarrow{\cA_{V,V,W}^{-1}} V\tens( & V\tens W)  \xrightarrow{1_V\tens\cM_{V,W}} V\tens(V\tens W)\nonumber\\
  &\xrightarrow{1_V\tens(\varepsilon_V\tens 1_W)} V\tens(\vac\tens W)\xrightarrow{1_V\tens l_W} V\tens W\xrightarrow{\mu_W} W.
 \end{align*}
 Then we simplify using naturality of the monodromy and associativity isomorphisms together with $\cM_{\vac,W}=1_{\vac,W}$:
\begin{equation*}
 W\xrightarrow{l_W^{-1}}\vac\tens W\xrightarrow{\widetilde{i}_V\tens 1_W} (V\tens V)\tens W\xrightarrow{(1_V\tens\varepsilon_V)\tens 1_W} )(V\tens\vac)\tens W\xrightarrow{\cA_{V,\vac,W}^{-1}} V\tens(\vac\tens W)\xrightarrow{1_V\tens l_W} V\tens W\xrightarrow{\mu_W} W.
\end{equation*}
Now $(1_V\tens l_W)\cA_{V,\vac,W}^{-1} = r_V\tens 1_W$ by the triangle axiom and we get
\begin{equation*}
 W\xrightarrow{l_W^{-1}}\vac\tens W\xrightarrow{[r_V(1_V\tens\varepsilon_V)\widetilde{i}_V]\tens 1_W} V\tens W\xrightarrow{\mu_W} W.
\end{equation*}
This is the right side of \eqref{eqn:sum_pi_g_identity} by the unit property of $W$.

\end{document}